%% file: main.tex
\documentclass[11pt]{article}
\usepackage{mathtools}
\usepackage[a4paper,top=3cm,bottom=3.5cm,left=2.5cm,right=2.5cm,marginparwidth=1.75cm]{geometry}
\DeclarePairedDelimiter\abs{\lvert}{\rvert}
\usepackage[hidelinks]{hyperref}
\usepackage{amsmath, amssymb, amsthm, bbm, graphicx, cleveref, caption, tikz, enumerate, subcaption, todonotes,comment}
\usetikzlibrary{decorations.pathreplacing,arrows.meta,calc,patterns,angles,quotes,decorations.markings}
\definecolor{mycolour}{rgb}{0,0.6,0.6}

\DeclareMathOperator{\Id}{Id}
\DeclareMathOperator{\particular}{{part}}
\DeclareMathOperator{\homogeneous}{{hom}}

\DeclareMathOperator{\Span}{{span}}
\DeclareMathOperator{\diag}{{diag}}

\newtheorem{theorem}{Theorem}
\newtheorem{lemma}[theorem]{Lemma}

\newtheorem{cor}[theorem]{Corollary}
\newtheorem{definition}[theorem]{Definition}
\newtheorem{remark}[theorem]{Remark}

\numberwithin{equation}{section}
\numberwithin{theorem}{section}

\renewcommand*{\Re}{\operatorname{Re}}
\renewcommand*{\Im}{\operatorname{Im}}

\newcommand{\R}{\mathbb{R}}
\renewcommand{\i}{\mathrm{i}\mkern1mu}

\graphicspath{{images/}}
\newcommand{\N}{\mathbb{N}}

\newcommand{\Z}{\mathbb{Z}}
\newcommand{\p}{\partial}
\renewcommand{\epsilon}{\varepsilon}

\newcommand{\nm}{\noalign{\smallskip}}
\newcommand{\ds}{\displaystyle}

\title{Asymptotic Floquet theory for first order ODEs with finite Fourier series perturbation  and its applications to Floquet metamaterials\thanks{\footnotesize
This work was supported in part by the Swiss National Science Foundation grant number
200021--200307.}}

\author{Habib Ammari\thanks{\footnotesize Department of Mathematics, ETH Z\"urich, R\"amistrasse 101, CH-8092 Z\"urich, Switzerland (habib.ammari@math.ethz.ch, thea.kosche@sam.math.ethz.ch).} \and Erik Orvehed Hiltunen \thanks{\footnotesize Department of Mathematics, Yale University, 51 Prospect Street, New Haven CT 06511, USA  (erik.hiltunen@yale.edu).}  \and Thea Kosche\footnotemark[2]}

\date{}

\begin{document}
	\maketitle
	
\begin{abstract}
Our aim in this paper is twofold. Firstly, we develop a new asymptotic theory for Floquet exponents. We consider a linear system of differential equations with a time-periodic coefficient matrix. Assuming that the coefficient matrix depends analytically on a small parameter, we derive a full asymptotic expansion of its Floquet exponents. Based on this, we prove that only the constant order Floquet exponents of multiplicity higher than one will be perturbed linearly. The required multiplicity can be achieved via folding of the system through certain choices of the periodicity of the coefficient matrix. Secondly, we apply such an asymptotic  theory for  the analysis of Floquet metamaterials. We provide a characterization of asymptotic exceptional points for a pair of subwavelength resonators with time-dependent material parameters. We prove that asymptotic exceptional points are obtained if the frequency components of the perturbations fulfill a certain ratio, which is determined by the geometry of the dimer of subwavelength resonators.
 \end{abstract}

\noindent{\textbf{Mathematics Subject Classification (MSC2000):} 35J05, 35C20, 35P20, 74J20
		
\vspace{0.2cm}
		
	\vspace{0.5cm}

\tableofcontents

\input{chapters/Introduction}

\input{chapters/Floquet_theory}

\input{chapters/Asymptotic_Floquet_theory}

\input{chapters/Asymptotic_analysis_of_Floquet_exponents}

\input{chapters/Application.tex}

\input{chapters/Concluding_Remarks}

\bibliographystyle{abbrv}
\bibliography{references}
\end{document}

%% file: chapters/Introduction.tex
\section{Introduction}
In the past two decades, metamaterials have revolutionized our approaches to the control and manipulation of wave-matter interactions at deep subwavelength scales 
\cite{ammari2021functional,kadic20193d,lemoult2016soda,yves2017crytalline,phononic1,phononic2}. 
Metamaterials are micro-structured materials with subwavelength resonators as building blocks that exhibit a variety of exotic and useful properties.  
Subwavelength resonators are small objects that possess  \emph{subwavelength resonances} and strongly scatter waves with comparatively large wavelengths.
Due to the subwavelength nature of the resonances, systems of subwavelength resonators enable wave control on very small length scales compared to the operating wavelengths. Most notably, subwavelength wave manipulations  such as subwavelength guiding, trapping, and focusing of waves, super-resolution imaging, and cloaking can be achieved by means of metamaterials \cite{review,review2,ammari2018subwavelength,ammari2015superresolution,lemoult2016soda,yves2017crytalline,milton2006cloaking,pendry2000negative,smith2004metamaterials}.
A fundamental property of subwavelength resonators is that their material parameters differ greatly from the background medium. In order to investigate their subwavelength resonances, a capacitance matrix formalism describing the properties of  systems of subwavelength resonators has been introduced and thoroughly discussed in  \cite{ammari2020biomimetic,ammari2019cochlea,ammari2021functional,ammari2019double}. The capacitance matrix formulation is a {\em discrete approximation} of the continuous model. It yields a mathematical foundation of metamaterials \cite{ammari2020exceptional,ammari2020high,ammari2019double,ammari2017effective,ammari2019bloch}.

Recently, the field of metamaterials has experienced tremendous advances by exploring the novel and promising area of time-modulations. Research on {\em Floquet metamaterials} (or time-modulated metamaterials) aims to explore new phenomena arising from the temporal modulation of the material parameters of the subwavelength resonators.  It has enabled to open new paradigms for the manipulation of wave-matter interactions in both spatial and temporal domains  \cite{ammari2020time,fleury2016floquet,rechtsman2013photonic,
koutserimpas2020electromagnetic,koutserimpas2018zero,koutserimpas2018zero,
wilson2019temporally,wilson2018temporal}. Moreover, recent advances in metamaterials have demonstrated the existence of {\em exceptional points} in systems of subwavelength resonators \cite{ammari2020high,ammari2020exceptional,heiss2012physics}. Exceptional points are parameter values at which the system's eigenvalues and their associated eigenvectors simultaneously coincide \cite{heiss2012physics}. Such points have a variety of applications, most notably to enhanced sensing \cite{ammari2020exceptional,ammari2020high,hodaei2017enhanced}.

Due to this variety of applications of exceptional points, the main motivation for this work is to provide a tool to detect exceptional points, in particular, to detect them for systems of time-modulated subwavelength resonators. To this end, an asymptotic Floquet theory will be developed. 
Indeed, Floquet's theory states that the fundamental solution of an ordinary differential equation (ODE)
$$\frac{dx}{dt} = A(t)x,$$
with $T$-periodic continuous coefficient matrix $A(t)$, can be decomposed into a $T$-periodic part $P(t)$ and a part $\exp(Ft)$, where $F$ is a complex matrix and is called \emph{Floquet's exponent matrix}. That is, the fundamental solution $X$ can be written as
$$ X(t) = P(t)\exp(Ft).$$
This allows to prove that a parameterized system ${dx}/{dt} = A(c,t)x$ is posed at an exceptional point if and only if its associated Floquet's exponent matrix is non-diagonalizable. In order to analyze further the Floquet's exponent matrix, we will consider a parameterized system ${dx}/{dt} = A_\epsilon(c,t)x$, where $c$ is a parameter and $A_\epsilon$ is an analytic function of $\epsilon$ in a neighborhood of $0$ and has the asymptotic expansion
$$A_\epsilon(c,t) =  A_0(c) + \epsilon A_1(c,t) + \epsilon^2A_2(c,t) + \ldots$$
as $\epsilon \rightarrow 0$ with $A_0(c)$ being diagonal and constant in time and $A_n(c,t)$ having a finite Fourier series with respect to time for all $n \geq 1$. We are interested in detecting those parameters $c$, where ${dx}/{dt} = A_\epsilon(c,t)x$ is posed at an exceptional point for small $\epsilon>0$. To this end an asymptotic analysis of the Floquet's exponent matrix or at least of its eigenvalues, the so-called \emph{Floquet exponents}, is needed.
Usual approximations of Floquet exponents are obtained using the techniques of Magnus expansion or Dyson series, or by simply integrating the system across one period; see, for instance, \cite{magnus1,magnus2,dyson,magnus3}. Unfortunately, these methods do not allow for detecting possible exceptional points in a parametrized system of ODEs. Nevertheless, the theory developed in this paper allows for investigating the system for existence of exceptional points and for detecting those parameter values $c_0$ at which the system is posed at an exceptional point. Indeed, this theory allows to classify asymptotic first order exceptional points in the setting of the classical harmonic oscillator and in the setting of a dimer of time-modulated subwavelength resonators. Furthermore, it gives some insights into the dependence of Floquet exponents on the parameters of the corresponding system. 

In Section \ref{Chapter:Floquet_theory}, Floquet's theory and its main results is  introduced. A general scheme for obtaining an asymptotic expansion of Floquet's exponent matrices for general systems ${dx}/{dt} = A_\epsilon(t)x$ is presented. This scheme is then applied in Section \ref{Chapter:Asymptotic_Floquet_theory} to the system presented above. It provides formulas for Floquet's exponent matrix. Those are then used for the case of the classical harmonic oscillator presented in Section \ref{Harmonic_oscillator}. In Section \ref{asymptotic_analysis_of_floquet_exponents}, asymptotic analysis of Floquet exponents is established using the asymptotic formulas for the Floquet's exponent matrix from Section \ref{Chapter:Asymptotic_Floquet_theory}. This asymptotic analysis allows for a full classification of the first order asymptotic exceptional points in the classical harmonic oscillator setting; see Section \ref{sec:exceptional_points}. In Section \ref{ch:metamaterials}, the developed theory is then applied to the original motivation: time-modulated subwavelength resonators. It is used to analyze the Floquet's exponent matrices and Floquet exponents of a dimer of time-modulated resonators and leads to a classification of its associated first order asymptotic exceptional points. Other applications of the theory developed in this paper include the analysis of unidirectional guiding of waves and the valley Hall effect in periodic time-modulated subwavelength structures \cite{jinghao}. The paper ends with some concluding remarks in Section \ref{concluding_remarks}. 

%% file: chapters/Floquet_theory.tex
\section{Floquet's and Lyapunov's reduction theorem} \label{Chapter:Floquet_theory}

In this section, we first recall the Floquet theorem and the Lyapunov reduction theorem. We refer the reader to \cite{Teschl} for their proofs. 

\subsection{Statements of the theorems}
Let $A: \mathbb{R} \rightarrow \mathbb{C}^{N \times N}$ be a $T$-periodic matrix valued function and $x_0 \in \mathbb{C}^N$ a $N$-dimensional vector. Throughout this paper we assume that $A$ is of class $C^0$. 
Considering the linear system of ordinary differential equations (ODEs) of the form
 \begin{equation} \label{basic_equation1}
 \left\{
    \begin{aligned}
        \frac{dx}{dt} &= A(t)x,\\
        x(0) &= x_0,
    \end{aligned}
    \right.
\end{equation}
we are interested in the structure of the fundamental solution $X$, which is uniquely characterized by
\begin{equation}   \label{fundamental_solution_equation}
   \left\{ \begin{aligned}
        \frac{dX}{dt} &= A(t)X, \\
        X(0) &= \Id_{N\times N},
    \end{aligned}
    \right.
    \end{equation}
since then the solution to \eqref{basic_equation1} will be given by
    \begin{align*}
        x(t) = X(t)x_0.
    \end{align*}

Floquet's theory elucidates the nature and structure of the fundamental solution of such a periodic system and it states the following theorem.

\begin{theorem}[Floquet's theorem]\label{Floquet-Bloch_theorem}
    The fundamental solution of the  $T$-periodic linear system of ODEs \eqref{basic_equation1} is given by
        \begin{align}
            X(t) = P(t)\exp(Ft),
        \end{align}
    for some $T$-periodic matrix-valued function $P: \mathbb{R} \rightarrow \mathbb{C}^{N \times N}$ and for a constant matrix $F \in \mathbb{C}^{N \times N}$.
\end{theorem}

The structure $X(t) = P(t)\exp(Ft)$ of the fundamental solution can be interpreted as separating the \emph{local and $T$-periodic part} of the fundamental solution $P(t)$ from its \emph{global, non-periodic behavior} $\exp(Ft)$. Here, some care needs to be taken, since $P(t)$ and $F$ are not unique and any matrix $R$ which satisfies $\exp(RT) = \Id_{N \times N}$ and $[F,R]=0$, i.e., $FR = RF$, gives rise to another decomposition of the fundamental solution $X$. In fact,
    \begin{align*}
        X(t) = (P(t)\exp(Rt))\exp((F-R)t)
    \end{align*}
is also a decomposition of the fundamental solution $X$, where $t \mapsto P(t)\exp(Rt)$ is the $T$-periodic part and $F-R \in \mathbb{C}^{N \times N}$ the associated constant matrix. Depending on the context of the system of equations \eqref{basic_equation1}, different choices of $F$ and $P(t)$ are more suitable than others.

Another way to understand Theorem \ref{Floquet-Bloch_theorem} is via a reformulation given by the following theorem due to Lyapunov. There, the $T$-periodic matrix-valued function $P$ is interpreted as a transformation of \eqref{basic_equation1} which reduces it to a constant system with constant coefficient matrix $F$.

\begin{theorem}[Lyapunov's reduction theorem]
    Given the system of linear $T$-periodic ODEs \eqref{basic_equation1}, there exists an invertible matrix-valued function $P: \mathbb{R} \rightarrow \mathbb{C}^{N \times N}$ and a constant matrix $F \in \mathbb{C}^{N \times N}$ such that the substitution
    \begin{align}
        x(t) = P(t)y(t),
    \end{align}
    reduces the time-dependent system of ODEs \eqref{basic_equation1} to the constant system of ODEs
    \begin{equation}
    \left\{
    \begin{aligned}
        \frac{dy}{dt} = Fy,\\
        y(0) = x_0.
    \end{aligned}
    \right.
    \end{equation}
\end{theorem}

In the following, we seek to develop methods for obtaining such a decomposition and in particular a Floquet exponent matrix $F$ corresponding to some decomposition of the fundamental solution $X(t) = P(t)\exp(Ft)$. The eigenvalues of such a Floquet exponent matrix $F$ will give insights on possible global behaviors of solutions 
of \eqref{basic_equation1}. They are the so-called \emph{Floquet exponents} of
 the system of  ODEs \eqref{basic_equation1} and they are uniquely defined modulo $2 \pi i/T$.

The usual and simplest procedure to obtain the Floquet exponent matrix or the Floquet exponents is to numerically solve the initial value problem \eqref{fundamental_solution_equation}  for the fundamental solution $X$ and then to evaluate $X$ after one period $X(T) = P(T)\exp(FT)$, which is equal to $\exp(FT)$, due to the $T$-periodicity of $P$. The eigenvalues of $\exp(FT)$ will then allow to deduce the Floquet exponents. However, this approach is only applicable to a concrete system with fixed parameters and inapplicable when one is interested in a more general setting or a class of parametrized ODEs. In this paper, we are interested in the dependence of a Floquet decomposition of the fundamental solution of the system ${dx}/{dt} = A_{\epsilon}(t)x$, where $A_\epsilon$ depends analytically on $\epsilon$ in a neighborhood of $0$. To this end, a differential equation for $P(t)$ and $F$ in terms of $A_{\epsilon}(t)$ will be useful.

\subsection{Differential equation for Floquet exponent matrix and Lyapunov transformation}\label{Diff_equation_for_Floquet_exp_mat_and_Lyapunov_trans}

Given \eqref{fundamental_solution_equation} 
for some $T$-periodic matrix-valued function $A: \mathbb{R} \rightarrow \mathbb{C}^{N \times N}$, Floquet's Theorem \ref{Floquet-Bloch_theorem} states that there exists a decomposition $X(t) = P(t)\exp(Ft)$ with some $T$-periodic matrix-valued function $P: \mathbb{R} \rightarrow \mathbb{C}^{N \times N}$ and a constant matrix $F \in \mathbb{C}^{N \times N}$. Fixing such a function $P$ and the corresponding matrix $F$, one can derive the following equation in $P$ and $F$:
    \begin{align*}
        \frac{dP}{dt}(t) \exp(Ft) + P(t)F\exp(Ft) = A(t)P(t)\exp(Ft),
    \end{align*}
which leads to the equation
    \begin{align*}
        \frac{dP}{dt}(t) = A(t)P(t) - P(t)F.
    \end{align*}
In order to find a decomposition of the fundamental solution $X(t) = P(t)\exp(Ft)$ of the system of ODEs \eqref{basic_equation1}, it thus suffices to solve the following system of equations:
\begin{equation} \label{Floquet_decomposition_system_equation_1}
\left\{
    \begin{aligned}
        \frac{dP}{dt}(t) &= A(t)P(t) - P(t)F,\\
        &P: \mathbb{R} \rightarrow \mathbb{C}^{N \times N} \text{ a $T$-periodic function},\\
        &P(0) = \Id_{N \times N},\\
        &F \in \mathbb{C}^{N \times N}.
    \end{aligned}
    \right.
    \end{equation}
%This system is applicable for all $T$-periodic first order linear ODEs and their Floquet decompositions. For such equations, finding a solution is equivalent to providing a Floquet decomposition of their respective fundamental solutions.

 Unfortunately, the system \eqref{Floquet_decomposition_system_equation_1} is too general to allow for an explicit procedure for finding a Floquet decomposition. This is also due to the lack of uniqueness of a solution. It would be of use to us to have an explicit procedure for \emph{finding} a decomposition and also for \emph{choosing} a suitable decomposition.

To this end, suppose that $A_\epsilon(t) = A_0(t) + \epsilon A_1(t) + \epsilon^2 A_2(t) + \ldots$ is an analytic function in $\epsilon$ for $\epsilon$ in a neighborhood of $0$, where $A_n(t)$ are $T$-periodic matrix-valued functions. We consider the following system parametrized by $\epsilon$:
\begin{equation}
\left\{
    \begin{aligned}
        \frac{dX_\epsilon}{dt} &= A_\epsilon(t)X_\epsilon, \\
            X_\epsilon(0) &= \Id_{N\times N}.
    \end{aligned}
    \right.
    \end{equation}
Under suitable assumptions (made precise in \Cref{Inductive_Identity_for_Lyapunov-Floquet_decomposition}), $X_\epsilon$ has an associated Floquet decomposition $X_\epsilon(t) = P_\epsilon(t)\exp(F_\epsilon t)$, which also depends analytically on $\epsilon$. Writing $P_\epsilon(t) = P_0(t) + \epsilon P_1(t) + \epsilon^2 P_2(t) + \ldots$ and $F_\epsilon = F_0 + \epsilon F_1 + \epsilon F_2 + \ldots$ for the respective analytic expansions, one deduces the following equations from \eqref{Floquet_decomposition_system_equation_1}:
    \begin{align}\label{Floquet_decomposition_equation_general_order}
        \frac{dP_n}{dt}(t) &= \sum_{i = 0}^n(A_{i}(t)P_{n-i}(t) - P_{n-i}(t)F_{i}) \text{ for } n \geq 0
    \end{align}
with the corresponding conditions 
\begin{equation} \label{Floquet_decomposition_equation_initial_condition_P_0}
    \left\{\begin{aligned}
        &P_0(0) = \Id_{N \times N},\\ 
        &P_n(0) = 0 \text{ for all } n \geq 1,\\
        &P_n: \mathbb{R} \rightarrow \mathbb{C}^{N \times N} \text{ is a $T$-periodic function for all } n \geq 0 \text{ and }\\
        &F_n \in \mathbb{C}^{N \times N} \text{ for all } n \geq 0.
    \end{aligned}
    \right.
    \end{equation}
It is worth noticing that the equation for $P_n(t)$ in the system of equations \eqref{Floquet_decomposition_equation_general_order} only depends on lower orders, that is, on $P_{n-1}(t), \ldots, P_0(t)$, which allows for an inductive solution procedure. Furthermore, the homogeneous part of every equation \eqref{Floquet_decomposition_equation_general_order}
is given by the \emph{same} linear differential equation, namely
    \begin{align}\label{General_homogeneous_equation_for_P_n}
        \frac{dY}{dt}(t) = A_{0}(t)Y(t) - Y(t)F_{0}.
    \end{align}
Furthermore, the inhomogeneous part of the $n$th equation is given by
    \begin{align}\label{General_inhomogeneous_part_of_equation_for_P_n}
        A_{n}(t)P_{0}(t) - P_{0}(t)F_{n} + \sum_{i = 1}^{n-1}(A_{i}(t)P_{n-i}(t) - P_{n-i}(t)F_{i}),
    \end{align}
which only depends on the orders of $P_{\epsilon}(t)$ and $F_\epsilon$ that are strictly smaller than $n$ and on $F_n$. Consequently, it suffices to determine all solutions $P^{\homogeneous}$ of the homogeneous equation \eqref{General_homogeneous_equation_for_P_n} and to inductively solve for a particular solution $P_{n,F_n}^{\particular}$ of the inhomogeneous part \eqref{General_inhomogeneous_part_of_equation_for_P_n}, which will be a particular solution in terms of $F_n$. In order to determine all possible $F_n$, the  conditions \eqref{Floquet_decomposition_equation_initial_condition_P_0} need to be fulfilled by
$$P^{\homogeneous} + P^{\particular}_{n,F_n},$$
for some solution $P^{\homogeneous}$ of the homogeneous ODE \eqref{General_homogeneous_equation_for_P_n} and some choice of $F_n$. Such a choice of $F_n$ will then give the $n$th order term of $F_\epsilon$ and the corresponding $n$th order term of $P_\epsilon(t)$; $P_n(t) = P^{\homogeneous}(t) + P^{\particular}_{n,F_n}(t)$. 

The following section elaborates such an inductive procedure more concretely in the case where the leading order term $A_0$ is constant and all higher order terms $A_n(t)$ are matrix-valued functions with finite Fourier series. It will turn out that the choice of $F_0$ uniquely determines the solution to the whole system. That is, the choice of $F_0$ uniquely determines the Lyapunov transformation $P_\epsilon(t)$ and all higher order terms of $F_\epsilon$.

%% file: chapters/Asymptotic_Floquet_theory.tex
\section{Asymptotic Floquet theory} \label{Chapter:Asymptotic_Floquet_theory}

This section will present an asymptotic approach for the resolution of \eqref{Floquet_decomposition_system_equation_1} by making use of \eqref{Floquet_decomposition_equation_general_order}--\eqref{Floquet_decomposition_equation_initial_condition_P_0} and by assuming that the leading order of $A_\epsilon$ is constant and diagonal, that is, $A_0(t) \equiv A_0$ is a constant and diagonal matrix. Under suitable analyticity assumptions on $A_\epsilon$ and on its Taylor coefficients with respect to $\epsilon$, an inductive procedure for the resolution of \eqref{Floquet_decomposition_equation_general_order}--\eqref{Floquet_decomposition_equation_initial_condition_P_0} is presented. 
At the end of this section, exact formulas for $F_0$, $F_1$ and some particular coefficients of $F_2$ are derived. Those allow for asymptotic formulas for the Floquet exponents in terms of $A_0$, $A_1$ and $A_2$.

In this section, the classical harmonic oscillator with periodically modulated damping and restoring force is considered. It  serves as an illustrative example and  allows to demonstrate how the different formulas can be applied. It is introduced in Section \ref{Harmonic_oscillator} and then subsequently used as a simple application example.

\subsection{Inductive solution scheme for Lyapunov's transformation and Floquet's exponent matrix}
The following theorem holds. 
\begin{theorem}[Inductive identity for Lyapunov-Floquet decomposition]\label{Inductive_Identity_for_Lyapunov-Floquet_decomposition}
    Let \begin{equation} \label{expA}
    A_\epsilon(t) = A_0 + \epsilon A_1(t) + \epsilon^2A_2(t) + \ldots \in \mathbb{C}^{N \times N} \end{equation} be a $T$-periodic continuous matrix-valued function, which depends analytically on $\epsilon$ in a neighborhood of $0$. Furthermore, suppose  that
        \begin{enumerate}
            \item[(i)] The constant order Taylor coefficient $A_0$ is diagonal and constant;
            \item[(ii)] Any other Taylor coefficient $A_n(t)$ has a finite Fourier series;
            \item[(iii)] The series in (\ref{expA}) is convergent for $|\epsilon| < r_0$, where $r_0$ is independent of $t$.
        \end{enumerate}
    Then there exists an analytically dependent Floquet-Lyapunov decomposition of 
                ${dx}/{dt} = A_\epsilon(t) x$, which is given by
        \begin{align*}
            F_\epsilon &= F_0 + \epsilon F_1 + \epsilon^2F_2 + \ldots,\\
            P_\epsilon(t) &= P_0(t) + \epsilon P_1(t) + \epsilon^2 P_2(t) + \ldots,
        \end{align*}
    where the leading order terms are uniquely defined by
        \begin{align*}
            F_0 = A_0 \mod \diag\left(\frac{2\pi \i}{T}\mathbb{Z}, \ldots, \frac{2\pi \i}{T}\mathbb{Z}\right), \\
            \Im(F_0) \in \diag \left( (-\frac{\pi}{T},\frac{\pi}{T}], \ldots, (-\frac{\pi}{T}, \frac{\pi}{T}] \right),
        \end{align*}
        and 
        \begin{align*}
            P_0(t) = \exp((A_0 -F_0)t).
        \end{align*}
    The higher order terms are uniquely determined by the above (and in fact any feasible) choice of $F_0$. They are inductively given by the Fourier coefficients of 
        \begin{align*}
            \sum_{i = 1}^{n-1}(A_i(t)P_{n-i}(t) - P_{n-i}(t)F_i) + A_n(t)P_0(t) = \sum_{m \in \mathbb{Z}}\exp(\frac{2 \pi \i}{T}mt)\Phi^m
        \end{align*}
    via the formulas
        \begin{align}\label{Inductive_formula_F_n}
            (F_n)_{kl} = \begin{cases}
                \ds (\Phi^{\frac{T}{2\pi \i}(A_0-F_0)_{kk}})_{kl} &\text{if } (F_0)_{kk} = (F_0)_{ll},\\
                \nm \ds
                ((F_0)_{ll} - (F_0)_{kk})\sum_{m \in \mathbb{Z}}\frac{(\Phi^m)_{kl}}{\frac{2 \pi \i}{T}m - (A_0)_{kk} + (F_0)_{ll}} &\text{if } (F_0)_{kk} \not= (F_0)_{ll},
            \end{cases}
        \end{align}
        and 
        \begin{align}\label{Inductive_formula_P_n}
            (P_n(t))_{kl} = \ds \sum_{m \in \mathbb{Z}} \exp(\frac{2\pi \i}{T}mt)\begin{cases} \frac{(\Phi^m)_{kl}}{\frac{2\pi \i}{T}m - (A_0)_{kk} + (F_0)_{ll}} & \text{ if } \frac{2\pi \i}{T}m \not= (A_0- F_0)_{kk},\\
            \nm \ds
            -\sum_{\frac{2\pi \i}{T}m' \not= (A_0- F_0)_{kk}} \frac{(\Phi^{m'})_{kl}}{\frac{2\pi \i}{T}m' - (A_0)_{kk} + (F_0)_{ll}} &\text{ if } \frac{2\pi \i}{T}m = (A_0- F_0)_{kk}.
            \end{cases}
        \end{align}
        \end{theorem}
    \begin{proof}
    Firstly, from the choice of  $F_0$ it is clear that it  has no distinct eigenvalues which are congruent modulo ${2\pi \i}/{T}$. Consequently, from \cite{Yakubovich},  we obtain by using (iii)  that $F_\epsilon$ and $P_\epsilon$ are analytic functions of $\epsilon$ at $\epsilon =0$. Secondly, the rest of the results in the theorem can be easily proven by inserting the above equations for $P_n$ and $F_n$ into  \eqref{Floquet_decomposition_equation_general_order} and substituting with $\sum_{m \in \mathbb{Z}}\exp(\frac{2\pi \i}{T}mt)\Phi^m$. Furthermore, from the above formulas, it is clear that $F_n$ is constant and that $P_n(t)$ is $T$-periodic and satisfies $P_n(0) = 0$ for all $n \geq 1$. An explanation on why $F_0$ has been chosen in that particular way is presented in the next section.
    \end{proof}
    The main idea of the derivation of Theorem \ref{Inductive_Identity_for_Lyapunov-Floquet_decomposition} and formulas (\ref{Inductive_formula_F_n}) and (\ref{Inductive_formula_P_n}) has already been presented at the end of Section \ref{Diff_equation_for_Floquet_exp_mat_and_Lyapunov_trans}. Although Theorem \ref{Inductive_Identity_for_Lyapunov-Floquet_decomposition} is self-contained and provides a complete algorithm to compute the Lyapunov's transformation and the Floquet's exponent matrix, the next sections will draw closer attention to the formulas of the first orders of the Lyapunov's transformation and the Floquet's exponent matrix. They will elucidate further the dependence of $P_\epsilon$ and $F_\epsilon$ on $A_0$, $A_1$ and $A_2$. To further exemplify this connection, the following section will introduce the classical harmonic oscillator, which will allow for a simple but enriching application of Theorem \ref{Inductive_Identity_for_Lyapunov-Floquet_decomposition} and the subsequently derived equations for $F_0$, $F_1$ and some entries of $F_2$.
    
    Then, Section \ref{choice_of_leading_order_floquet_exponent} explains different possible choices of $F_0$ and their consequences for the Floquet-Lyapunov decomposition. It also gives reasons why the precise choice of $F_0$ in Theorem \ref{Inductive_Identity_for_Lyapunov-Floquet_decomposition} is made.

\subsection{Classical harmonic oscillator with modulated damping and restoring force}\label{Harmonic_oscillator}

The classical harmonic oscillator describes a point mass of mass $m$, attached to a weightless spring with suspension constant $k$, that is placed in a viscous liquid which leads to a damping factor $c$. Its dynamics is described by the differential equation
    $$m\frac{d^2 x}{dt^2} + c\frac{d x}{dt} + kx = 0,$$
where $x$ describes the displacement from the static state.
In the following, we will consider a harmonic oscillator where the damping constant $c$ is perturbed by a time dependent term $\epsilon\zeta(t)$ and where the suspension constant $k$ is also perturbed by a time dependent term $\epsilon\kappa(t)$. That is, the displacement $x$ is subject to the following equation:
    $$m\frac{d^2 x}{dt^2} + (c + \epsilon \zeta(t))\frac{d x}{dt} + (k + \epsilon\kappa(t))x = 0.$$
The equivalent first order ODE reads
    \begin{align*}
        \frac{d \Phi}{dt}= \left[\begin{pmatrix} \hspace{8pt}0 &\hspace{8pt}1\\
        -k &-c
        \end{pmatrix} + \epsilon\begin{pmatrix} 0 &0\\
        -\kappa(t) & -\zeta(t)
        \end{pmatrix}\right]\Phi,
    \end{align*}
where the mass $m$ has been set to 1.
Diagonalizing the leading order coefficient and writing $\alpha = \sqrt{c^2-4 k}$ then leads to the ODE
    \begin{align*}
        \frac{d \Phi}{dt} = \left( \begin{pmatrix}-\frac{c + \alpha}{2}\\
        &-\frac{c - \alpha}{2}\end{pmatrix} + \epsilon\left(\frac{\kappa(t)}{\alpha}\begin{pmatrix}1 &\frac{(c + \alpha)^2}{ 4  k}\\
        \frac{(c - \alpha)^2}{ 4  k} &-1\end{pmatrix} + \frac{\zeta(t)}{2}\begin{pmatrix}-1 - \frac{c}{\alpha} & -1 - \frac{c}{\alpha}\\
        -1 + \frac{c}{\alpha} & -1 + \frac{c}{\alpha} \end{pmatrix}\right)\right)\Phi.
    \end{align*}
In order to be able to  apply Theorem \ref{Inductive_Identity_for_Lyapunov-Floquet_decomposition} to the above equation, $\zeta$ and $\kappa$ need to be periodic functions with finite Fourier series. For the sake of illustration and simplicity, we will thus assume that
    \begin{align*}
        \zeta(t) &= \cos\left(\frac{2\pi}{T}at\right),\\
        \kappa(t) &= \cos\left(\frac{2\pi}{T}bt + \phi\right),
    \end{align*}
where $a$ and $b$ satisfy $\gcd(a,b) = 1$.
With the notation of Theorem \ref{Inductive_Identity_for_Lyapunov-Floquet_decomposition}, it holds that
    \begin{align*}
        A_0 &= \begin{pmatrix}-\frac{c+ \alpha}{2} \\
        &-\frac{ c - \alpha}{2}\end{pmatrix},\\
        A_1 &= \exp(\frac{2\pi}{T}(-a)t)\frac{1}{4}\begin{pmatrix}-1 - \frac{c}{\alpha} & -1 - \frac{c}{\alpha}\\
        -1 + \frac{c}{\alpha} & -1 + \frac{c}{\alpha} \end{pmatrix} +\exp(\frac{2\pi}{T}at)\frac{1}{4}\begin{pmatrix}-1 - \frac{c}{\alpha} & -1 - \frac{c}{\alpha}\\
        -1 + \frac{c}{\alpha} & -1 + \frac{c}{\alpha} \end{pmatrix}\\ &+\exp(\frac{2\pi \i}{T}(-b)t)\frac{e^{-\i\phi}}{2\alpha}\begin{pmatrix}1 &\frac{(c + \alpha)^2}{ 4  k}\\
        \frac{(c - \alpha)^2}{ 4  k} &-1\end{pmatrix} + \exp(\frac{2\pi \i}{T}bt)\frac{e^{\i\phi}}{2\alpha}\begin{pmatrix}1 &\frac{(c + \alpha)^2}{ 4  k}\\
        \frac{(c - \alpha)^2}{ 4  k} &-1\end{pmatrix},\\
        A_n &= 0 \text{ for } n \geq 2.
    \end{align*}

\subsection{Choice of the leading order Floquet exponent}\label{choice_of_leading_order_floquet_exponent}
    For the leading order Floquet exponent, that is, for $n=0$, the differential equation is given by
            $$\frac{d P_0}{dt}(t) = A_0P_0(t) - P_0(t)F_0,$$
            subject to the following conditions:
            \begin{equation*}
            \left\{
            \begin{aligned}
                &P_0(0) = \Id,\\
                &P_0: \mathbb{R} \rightarrow \mathbb{C}^{N\times N} \text{ is } T\text{-periodic},\\
                &F_0 \in \mathbb{C}^{N\times N}.
            \end{aligned}
            \right.
            \end{equation*}            
        Thus $P_0(t)$ is given by
            $P_0(t) = \exp((A_0-F_0)t)$
        and $F_0$ has to be chosen in such a way that $P_0(t)$ is $T$-periodic. In the case where $A_0$ is diagonal, this can be achieved by taking $F_0$ such that $A_0 -F_0$ is of the form
            $$A_0 - F_0 = \begin{pmatrix}
            \frac{2\pi \i}{T} \mathbb{Z}\\
            & \ddots\\
            &&\frac{2\pi \i}{T} \mathbb{Z}
            \end{pmatrix}.$$
        A natural choice would thus be to take $F_0$ diagonal in which case the real parts of its diagonal entries need to be given by the real parts of the diagonal entries of $A_0$. 
        For the imaginary parts there are more possibilities. A generally meaningful choice is however given by taking the imaginary parts inside the so-called \emph{first Brillouin zone} $\Im((F_0)_{ii}) \in [-\pi/T,\pi/T)$, that is, taking $\Im((F_0)_{ii})$ of minimal modulus for all $1 \leq i \leq N$. This makes sense when one interprets the Floquet decomposition as a decomposition of the solutions of the differential equation into a local, $T$-periodic part $P(t)$ and a global, non-periodic part $\exp(Ft)$. In order to differentiate between both, it is beneficial to choose $F$ in a way that minimizes the $T$-periodic oscillations that appear in $\exp(Ft)$. This is precisely achieved by taking the imaginary parts of $F_0$ to lie in the first Brillouin zone. Therefore, we set
            \begin{align*}
                (F_0)_{ii} = \Re((A_0)_{ii}) + \i\Im((F_0)_{ii}),
            \end{align*}
        where $\Im((F_0)_{ii})$ is chosen to be in $[-\frac{\pi }{T},\frac{\pi }{T})$ such that $\Im((A_0)_{ii}) - \Im((F_0)_{ii}) \in \frac{2 \pi }{T}\mathbb{Z}$. Another way to state this is that $\Im((F_0)_{ii})$ is the unique representative of $\Im((A_0)_{ii}) \in \mathbb{R}/\frac{2 \pi}{T}\mathbb{Z}$ in $[-\frac{\pi }{T},\frac{\pi }{T})$. Although this is the choice made for this paper, the same reasoning will also apply to a different choice of $F_0$, that is, a different choice of a set of representatives for elements in $\mathbb{R}/\frac{2 \pi}{T}\mathbb{Z}$ for $\Im((A_0)_{ii})$. The same derivations will then also apply. However, some degeneracies might need to be treated differently.
        
        For every choice of a representative of $\mathbb{R}/\frac{2 \pi}{T}\mathbb{Z}$, one can nevertheless write 
           \begin{align*}
               A_0 - F_0 = \begin{pmatrix}
            \frac{2\pi \i}{T}n_1\\
            &\frac{2\pi \i}{T}n_2\\
            && \ddots \\
            &&& \frac{2\pi \i}{T}n_{N}
            \end{pmatrix},
           \end{align*}
        with $n_1, \ldots, n_n \in \mathbb{Z}$ and the constant order Lyapunov's transformation $P_0(t)$ always taking the form
            \begin{align*}
                P_0(t) = \exp\left(\frac{2\pi \i}{T}n_1\right)E_{11} + \exp\left(\frac{2\pi \i}{T}n_2\right)E_{22} + \ldots + \exp\left(\frac{2\pi \i}{T}n_{N}\right)E_{NN},
            \end{align*}
        where $E_{ij}$ is the matrix who's entries are zero except at $(i,j)$ where its entry is  $1$.
        
        The following definitions will set the terminology.
        
        \begin{definition}[Constant order Floquet exponent, folding number]
            Let $A_\epsilon(t) = A_0 + \epsilon A_1(t) + \epsilon^2A_2(t) + \ldots \in \mathbb{C}^{N \times N}$ be an analytic square matrix-valued function and suppose all of the conditions of Theorem \ref{Inductive_Identity_for_Lyapunov-Floquet_decomposition} hold. Let $F_0 + \epsilon F_1 + \epsilon^2 F_2 + \ldots$ be as in Theorem \ref{Inductive_Identity_for_Lyapunov-Floquet_decomposition}. Then we call 
                $$(F_0)_{ii} \in \mathbb{R} \times \i[-\pi/T,\pi/T)$$
            the {constant order Floquet exponents} of the system 
            ${dx}/{dt} = A_\epsilon(t)x$
            and
                $$n_i := \frac{T((A_0)_{ii} - (F_0)_{ii}))}{2\pi \i} \in \mathbb{Z}$$
            their associated \emph{folding numbers}.
        \end{definition}
%    
%    The terminology is due to the fact that when computing the Floquet exponents $\{f_i(\epsilon)\}_{i\in\{1,\ldots,N\}}$ of the system ${dx}/{dt} = A_\epsilon(t)x$, it holds that the \emph{constant order Floquet exponents} $\{(F_0)_{ii}\}_{i\in\{1,\ldots,N\}}$ coincide with the constant order Taylor coefficients of the \emph{actual} Floquet exponents $\{f_i(\epsilon)\}_{i\in\{1,\ldots,N\}}$ of the system ${dx}/{dt}x = A_\epsilon(t)x$.
%    
    
In fact, the \emph{constant order Floquet exponents} are equal to the actual Floquet exponents up to an error of order $\epsilon$. That is, the constant order Floquet exponents are actually the constant order components of the actual Floquet exponents.

    Folding is a more subtle phenomenon. Depending on the period $T$ of the modulation, $(F_0)_{ii}$ might be equal to or different from $(A_0)_{ii}$, depending on whether $(A_0)_{ii} \in \mathbb{R} \times \i[-\pi/T,\pi/T)$ or not.
    In the latter case, $(A_0)_{ii} \in \mathbb{R} \times \i[(n_i-1)\pi/T,(n_i+1)\pi/T)$, where $n_i$ is the folding number associated to $(F_0)_{ii}$ and is \emph{not} equal to zero. In that case, one says that $F_0$ is \emph{folded}. The terminology \emph{folding}  refers to the fact that when computing
        \begin{align}
             &\mathbb{C} &\longrightarrow\hspace{40pt} &\mathbb{R} \times \i[-\pi/T,\pi/T)\\
            &(A_0)_{jj} &\longmapsto\hspace{40pt} &(F_0)_{jj} = (A_0)_{jj} \mod \frac{2\pi \i}{T},
        \end{align}    
    the space $\mathbb{R} \times \i\mathbb{R} \cong \mathbb{C}$ is ``wrapped around'' the infinitely extended cylinder 
        $$ \mathbb{R} \times \i[-\pi/T,\pi/T) \xrightarrow[\sim]{\exp((\cdot)/T)} \mathbb{R} \times \mathbb{S}^1,$$
    and each strip $\mathbb{R} \times \i[(n-1)\pi/T,(n+1)\pi/T)$ with $n \in \mathbb{Z}$ will correspond to another 'layer' that is wrapped around the cylinder. The folding number $n_i$ then specifies the 'layer' in which $(A_0)_{ii}$ is in, that is, the strip $\mathbb{R} \times \i[(n_i-1)\pi/T,(n_i+1)\pi/T)$ it lies on. Both the constant order Floquet exponent $(F_0)_{ii}$ and the corresponding folding number $n_i$  allow to recover $(A_0)_{ii}$. Namely, they satisfy
        $$(A_0)_{ii} = (F_0)_{ii} + n_i \frac{2\pi \i}{T}.$$
        
    Of course, $(F_0)_{ii}$ does not lie on a cylinder, but on the strip $\mathbb{R} \times \i[-\pi/T,\pi/T)$. The boundaries $\mathbb{R} \times \{-\pi \i/T\}$ and $\mathbb{R} \times \{\pi \i/T\}$ of the strip $\mathbb{R} \times \i[-\pi/T,\pi/T)$ will later lead to some discontinuities, which would not be there when one thinks of the Floquet exponents  as lying on a cylinder that does not have boundaries. 
    
%    However, the perturbation theory developed in this paper is only applicable, when $(F_\epsilon)_{ii}$ for $i=1, \ldots, N$, lie on a subset of $\mathbb{C}$, since only then we can write $F_0 + \epsilon F_1 + \epsilon^2 F_2 + \ldots$ for $F_\epsilon$.
    
    Folding might lead to coinciding diagonal entries of $F_0$, even when $A_0$ has distinct eigenvalues. This precisely happens when two entries, say $(A_0)_{ii}$ and $(A_0)_{jj}$, satisfy 
        $$(A_0)_{ii}-(A_0)_{jj} = n \frac{2\pi \i}{T}$$
    for some $n \in \mathbb{Z}$. In that case, $n$ is given by $n = n_i - n_j$. From another perspective, to make the constant order Floquet exponents $(F_0)_{ii}$ and $(F_0)_{jj}$ coincide, one needs to choose the modulation frequency $T$ in such a way that 
        $$((A_0)_{ii}-(A_0)_{jj})\frac{T}{2\pi \i} \in \mathbb{Z}.$$
    This will be very useful when one wants to \emph{create exceptional points}, i.e., points where the eigenvectors and eigenvalues of $F_\epsilon$ coincide. In other words, exceptional points are points where $F_\epsilon$ is not diagonalizable, but has a Jordan block of dimension strictly greater than $1$. This topic will be continued in the asymptotic analysis of Floquet exponents in Section \ref{asymptotic_analysis_of_floquet_exponents}.
    
    The example of the classical harmonic oscillator, see Section \ref{Harmonic_oscillator}, illustrates the mechanism of folding and folding numbers very well. 
    Figure \ref{fig:Plot_of_A0} displays the real and imaginary parts of the diagonal entries of $A_0$ when one parametrizes the system of the harmonic oscillator 
        \begin{align}\label{eq:Basic_mod_sys_ClaHarm}
            \frac{d \Phi}{dt} = \begin{pmatrix} 0 &1\\ -k-\epsilon\kappa(t) &-c-\epsilon\zeta(t)\end{pmatrix}\Phi, 
        \end{align}
    by varying the damping factor $c$ and setting the spring constant $k=c$.
    \begin{figure}[h!]
        \centering
        \includegraphics[width=0.7\columnwidth]{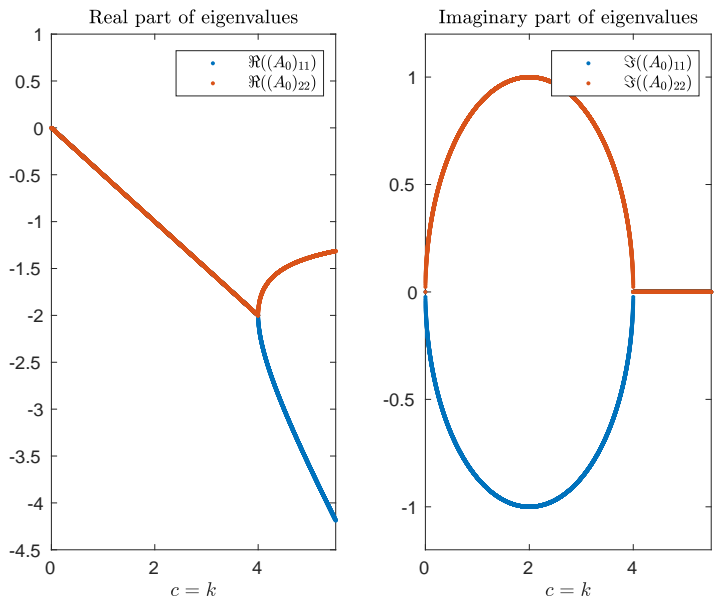}
        \caption{Constant order eigenvalues of the classical harmonic oscillator when one parametrizes the system \eqref{eq:Basic_mod_sys_ClaHarm} by varying the damping factor $c \in [0,5]$ and setting the spring constant $k=c$.}
        \label{fig:Plot_of_A0}
    \end{figure}
    In that setting, most $c \in [0,5]$ provide distinct diagonal entries, but for $c \in \{0,4\}$ the diagonal entries of $A_0$ are the same. Actually to be precise, at $c \in \{0,4\}$ the constant order system 
        \begin{align*}
            \frac{d \Phi}{dt} = \begin{pmatrix} 0 &-1\\ -c &-c \end{pmatrix}\Phi
        \end{align*}
    is not diagonalizable, that is, the system is posed at an \emph{exceptional point}. At those values for $c$, Theorem \ref{Inductive_Identity_for_Lyapunov-Floquet_decomposition} is not applicable. However, at all others $c \in [0,5]\setminus\{0,4\}$, it is. Furthermore, we will see later that exceptional points can be created through modulation by folding and at those exceptional points the theory is applicable.
    
    \begin{figure}
        \begin{center}
            \makebox[\textwidth]{\includegraphics[]{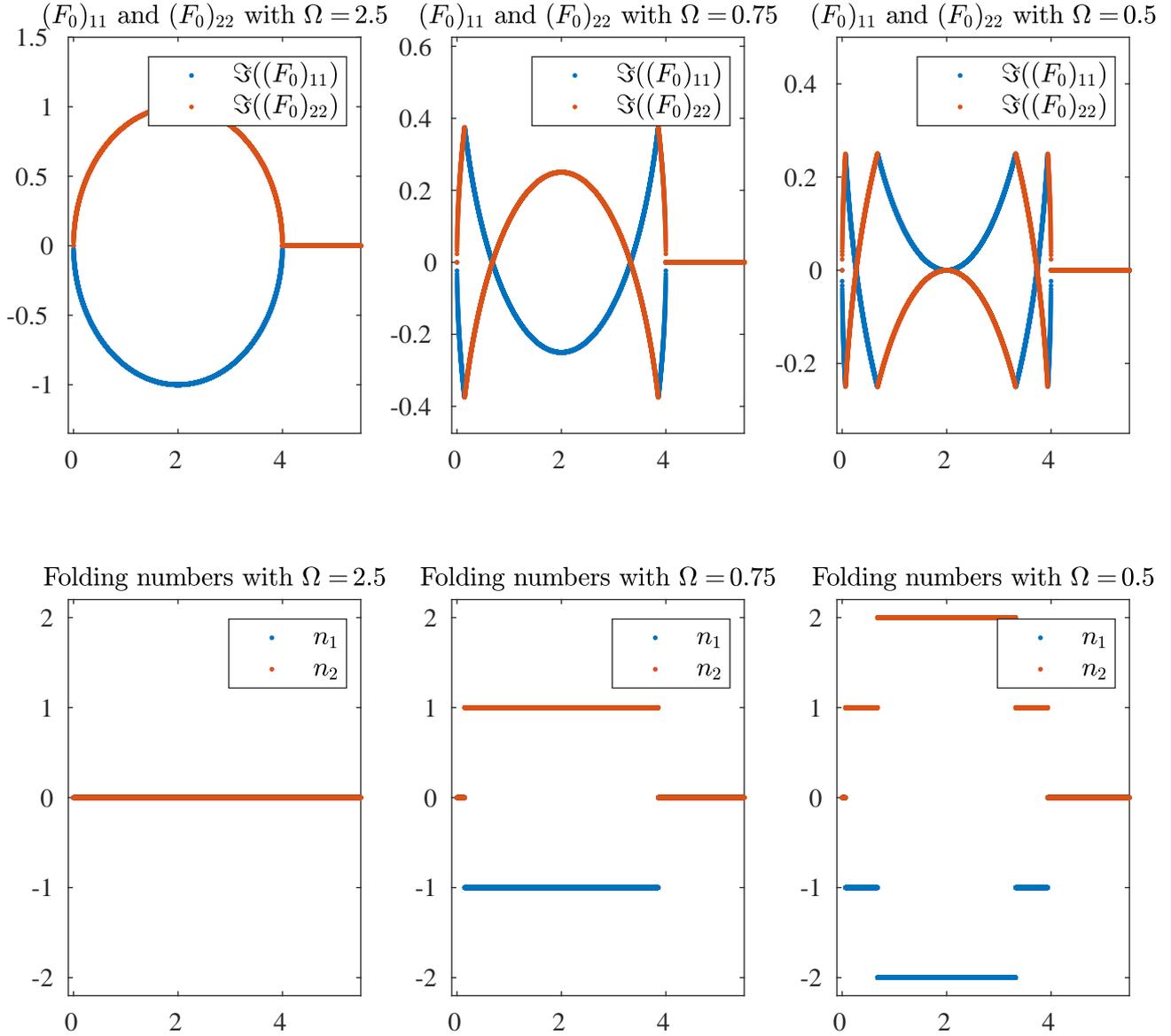}}
        \end{center}
        \centering
        \caption{Constant order Floquet exponents with different modulation frequencies $2\pi/T$ of the classical harmonic oscillator when one parametrizes the system \eqref{eq:Basic_mod_sys_ClaHarm} by varying the damping factor $c \in [0,5]$ and setting the spring constant $k=c$. The respective folding numbers are plotted in the second row.}
        \label{fig:Plot_of_F0_and_A0-F0}
    \end{figure}
    
    Depending on the period $T$ of the modulation $\kappa(t)$ and $\zeta(t)$, $F_0$ might differ from $A_0$, as can be seen in Figure \ref{fig:Plot_of_F0_and_A0-F0}. In that case, the folding numbers $n_1$ and $n_2$ differ from zero and $F_0$ might even have double eigenvalues at places where the eigenvalues of $A_0$ were distinct. 
    Concerning folding, one can very clearly see how the bands for the different eigenvalues leave the fundamental domain $\mathbb{R} \times [\frac{-\pi}{T},\frac{\pi}{T})$ and enter on the opposite side. At such a point $c_0$ where $(F_0(c_0^-))_{ii} \not= (F_0(c_0^+))_{ii}$, one should  think of the Floquet exponent as living on a cylinder, where the top part of the band $[\frac{-\pi}{T},\frac{\pi}{T})$ is glued to its bottom part in such a way one would have $F_0(c_0^-) = F_0(c_0^+)$. However, this thinking is not applicable here, since we are interested in an analytic expansion $F_\epsilon = F_0 + \epsilon F_1 + \ldots$, which is not possible on a cylinder.  In our setting, at such a point $c_0$, where $(F_0)_{ii}$ leaves the fundamental domain and enters on the other side, one can observe that its corresponding folding number $n_i$ changes by $1$. This precisely reflects the fact that at those ``exit''/``entry'' points, the corresponding eigenvalue of $A_0$ leaves the strip $\mathbb{R}\times \i[\frac{(n_i(c_0^-)-1)\pi}{T},\frac{(n_i(c_0^-)+1)\pi}{T})$ and enters the new strip $\mathbb{R}\times \i[\frac{(n_i(c_0^+)-1)\pi}{T},\frac{(n_i(c_0^+)+1)\pi}{T})$.

\subsection{Exact formulas for the first order Floquet's exponent matrix}\label{sec:exact_formula_F1}
    The previous section was dedicated to the study of the constant order Floquet exponent $F_0$ and its associated folding numbers $n_i$. However, without considering higher order analysis, the considerations made there are in some sense superficial. Strictly speaking, the constant order Floquet's exponent matrix is observed in nature\footnote{By this we mean an actual experiment where one provides some small modulations that one slowly turns down and estimates the associated Floquet's exponent matrix.} only by taking the limit of the actual Floquet's exponent matrix $F_\epsilon$ of the modulated system
        \begin{align*}
            \frac{d \Phi}{dt} = A_\epsilon(t)\Phi.
        \end{align*}
    That is, $F_0$ is obtained as
        \begin{align*}
            F_0 = \lim_{\epsilon \searrow 0} F_\epsilon,
        \end{align*}
    and is in most cases different from Floquet's exponent matrix in the case without modulation.  In fact, for the \emph{constant} system
        \begin{align*}
            \frac{d \Phi}{dt} =A_0\Phi,
        \end{align*}
    which is \emph{not} modulated in time, there is no folding and hence Floquet's exponent matrix is $A_0$. 
    This is why it is crucial to study the whole \emph{modulated} system and to investigate higher orders. In this section, closed formulas for the first order coefficients $F_1$ and $P_1(t)$ of the Floquet's exponent matrix $F_\epsilon$ and of the Lyapunov's transformation $P_\epsilon(t)$ are presented. Then, these formulas are illustrated by the example of a modulated classical harmonic oscillator from Section \ref{Harmonic_oscillator}.
    
    To determine $F_1$ and $P_1(t)$, one needs to either use Theorem \ref{Inductive_Identity_for_Lyapunov-Floquet_decomposition} or make the same derivation as explained at the end of Section \ref{Diff_equation_for_Floquet_exp_mat_and_Lyapunov_trans}. Either way one obtains the following theorem.

    \begin{theorem}[First order Floquet matrix and Lyapunov's transformation]\label{First_order_Floquet_matrix_and_Lyapunov_transformation}
    With the assumptions of Theorem \ref{Inductive_Identity_for_Lyapunov-Floquet_decomposition}, the choice of $F_0$ made there, and the notation $n_k = \frac{T}{2\pi \i}(A_0-F_0)_{kk}$, it holds that
    \begin{align*}
        (F_1)_{kl} = \begin{cases} \ds (A_1^{n_{k}-n_{l}})_{kl} &\text{ if } (F_0)_{kk} = (F_0)_{ll} ,\\
        \nm \ds
        ((F_0)_{ll}- (F_0)_{kk})\sum_{m \in \mathbb{Z}} \frac{(A_1^{m})_{kl}}{\frac{2\pi \i}{T}m +(A_0)_{ll}- (A_0)_{kk}} &\text{ if } (F_0)_{kk} \not= (F_0)_{ll},
        \end{cases}
    \end{align*}
    and that 
    \begin{align*}
        (P_1(t))_{kl} = \ds \sum_{m \in \mathbb{Z}} \exp(\frac{2\pi \i}{T}mt) \begin{cases} \frac{(A_1^{m-n_{l}})_{kl}}{\frac{2\pi \i}{T}m - (A_0)_{kk} +(F_0)_{ll}} &\text{ if } m \not= n_{k} , \\
        \nm \ds
        -\sum_{\Tilde{m} \not= n_{k}}\frac{(A_1^{\Tilde{m}-n_{l}})_{kl}}{\frac{2\pi \i}{T}\Tilde{m} - (A_0)_{kk} +(F_0)_{ll}} & \text{ if } m = n_{k}. \end{cases}
    \end{align*}
    \end{theorem}
    
    \begin{proof}
        Applying Theorem \ref{Inductive_Identity_for_Lyapunov-Floquet_decomposition}, one needs to compute the Fourier series coefficients of 
            \begin{align*}
                A_1(t)P_0(t) = \sum_{m \in \mathbb{Z}}\Phi^m.
            \end{align*}
        Denoting the respective Fourier series of $A_1$ and $P_0$ by 
            \begin{align*}
                A_1(t) = \sum_{m \in \mathbb{Z}}\exp(\frac{2\pi \i}{T}mt)A_1^m \quad \text{and} \quad
                P_0(t) = \sum_{m \in \mathbb{Z}}\exp(\frac{2\pi \i}{T}mt)P_0^m,
            \end{align*}
        respectively, it holds
            \begin{align*}
                \Phi^m = \sum_{n \in \mathbb{Z}}A_1^{m-n}P_0^n.
            \end{align*}
        Using the fact that $P_0(t) = \exp((A_0 -F_0)t)$ yields
            \begin{align*}
                (\Phi^m)_{kl} &= \sum_{j = 1}^N\sum_{n \in \mathbb{Z}}(A_1^{m-n})_{kj}(P^n_0)_{jl}\\
                            &= (A_1^{m - \frac{T}{2\pi \i}(A_0 - F_0)_{ll}})_{kl}.
            \end{align*}
        Inserting this into  equations \eqref{Inductive_formula_F_n}--\eqref{Inductive_formula_P_n} of Theorem \ref{Inductive_Identity_for_Lyapunov-Floquet_decomposition} and after re-indexing, the desired result follows. 
    \end{proof}
    
    Theorem \ref{First_order_Floquet_matrix_and_Lyapunov_transformation} shows that the entries of $F_1$ depend, firstly, on the first order modulation $A_1$, its frequencies $m$ and the frequency components $A_1^m$. Secondly, the entries of $F_1$ depend on the constant order Floquet exponent matrix $F_0$. Importantly,  the entries of $F_0$ determine which formula to use for the different entries of $F_1$. In particular, we have that the diagonal entries of $F_1$ are always given by the constant part of the first order modulation, namely
        \begin{align*}
            (F_1)_{kk} = (A_1^0)_{kk} \text{ for all } k \in \{1,\ldots,N\}
        \end{align*}
    and one obtains the following lemma.
    \begin{lemma}[Diagonal entries of $F_1$]\label{lemma:diagonal_entries_of_F1}
        Assuming the setting of Theorem \ref{Inductive_Identity_for_Lyapunov-Floquet_decomposition} and let $k \in \{1, \ldots, N\}$. Then,
            \begin{align*}
                (F_1)_{kk} = \big(\frac{T}{2\pi}\int_{0}^{\frac{2\pi}{T}} A_1(t)\,dt\big)_{kk}.
            \end{align*}
        In particular, if the the modulation does not have a constant Fourier part, then
            \begin{align*}
                (F_1)_{kk} = 0,
            \end{align*}
            and $F_1$ is off-diagonal.
    \end{lemma}
	
    The latter is a natural assumption, and is true for the classical harmonic oscillator whenever both modulation frequencies $a$ and $b$ are non-zero. For all other entries of $F_1$ nothing can be said in general, since its off-diagonal entries $(F_1)_{kl}$ depend on whether $(F_0)_{kk} = (F_0)_{ll}$ or not. To the end of clarifying further the off-diagonal entries of $F_1$, both cases are treated separately in the following two sections and are illustrated by the classical harmonic oscillator example.

    \subsubsection{Example when $F_0$ has distinct diagonal entries}\label{sec:example_when_F0_distinct_diag_entries}
        The case where $F_0$ has distinct diagonal entries corresponds, to some extent, to the \emph{generic case} and the case where $F_0$ has non-distinct eigenvalues corresponds to the rather \emph{exotic case}. By this we mean that choosing the diagonal entries of a matrix at random, with respect to a Lebesgue-type measure, almost surely provides a diagonal matrix with distinct eigenvalues. Looking at the example of the classical harmonic oscillator (see Figure \ref{fig:Plot_of_F0_and_A0-F0}) also reflects this kind of thinking, since the points $c\in [0,5]$ where $F_0$ has non-distinct eigenvalues are only finitely many points. 
        Thus, generically it holds
            $$(F_1)_{kl} = ((F_0)_{ll}- (F_0)_{kk})\sum_{m \in \mathbb{Z}} \frac{(A_1^{m})_{kl}}{\frac{2\pi \i}{T}m +(A_0)_{ll}- (A_0)_{kk}}$$
        for $k \not= l$.
        \begin{figure}
            \begin{center}
                \makebox[\textwidth]{\includegraphics[]{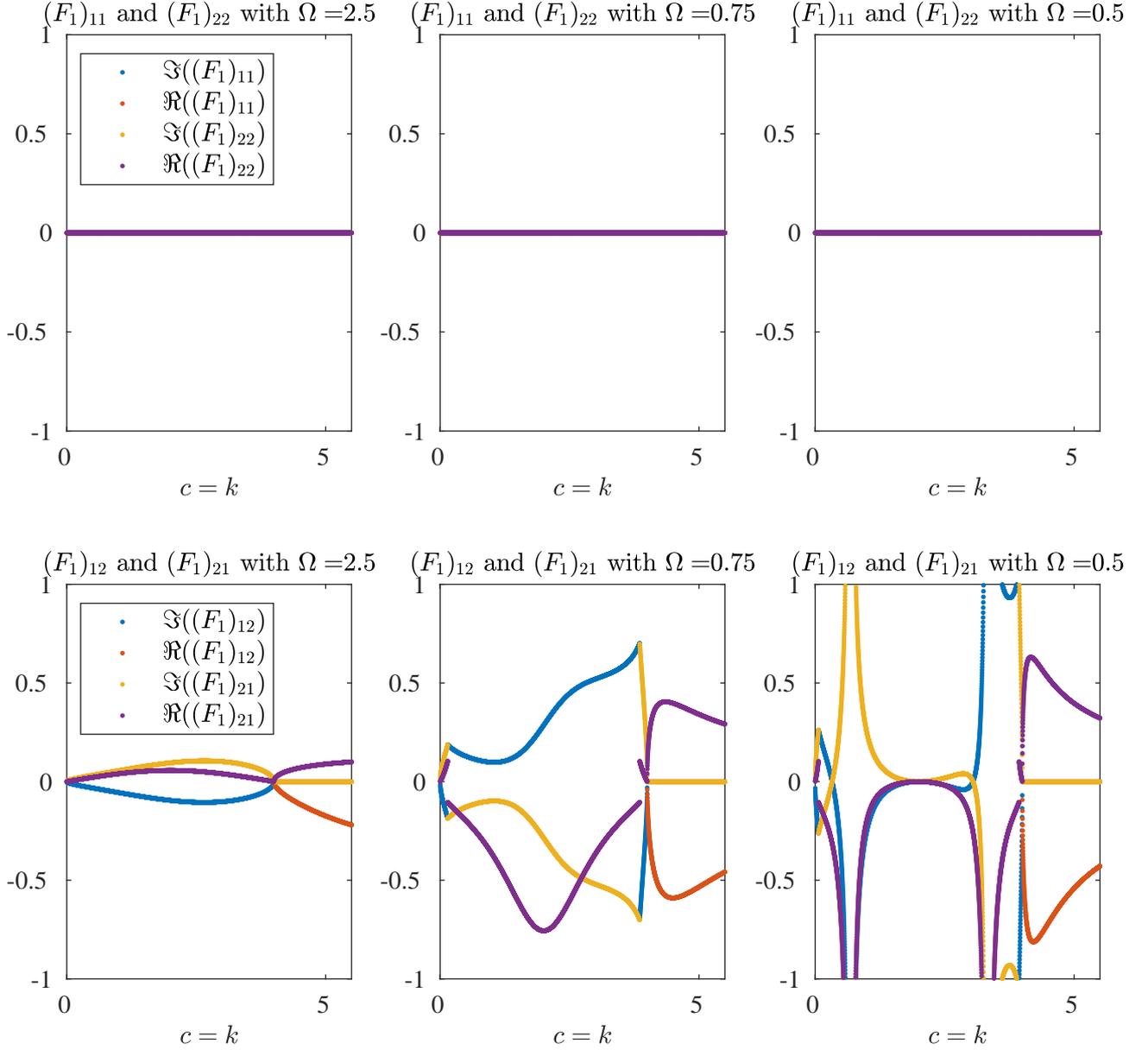}}
            \end{center}
            \centering
            \caption{Diagonal and off-diagonal entries of $F_1$ with $a=2$ and $b=3$ when one parametrizes the system \eqref{eq:Basic_mod_sys_ClaHarm} by varying the damping factor $c \in [0,5]$ and setting the spring constant $k=c$. Remarkable behavior appears around points $c$ where bands intersect (c.f. \Cref{fig:Plot_of_F0_and_A0-F0}). The off-diagonal entries of $F_1$ might be discontinuous, and if the modulation frequency coincides with the resonant frequency (c.f. \Cref{def_resonant_freq}), the entries might have poles. In either case, the diagonal entries vanish.}
            \label{fig:my_label2}
        \end{figure}
        \begin{figure}
            \begin{center}
                \makebox[\textwidth]{\includegraphics[]{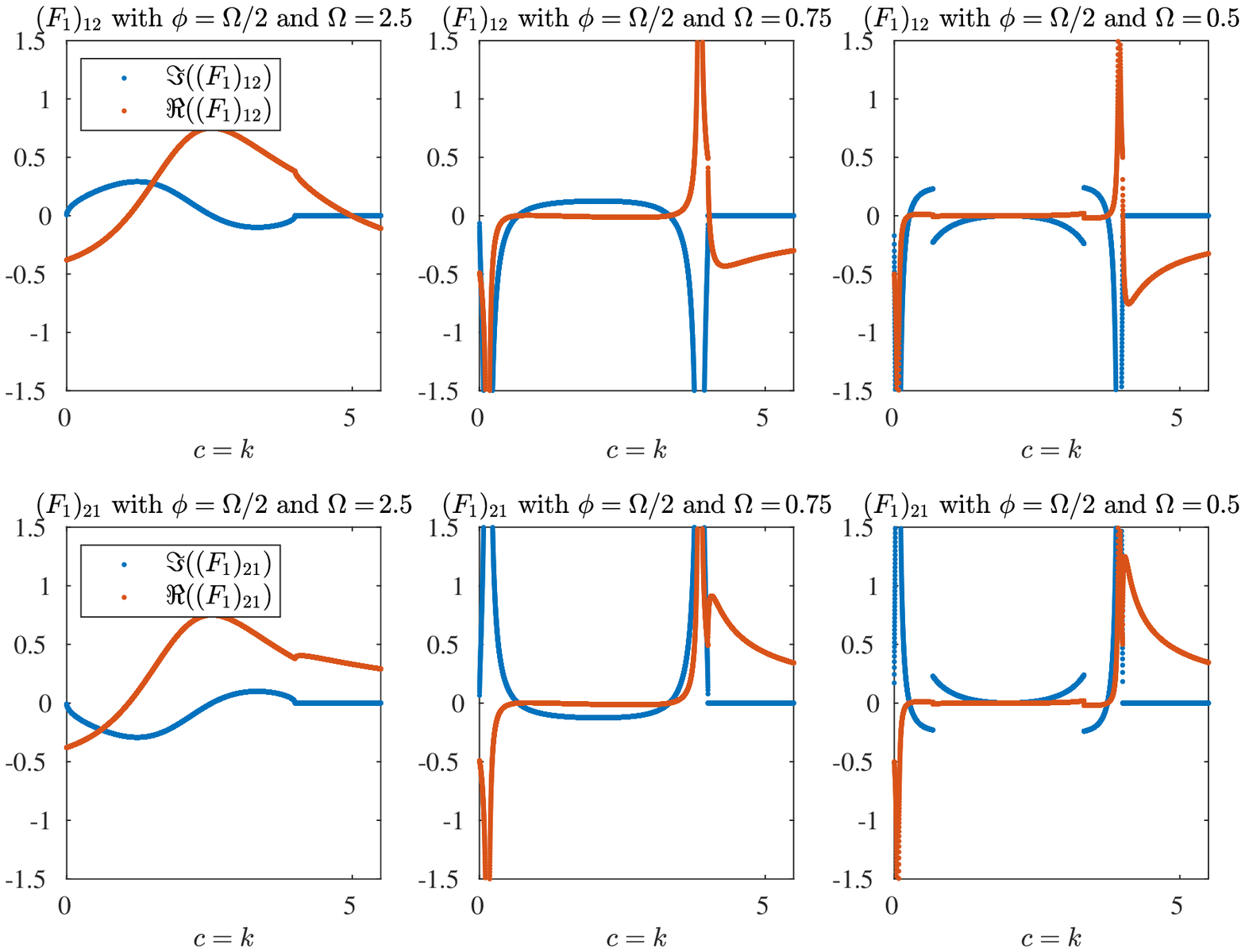}}
              \end{center}
            \caption{Off-diagonal entries of $F_1$ with $a=b=1$ when one parametrizes the system \eqref{eq:Basic_mod_sys_ClaHarm} by varying the damping factor $c \in [0,5]$ and setting the spring constant $k=c$. A comparison with \Cref{fig:my_label2} shows that different poles are excited in the current case, which are the ones corresponding to the resonant frequencies $n_i-n_j = 1$.}
            \label{fig:my_label1}
        \end{figure}

        Looking at a parametrized system $A_\epsilon(c,t)$ as the classical harmonic oscillator example from Section \ref{Harmonic_oscillator}, one can see that the matrix $F_1(c)$ depends continuously on $c\in D$, with $D$ being the domain where $F_0(c)$ has distinct eigenvalues. Outside of such a domain, the off-diagonal entries of $F_1(c)$ may have discontinuities. An example of those points are locations $c$ where $A_0(c)$ traverses the boundary of a fundamental domain $\i[\frac{(n-1)\pi}{T},\frac{(n+1)\pi}{T})\times \mathbb{R}$, with $n \in \mathbb{Z}$, since at those locations $F_0$ is discontinuous. Another location can be at places $c_0$ where $F_0(c_0)$ has non-distinct eigenvalues. In such cases, $F_1$ might even have divergent off-diagonal entries and its off-diagonal entries at $c_0$ might be given by another formula. This case is discussed in the next subsection.

    \subsubsection{Example when $F_0$ has non-distinct diagonal entries}\label{F1_when_F0_has_double_eigenvalue}
    As was pointed out in the previous section, the entries of $F_1$ might be unbounded for $c$ in a neighborhood of some critical value $c_0$ such that $F_0(c_0)$ has non-distinct diagonal entries. This phenomenon will actually persist until the analysis of the Floquet exponents. In some sense, these singularities correspond to the fact that the eigenvalues of a \emph{generic} $F_\epsilon$ have no linear component in  $\epsilon$. To be more precise, when one considers the Taylor expansion in $\epsilon$ of an eigenvalue $\lambda_\epsilon$ of $F_\epsilon$, then generically it will have the form
        \begin{equation}\label{eq:exp}
        	\lambda_\epsilon = \lambda_0 + \epsilon^2\lambda_2 + O(\epsilon^3),
        	\end{equation}
    where its linear term $\lambda_1$ is zero. However, when one considers a parametrized system $F_\epsilon(c)$, there can be choices of $c_0$ where the eigenvalues of $F_\epsilon(c_0)$ do have a first order dependence on $\epsilon$.
That is, for all sufficiently small $\delta \neq 0$, the eigenvalue of 
$F_\epsilon (c_0 +\delta)$ will have no linear term in its Taylor expansion.
     Assuming that the eigenvalues of $F_\epsilon(c)$ depend \emph{continuously} on $\epsilon$ \emph{and} on $c$, the only possibility of transition between $F_\epsilon(c_0)$ that has a linear component and $F_\epsilon(c_0 + \delta)$, that has no linear component, is to have big second order Taylor coefficients, so that in the limit $\delta \rightarrow 0$, the term $\epsilon^2\lambda_2(c_0 + \delta)$ ``seems'' to be linear for small $\epsilon$. This is obviously due to the fact that the eigenvalues of $F_\epsilon(c)$ do not depend analytically on $c$, but have a root-like behavior close to a double eigenvalue of $F_\epsilon(c_0)$. Consequently, \eqref{eq:exp} is valid pointwise in $c$ and $\lambda_2(c)$ does not depend continuously on $c$.
     
    In a general setting of a parametrized equation $dx(t)/dt = A_\epsilon(c,t)x(t)$, the discussion above shows that we cannot obtain asymptotic expansions of the form \eqref{eq:exp} which are valid uniformly for $c$ in a neighborhood of such critical value $c_0$. Nevertheless, in the setting of \Cref{Inductive_Identity_for_Lyapunov-Floquet_decomposition} (or for the classical harmonic oscillator with fixed $c$) we can obtain explicit asymptotic expansions. Considering  a double eigenvalue of $F_0$, we have the following lemma.
    %\todo[inline]{How about if we start this section with the lemma? There are two cases. Either we have $x' = A_\epsilon(t)x$ or we have a parametrized system $x' = A_{\epsilon,c}(t)x$. All the analysis is valid for the first system (or equivalently, for the second system with fixed value of $c$). Treating the general parametrized system is much harder, which is why we get these non-uniform asymptotics. Therefore, I think we should start by stating the result for the system without parameter, and then discuss the parametrized system.}
    \begin{lemma}\label{FloquetExponentMatrixForDoubleFloquetExponent}
        Assuming the setting of Theorem \ref{Inductive_Identity_for_Lyapunov-Floquet_decomposition} and supposing $F_0$ has an eigenvalue of multiplicity 2,
        \begin{align*}
            (F_0)_{kk} = (F_0)_{ll} \text{ for some fixed } k \not= l,
        \end{align*}
        then the corresponding entries of $F_1$ are given by
        \begin{align*}
            \begin{pmatrix}
        (F_1)_{kk} & (F_1)_{kl}\\
        (F_1)_{lk} & (F_1)_{ll}
        \end{pmatrix} = \begin{pmatrix}
        (A_1^0)_{kk} & (A_1^{n_{k}-n_{l}})_{kl}\\
        (A_1^{n_{l}-n_{k}})_{lk} & (A_1^0)_{ll}
        \end{pmatrix}.
        \end{align*}
    \end{lemma}
    \begin{proof}
        This directly follows from Theorem \ref{Inductive_Identity_for_Lyapunov-Floquet_decomposition}.
    \end{proof}
    This already shows that the modulation frequencies highly influence whether this double eigenvalue is perturbed linearly or quadratically. Namely, to first order, the perturbed eigenvalues are given by the eigenvalues of the matrix
    \begin{align*}
        \begin{pmatrix}
        (F_0)_{kk} \\
        &(F_0)_{ll}
        \end{pmatrix} + \epsilon
        \begin{pmatrix}
        (F_1)_{kk} & (F_1)_{kl}\\
        (F_1)_{lk} & (F_1)_{ll}
        \end{pmatrix}.
    \end{align*}
    Supposing that there is no constant first order perturbation of the system, i.e. $A_0 = 0$, the eigenvalues of $F_\epsilon$ that correspond to $(F_0)_{kk}$ and $(F_0)_{ll}$ are then given by the roots of the polynomial
    \begin{align*}
        (\lambda -(F_0)_{kk})(\lambda -(F_0)_{ll}) - \epsilon^2(A_1^{n_{k}-n_{l}})_{kl}(A_1^{n_{l}-n_{k}})_{lk}.
    \end{align*}
    Since $(F_0)_{kk} = (F_0)_{ll}$, they are thus given by
    \begin{align*}
        (F_0)_{kk} \pm \epsilon\sqrt{(A_1^{n_{k}-n_{l}})_{kl}(A_1^{n_{l}-n_{k}})_{lk}} + O(\epsilon^2).
    \end{align*}
	Hence, the perturbation of the constant system Floquet exponent $(F_0)_{kk}$ depends heavily on the frequency components of $A_1$. If $(F_0)_{kk}$ is a double eigenvalue and the modulation $A_1$ has non-zero frequency components of order $n_k-n_l$ and $n_l-n_k$, it follows that the double eigenvalue will be perturbed linearly. This leads to the following definition.
    
    \begin{definition}[Resonant frequencies of multiple Floquet exponent]\label{def_resonant_freq}
        Supposing the setting of Theorem \ref{Inductive_Identity_for_Lyapunov-Floquet_decomposition} and assuming that $(F_0)_{kk}$ is a multiple constant order Floquet exponent and let $l_1, \ldots, l_m \subset \{1,\ldots ,N\}$ be a complete non-redundant list of indices such that
            $$ (F_0)_{kk} = (F_0)_{l_jl_j}.$$
        Then the \emph{resonant frequencies of the multiple Floquet exponent $(F_0)_{kk}$} are given by
            $$\{n_{l_j} - n_{l_i} : i,j \in \{1, \ldots, m\} \text{ and } l_i \not= l_j \},$$
       where $n_{l_i}$ and $n_{l_j}$ denote the folding numbers of $(F_0)_{l_il_i}$ and $(F_0)_{l_jl_j}$, respectively.\footnote{The resonant frequencies, as defined here, are integer-valued and correspond to the physical frequencies $\frac{2\pi}{T}(n_{l_j} - n_{l_i})$. }
    \end{definition}
    Although the folding numbers of Floquet exponents are dependent on the choice of representative, the resonant frequencies of a multiple Floquet exponent are not. The next lemma will clarify why the terminology of \emph{resonant} frequencies is chosen.
    
    \begin{lemma}
        Assuming the setting of Theorem \ref{Inductive_Identity_for_Lyapunov-Floquet_decomposition}, let $(F_0)_{kk}$ be a Floquet exponent of the constant system
            $$\frac{dx}{dt} = A_0x$$
        and assume that $A_1$ has no constant component, then the following implication holds.
        If $(F_0)_{kk}$ is perturbed linearly in $\epsilon$, then \begin{enumerate}[(i)]
            \item The Floquet exponent $(F_0)_{kk}$ is a multiple Floquet exponent. In other words, there exists a complete list $l_1 \ldots l_m \subset \{1,\ldots ,N\}$ with $l_1 < \ldots < l_m$ and $m \geq 2$, such that $(F_0)_{kk} = (F_0)_{l_jl_j}$ for all $j \in \{1,\ldots, m\}$;
            \item The first order modulation has at least \emph{two} non-zero frequency components corresponding to the resonant frequencies 
            $$\{n_{l_j} - n_{l_i} : i,j \in \{1, \ldots, m\} \text{ and } l_i \not= l_j \}$$
            of the multiple Floquet exponent $(F_0)_{kk}$.
        \end{enumerate}
    \end{lemma}
    
    \begin{proof}
        The lemma is a consequence of the explanation before Definition \ref{def_resonant_freq}.
    \end{proof}
        
    For a better understanding of resonant frequencies and their influence on linear perturbation of multiple Floquet exponents, the example of the classical harmonic oscillator will be considered. The following will make use of the terminology and notations introduced in Section \ref{Harmonic_oscillator}.
    
    When fixing the modulation frequency $\Omega:= 2\pi/T$, but modifying the modulation frequencies $a$ and $b$, one can observe different degeneracies of the second order Floquet exponent perturbation; see for example Figure \ref{fig:Floquet_exponent_pertubation_ClaHarmOs_Different_modulation_frequencies}.
    At those degeneracies, as already discussed at the beginning of this section, the perturbation of the corresponding Floquet exponent has a non-zero linear component and is thus of \emph{first} and not of second order. To \emph{detect} those points, one needs to perform the following:
    \begin{enumerate}[(i)]
        \item Compute where band crossings occur in the folded system at $\epsilon=0$, that is, detect the presence of multiple constant order Floquet exponents;
        \item Compute the resonant frequencies of the multiple Floquet exponent and verify whether the modulation frequencies coincide with the resonant frequencies.
    \end{enumerate}
    If the last condition is met, then it is very likely that the corresponding Floquet exponent has a linear perturbation component. However, considering Lemma \ref{FloquetExponentMatrixForDoubleFloquetExponent}, one sees that the entries corresponding to the multiple eigenvalue of the perturbation matrix $A_1$ have to be non-zero and only then the perturbation of the multiple Floquet exponent is of first order.
    
    To \emph{generate} a first order perturbation of some Floquet exponents of a system, one thus needs to perform the following. Say, one would like to perturb the Floquet exponent corresponding to $(A_0)_{kk}$ linearly. Then
    \begin{enumerate}[(i)]
        \item One needs to fold the system in a way that the corresponding Floquet exponent $(F_0)_{kk}$ is a multiple Floquet exponent. That is, one has to choose some $l \in \{1, \ldots, N\}-\{k\}$ such that $\Re((A_0)_{kk}) = \Re((A_0)_{ll})$ and then one possible modulation period $\Omega$ would be given as
            $$\Omega = \Im((A_0)_{kk}) - \Im((A_0)_{ll});$$
        \item In a second step one needs to determine the resonant frequencies of the multiple Floquet exponent $(F_0)_{kk}$ and modulate the system with those frequencies in the first order.
    \end{enumerate}
    This procedure will generate a linear perturbation of the Floquet exponent $(F_0)_{kk}$ as long as the corresponding entries of $A_1$ are \emph{non-zero}.
    
    \begin{figure}
        \centering
            \includegraphics[width=\textwidth]{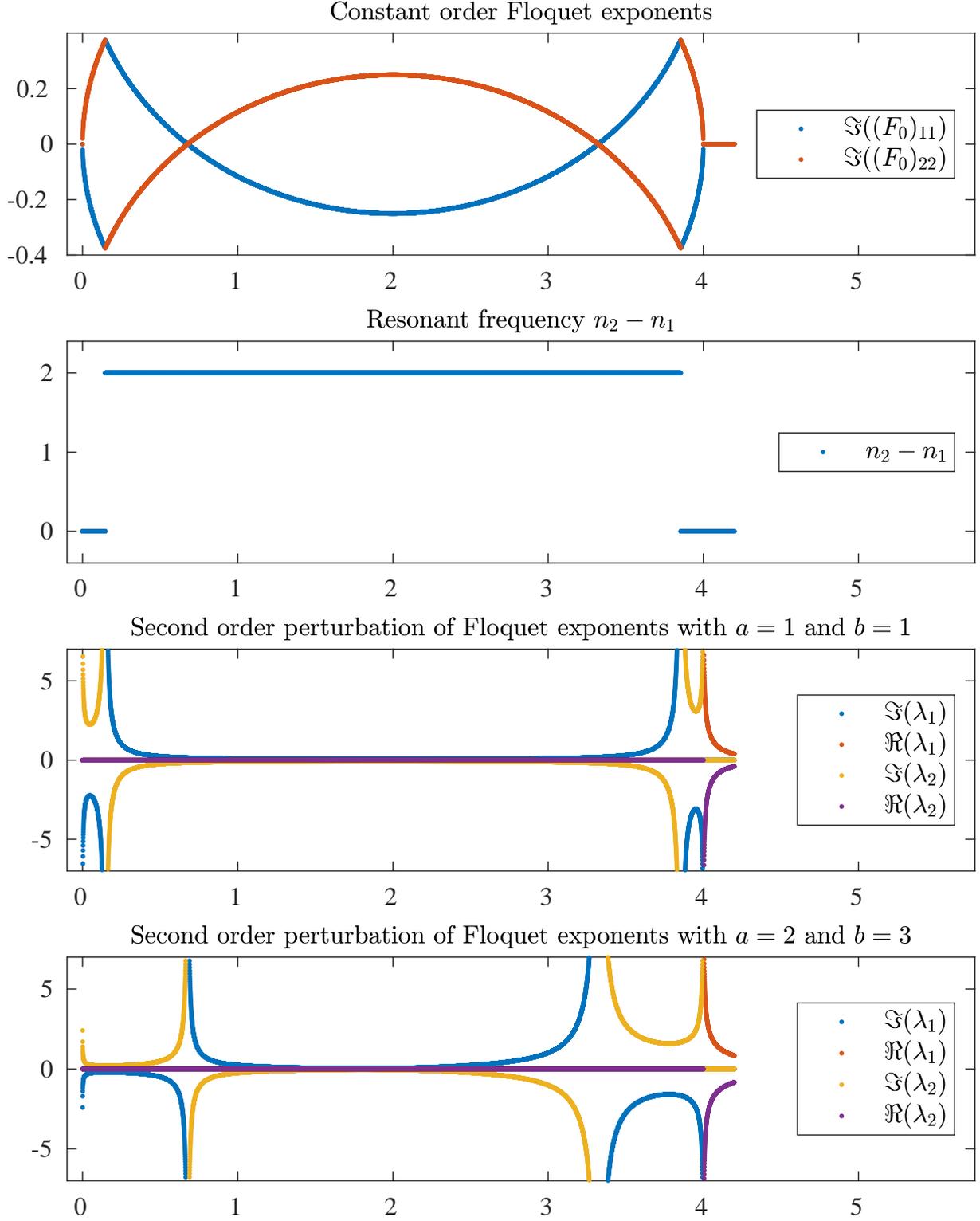}
            \caption{Here, the system is parametrized as in Figure \ref{fig:Plot_of_A0}. The modulation frequency $\Omega= 2\pi/T$ is set to be equal to $0.75$ and the modulation frequencies of $\zeta$ and $\kappa$ are given by $a$ and $b$, respectively. The first row displays the leading order Floquet exponents, that is, the eigenvalues of $F_0$. The second row displays the difference of the folding numbers and the last two rows display the second order eigenvalue perturbation with different modulation modes of $\kappa$ and $\zeta$. One observes that, depending on the modulation frequencies of $\kappa$ and $\zeta$, poles occur at \emph{particular} double eigenvalues of $F_0$. Indeed, in order to obtain a pole at a double eigenvalue of $F_0$, the modulation frequency needs to coincide with the resonant frequency.}
        \label{fig:Floquet_exponent_pertubation_ClaHarmOs_Different_modulation_frequencies}
    \end{figure}
    
    To make the above procedure clearer, we will consider the setting of the classical harmonic oscillator. The eigenvalues of its associated constant system, that is, the unperturbed system, are given by the diagonal entries of $A_0$ and are thus equal to
        $$ -\frac{c \pm \sqrt{c^2 - 4k}}{2}.$$
    We will consider the parametrized system of the classical harmonic oscillator with $k=c$ and $c \in [0,5.5]$, and we will set the modulation period $\Omega$ as $\Omega = 0.75$. Then, a multiple constant Floquet exponent occurs whenever
        $$ -\frac{c - \sqrt{c^2 - 4k}}{2} = -\frac{c + \sqrt{c^2 - 4k}}{2} \mod \i\Omega,$$
    that is, whenever
        $$ c = 2 \pm \sqrt{4 - (n \Omega)^2}, $$
    for some $n \in \mathbb{Z}$ and $c \in [0,4]$.\footnote{Note that the interval $[0,4]$ is strictly smaller than the interval $[0,5.5]$. This is due to the fact that the eigenvalues of the constant system have distinct real parts when $k=c$ and $ c>4$. Thus, they cannot be folded onto each other.} Thus, for $\Omega = 0.75$, multiple Floquet exponents occur at roughly
        $$ c \in \{ 0.146, 0.677, 3.323, 3.854\},$$
    where the first and last values correspond to resonant frequencies $\pm 2$ and the second and third to resonant frequencies $\pm 1$. This can also be observed in Figure \ref{fig:Floquet_exponent_pertubation_ClaHarmOs_Different_modulation_frequencies}. Namely, taking $c = 0.146$, one obtains
        \begin{align*}
            (A_0)_{22} = -\frac{c + \sqrt{c^2 - 4k}}{2} =  -0.0730 - 0.3750\i,\\
            (A_0)_{11} = -\frac{c - \sqrt{c^2 - 4k}}{2} = -0.0730 + 0.3750\i,
        \end{align*}
    and thus 
        \begin{align*}
            &(A_0)_{22} \hspace{23pt}= -0.0730 - 0.3750\i &= (F_0)_{11},\hspace{23pt}\\
            &(A_0)_{11} \hspace{23pt}= -0.0730 - 0.3750\i + \i\Omega &= (F_0)_{22} + \i\Omega,
        \end{align*}
    where one can see that the folding numbers of $(F_0)_{11}$ and $(F_0)_{22}$ are given by $n_0 = 0$ and $n_1 = 1$, respectively. Therefore, the resonant frequencies of the multiple Floquet exponent $(F_0)_{11} = (F_0)_{22}$ are given by $\pm 1$.
    Considering the procedure presented above, one needs to modulate the system with a first order modulation of frequency $-1$ and $1$, in order to perturb the Floquet exponents $(F_0)_{11}$ and $(F_0)_{22}$ linearly. For example, one can take $a=1$ and $b=1$, where the notation of Section \ref{Harmonic_oscillator} is used. The behavior of the second order perturbation of the Floquet exponents is depicted in Figure \ref{fig:Floquet_exponent_pertubation_ClaHarmOs_Different_modulation_frequencies}.\footnote{We refer the reader to Section \ref{sec:exact_formula_F2} for exact formulas for the necessary entries of the second order Floquet exponent matrix $F_2$. Additionally, Section \ref{asymptotic_analysis_of_floquet_exponents} is devoted to the analysis of the asymptotic behavior of Floquet exponents.} In the setting where the modulation frequencies $a$ and $b$ are set to $a= 1$ and $b= 1$, one can clearly see the degeneracy of the second order perturbation at $c = 0.146$. However, there is no degeneracy when one considers the case where the modulation frequencies are given by $a= 2$ and $b= 3$. In the first case, $F_0 + \epsilon F_1$ is given by
        \begin{align*}
            \begin{pmatrix}
        (F_0)_{11} \\
        &(F_0)_{22}
        \end{pmatrix} + \epsilon
        \begin{pmatrix}
        0 & (A^{-1}_1)_{12}\\
        (A^{1}_1)_{21} & 0
        \end{pmatrix},
        \end{align*}
    where 
        \begin{align*}
            (A^{(-1)}_1)_{12} &= \frac{-\alpha-c}{4\alpha} + \frac{e^{-\i\phi}(c+\alpha)^2}{8\alpha k}\\
            & =    0.0000 + 0.6667\i,\\
            (A^{(1)}_1)_{21} &=  \frac{-\alpha+c}{4\alpha} + \frac{e^{\i\phi}(c-\alpha)^2}{8\alpha k}\\
            & =   -0.5000 + 0.5694\i,
        \end{align*}
    when the phase shift $\phi$ between the modulation of the damping constant and the modulation of the spring constant is equal to $\phi = 0$.
    Thus the asymptotic expansions of the Floquet exponents $(f_\epsilon)_{1}$ and $(f_\epsilon)_{2}$ at $c = 0.146$ are given by
        \begin{align*}
            (f_\epsilon)_{1} &= (F_0)_{11} + \epsilon \sqrt{(A^{(-1)}_1)_{12}(A^{(1)}_1)_{21}} + O(\epsilon^2)\\
            &=  -0.0730 - 0.3750\i - \epsilon(0.2506 - 0.6651\i) + O(\epsilon^2)\\
            (f_\epsilon)_{2} &= (F_\epsilon)_{22} + \epsilon \sqrt{(A^{(-1)}_1)_{12}(A^{(1)}_1)_{21}} + O(\epsilon^2)\\
            &=-0.0730 - 0.3750\i + \epsilon(0.2506 - 0.6651\i) + O(\epsilon^2),
        \end{align*}
    but at $c \pm \delta$ for all small $\delta > 0 $, the Floquet exponents are given by
        \begin{align*}
            (f_\epsilon)_i = (F_0)_{ii} + O(\epsilon^2),
        \end{align*}
    and have no linear components in their Taylor expansions.

\subsection{Exact formulas of some particular values for second order}\label{sec:exact_formula_F2}
    We are interested in the asymptotic behavior of Floquet exponents of ODEs that satisfy the conditions of Theorem \ref{Inductive_Identity_for_Lyapunov-Floquet_decomposition}. In that particular setting, we proved before that Floquet exponents are generically  perturbed \emph{quadratically}. Hence, in order to analyze their asymptotic behavior, it is necessary to consider the second order Floquet exponent matrix $F_2$. We seek to determine the eigenvalues of $F_\epsilon = F_0 + \epsilon F_1 + \epsilon^2F_2 + \ldots$ up to order two in Section \ref{asymptotic_analysis_of_floquet_exponents}. To this end, it suffices to determine all coefficients of $F_2$ that correspond to the \emph{same} eigenvalue of $F_0$. That is, one needs to determine
        \begin{align*}
            (F_2)_{kl} \text{ for } k,l \text{ such that } (F_0)_{kk} = (F_0)_{ll}.
        \end{align*}
    To this end, it suffices to compute
        \begin{align*}
            (\Phi^{\frac{T}{2\pi \i}(A_0 -F_0)_{kk}})_{kl},
        \end{align*}
    where $(\Phi^m)_{m \in \mathbb{Z}}$ is given by
        \begin{align*}
            \sum_{m \in \mathbb{Z}}\Phi^m = A_1(t)P_1(t) - P_1(t)F_1 + A_2(t)P_0(t).
        \end{align*}
   
This gives rise to the following theorems. Theorem \ref{fisttheorem} expresses the
entries in terms of $A_0,\,A_1,\,F_0$ and $F_1$, whereas Theorem \ref{fisttheorem2} is in terms of $A_0,\,A_1$ and $F_0$.
    
    \begin{theorem}[Some particular entries of second order Floquet's exponent matrix] \label{fisttheorem}
        With the assumptions of Theorem \ref{Inductive_Identity_for_Lyapunov-Floquet_decomposition}, the choice of $F_0$ made there and the notation $n_k = \frac{T}{2\pi \i}(A_0 - F_0)_{kk}$, it holds that
            \begin{align*}
                (F_2)_{kl} = \biggl(\sum_{m \in \mathbb{Z}}A_1^{n_{k}-m}(t)P^m_1(t)\biggr)_{kl} 
                - (P_1^{n_{k}}(t)F_1)_{kl} + \biggl(\sum_{m \in \mathbb{Z}}A_2^{n_{k}-m}(t)P_0^m(t)\biggr)_{kl},
            \end{align*}
        where
            \begin{align*}
                &\biggl(\sum_{m \in \mathbb{Z}}A_1^{n_{k}-m}(t)P^m_1(t)\biggr)_{kl}
                    = \sum_{j = 1}^N\sum_{m \not= n_{j}-n_l}\frac{(A_1^{n_{k}-n_l-m} - A_1^{n_{k}-n_{j}})_{kj}(A_1^{m})_{jl}}{\frac{2\pi \i}{T}m + (A_0)_{ll} - (A_0)_{jj}},
            \end{align*}
            \begin{align*}
                &(P_1^{n_{k}}(t)F_1)_{kl}
                    = - \sum_{j=1}^N \sum_{m \not= n_{k}-n_j}\frac{(A_1^{m})_{kj}(F_1)_{jl}}{\frac{2\pi \i}{T}m + (A_0)_{jj} -(A_0)_{kk}}
            \end{align*}
            and
            \begin{align*}
                \biggl(\sum_{m \in \mathbb{Z}}A_2^{n_{k}-m}(t)P_0^m(t)\biggr)_{kl} = (A_2^{n_{k} - n_{l}})_{kl}.
            \end{align*}
            That is, $(F_2)_{kl}$ is given by 
            \begin{align*}
                (F_2)_{kl} = \sum_{j = 1}^N\sum_{m \not= n_{j}-n_l}\frac{(A_1^{n_{k}-n_l-m} - A_1^{n_{k}-n_{j}})_{kj}(A_1^{m})_{jl}}{\frac{2\pi \i}{T}m + (A_0)_{ll} - (A_0)_{jj}} + \sum_{j=1}^N \sum_{m \not= n_{k}-n_j}\frac{(A_1^{m})_{kj}(F_1)_{jl}}{\frac{2\pi \i}{T}m + (A_0)_{jj} -(A_0)_{kk}} \\ \nm \qquad \ds  + (A_2^{n_{k} - n_{l}})_{kl}.
            \end{align*}
    \end{theorem}
    
    Expressing also the first order Floquet exponent matrix in terms of $A_0,\, A_1$ and $F_0$, leads to longer expressions, which are given by the following theorem.

    \begin{theorem}[Some particular entries of second order Floquet's exponent matrix] \label{fisttheorem2}
        With the assumptions of Theorem \ref{Inductive_Identity_for_Lyapunov-Floquet_decomposition}, the choice of $F_0$ made there and the notation $n_k = \frac{T}{2\pi \i}(A_0 - F_0)_{kk}$, it holds that
            \begin{align*}
                (F_2)_{kl} = \biggl(\sum_{m \in \mathbb{Z}}A_1^{n_{k}-m}(t)P^m_1(t)\biggr)_{kl} 
                - (P_1^{n_{k}}(t)F_1)_{kl} + \biggl(\sum_{m \in \mathbb{Z}}A_2^{n_{k}-m}(t)P_0^m(t)\biggr)_{kl},
            \end{align*}
        where
            \begin{align*}
                &\biggl(\sum_{m \in \mathbb{Z}}A_1^{n_{k}-m}(t)P^m_1(t)\biggr)_{kl} = \sum_{j = 1}^N\sum_{m \not= n_{j}-n_l}\frac{(A_1^{n_{k}-n_l-m} - A_1^{n_{k}-n_{j}})_{kj}(A_1^{m})_{jl}}{\frac{2\pi \i}{T}m + (A_0)_{ll} - (A_0)_{jj}},
            \end{align*}
            \begin{align*}
                &(P_1^{n_{k}}(t)F_1)_{kl} \\
                    &= -\sum_{j = 1}^N\begin{cases}
                        \ds \sum_{m \not= n_{k}-n_j}\frac{(A_1^{m})_{kj}(A_1^{n_{j} - n_{l}})_{jl}}{\frac{2\pi \i}{T}m + (A_0)_{jj} -(A_0)_{kk}} \quad \text{ if } (F_0)_{jj} = (F_0)_{ll},\\
                        \nm \ds
                        \bigl((F_0)_{ll} - (F_0)_{jj}\bigr)\Biggl(\sum_{m \not= n_{k}-n_j}\frac{(A_1^{m})_{kj}}{\frac{2\pi \i}{T}m + (A_0)_{jj} -(A_0)_{kk}}\Biggr)\biggl(\sum_{\Tilde{m} \in \mathbb{Z}}\frac{(A_1^{\Tilde{m}})_{jl}}{\frac{2\pi \i}{T}\Tilde{m} + (A_0)_{ll} - (A_0)_{jj}}\biggr) \\ \nm \qquad \qquad \ds \text{ if } (F_0)_{jj} \not= (F_0)_{ll},
                    \end{cases}
            \end{align*}
            and
            \begin{align*}
                \biggl(\sum_{m \in \mathbb{Z}}A_2^{n_{k}-m}(t)P_0^m(t)\biggr)_{kl} = (A_2^{n_{k} - n_{l}})_{kl}.
            \end{align*}
    \end{theorem}

\begin{proof}
    Both proofs are completely analogous to the proof of Theorem \ref{First_order_Floquet_matrix_and_Lyapunov_transformation}.
\end{proof}
    
\subsection{Asymptotic expansion of the fundamental solution}\label{sec:asymp_exp_of_fund_sol}

Being in the setting of Theorem \ref{Inductive_Identity_for_Lyapunov-Floquet_decomposition}, that is, considering the ODE
    \begin{align*}
        \frac{d x}{dt} = A_\epsilon(t)x,
    \end{align*}
where
    \begin{align*}
        A_\epsilon = A_0 + \epsilon A_1(t) + \epsilon^2A_2(t) + O(\epsilon^3)
    \end{align*}
depends analytically  on the parameter $\epsilon$, then  the associated fundamental solution $X_\epsilon(t)$ depends analytically  on $\epsilon$ and has the following asymptotic expansion as $\epsilon  \rightarrow 0$: 
    \begin{align*}
        X_\epsilon&(t) = P_0(t)\exp(F_0t) + \epsilon\Big(P_1(t)\exp(F_0t) + 2P_0(t)\sum_{k = 1}^\infty\sum_{i=0}^{k-1} F_0^i F_1 F_0^{k-i-1} \frac{t^k}{k!}\Big)\\
        &+ \epsilon^2\Big( P_2(t)\exp(F_0t) + \sum_{k = 1}^\infty\sum_{i=0}^{k-1}\Big( P_1(t)F_0^i F_1 F_0^{k-i-1} + P_0(t)F_0^i F_2 F_0^{k-i-1} \\ \nm & + 2\sum_{j=0}^{i-1}F_0^j F_1 F_0^{i-j-1} F_1 F_0^{k-i-1}\Big)\frac{t^k}{k!}\Big)
        + O(\epsilon^3),
       \end{align*}
which is valid uniformly for $t$ in any compact subset of $\R$. This provides a possibility for further research. Depending on the operator norms of $F_0, \, F_1$ and $F_2$, one could estimate the number of terms needed for obtaining a reliable approximation of the fundamental solution. This would provide a new technique and supplement the usual procedure of applying standard ODE-algorithms which might be unstable when the coefficients of the corresponding ODE are modulated.

%% file: chapters/Asymptotic_analysis_of_Floquet_exponents.tex
\section{Asymptotic analysis of Floquet exponents}\label{asymptotic_analysis_of_floquet_exponents}

The formulas for the Floquet exponent matrix $F_\epsilon = F_0 + \epsilon F_1 + \epsilon^2 F_2 + O(\epsilon^3)$ now allow us to deduce asymptotic equations for the Floquet exponents of the linear system of ODEs
    \begin{align*}
        \frac{d x}{dt} = A_\epsilon(t)x.
    \end{align*}
In the following sections, asymptotic analysis of Floquet exponents will be established. That is, the eigenvalues of $F_\epsilon$ will be studied asymptotically in terms of $\epsilon$. To this end, the two cases introduced in Section \ref{sec:exact_formula_F1} will be treated.  Thus, the first case will be when $(F_0)_{ii}$ is a simple eigenvalue and the second case will be when $(F_0)_{ii}$ is a multiple eigenvalue of $F_0$. The first case will be called \emph{non-degenerate} and the second one will be called \emph{degenerate}. 

For both cases, the following result will be used to perform asymptotic analysis of their respective eigenvalues. We have the following standard result (see, e.g. \cite{jinghao,qmimperial}).

\begin{lemma}\label{Eigenvalue_pertubation}
    Let $M_{\epsilon} = M_0 + \epsilon M_1 + \epsilon^2 M_2 + O(\epsilon^3)$ be a diagonalizable matrix such that its eigenvalues and eigenvectors depend analytically on $\epsilon$. Furthermore, suppose that $M_0$ is diagonal. Then the eigenvalues of $M_\epsilon$ are given by
        $$ (M_0)_{ii} + \epsilon \lambda + O(\epsilon^3),$$
    where $\lambda = \lambda_1 + \epsilon \lambda_2$ is an eigenvalue of 
        \begin{align*}
            M_1 &+ \epsilon(M_1G_iM_1 + M_2)\\ &\text{ restricted to the respective eigenspace } \Span(e_j)_{(M_0)_{jj} = (M_0)_{ii}},
        \end{align*}
    with $G_i$ being a diagonal matrix given by
        \begin{align*}
            (G_i)_{jj} = \begin{cases} 0 &\text{ if } (M_0)_{jj} = (M_0)_{ii}, \\
            \frac{1}{(M_0)_{ii} - (M_0)_{jj}} &\text{ if } (M_0)_{jj} \not= (M_0)_{ii}.
            \end{cases}
        \end{align*}
\end{lemma}

Phrased differently, the eigenvalues of $M_\epsilon$ are asymptotically given up to third order in $\epsilon$ by the eigenvalues of the \emph{effective Hamiltonian}
    \begin{align*}
        (M_0 + \epsilon M_1 + \epsilon^2(M_1G_iM_1 + M_2) )\big|_{\Span(e_j)_{(M_0)_{jj} = (M_0)_{ii}}},
    \end{align*}
with $i \in \{1,\ldots,N\}$.

Before treating the degenerate and non-degenerate cases separately, some general results will be introduced in the following section.

\subsection{First observations}

In Section \ref{Chapter:Asymptotic_Floquet_theory}, it was explained how different choices of constant order Floquet exponent matrices would be possible and how they lead to different decompositions of the fundamental solution $X$ of the linear system of ODEs ${dx}/{dt} = A_\epsilon(t)x$. Nevertheless, as long as the choice of a lifting to a fundamental domain is consistent, the higher order perturbations of the respective Floquet exponents, that is, of the respective eigenvalues of $F_\epsilon$, are independent of the choice of $F_0$.

\begin{lemma}[Perturbations of Floquet exponents are independent of the choice of a representative]
    With the assumptions of Theorem \ref{Inductive_Identity_for_Lyapunov-Floquet_decomposition}, it holds that the eigenvalues of $F_\epsilon$ have uniquely defined higher order Taylor coefficients, that is, if $f_\epsilon = f_0 + \epsilon f_1 + \epsilon^2 f_2 + O(\epsilon^3)$ is an eigenvalue of $F_\epsilon$, then $f_1,f_2,\ldots$ are uniquely defined and thus they are independent of the choice of $f_0$, that is $F_0$.\footnote{In fact, the statement holds in a wider generality. The same proof applies whenever the Floquet exponents depend analytically on the parameter $\epsilon$ of the parametrized system of ODEs.}
\end{lemma}

\begin{proof}
    Let $X_\epsilon$ denote the fundamental solution of the $T$-periodic equation ${dx}/{dt} = A_\epsilon(t)x$ and denote by $\xi_1, \ldots, \xi_N$ the eigenvalues of $X_\epsilon(T)$. Then the Floquet exponents are defined by $\log(\xi_1),\ldots, \log(\xi_N)$ and \emph{depend} on the chosen logarithm branch. Let $j \in \{1,\ldots,N\}$, then 
        $$\xi_j = \exp(f_0 + \epsilon f_1 + \epsilon^2 f_2 + O(\epsilon^3))$$
    for some $f_0,f_1,f_2,\ldots$, which are given by some eigenvalues of a \emph{certain} choice of $F_\epsilon$. That is, $f_0 + \epsilon f_1 + \epsilon^2 f_2 + O(\epsilon^3)$ is some eigenvalue of a certain choice of Floquet exponent matrix $F_\epsilon$. Let $g_0 + \epsilon g_1 + \epsilon^2 g_2 + O(\epsilon^3)$ be the eigenvalue of another choice of $F_\epsilon$ such that 
        $$\xi_j = \exp(g_0 + \epsilon g_1 + \epsilon^2 g_2 + O(\epsilon^3)).$$
    Since $\xi_j$ is uniquely defined by $X_\epsilon$, it follows that 
        \begin{align*}
            \frac{d \xi_j}{d\epsilon} &= \frac{d}{d\epsilon}\exp(f_0 + \epsilon f_1 + \epsilon^2 f_2 + O(\epsilon^3)) \\
            &= \exp(f_0 + \epsilon f_1 + \epsilon^2 f_2 + O(\epsilon^3))(f_1 + \epsilon f_2 + O(\epsilon^2))
        \end{align*}
    is also uniquely defined and thus, for $\epsilon = 0$, one obtains
        \begin{align*}
            \exp(f_0)f_1=\exp(g_0)g_1.
        \end{align*}
    Now, $\exp(f_0) = \xi_j|_{\epsilon = 0} = \exp(g_0)$ implies that $f_1 = g_1$. It then follows  inductively that all higher order Taylor coefficients of $f_\epsilon$ and $g_\epsilon$ need to coincide and thus it follows that $f_n$ is uniquely defined for all $n \geq 1$.
\end{proof}

Another general behavior is that (as for the eigenvalues of a real-valued matrix) the Floquet exponents of an ODE with real coefficients occur in conjugate pairs. 

\begin{lemma} \label{conjugate}
    Let $A(t)$ be a matrix valued function and assume that the associated linear system of ODEs ${dx}/{dt} = A(t)x$ has a unique fundamental solution. It holds that whenever $A$ is real-valued, the associated Floquet exponents $f_1, \ldots, f_N$ occur in complex conjugate pairs modulo ${2\pi \i}/{T}$.
\end{lemma}

\begin{proof}
    Let $f_i$ be a Floquet exponent of the $T$-periodic linear system of ODEs
        $$\frac{dx}{dt} = A(t)x,$$
    and denote by $X$ its fundamental solution. Then
        $\exp(f_iT)$
    is an eigenvalue of $X(T)$. Since ${dx}/{dt} = A(t)x$ is preserved under complex conjugation it follows that $\bar{X}$ is also a fundamental solution to the linear system of ODEs. Since the fundamental solution is unique, it thus follows that 
        $\exp(\bar{f}_iT)$
    is also an eigenvalue of $X(T)$. Therefore $\bar{f}_i$ and $f_i$ are Floquet exponents of the ODE ${dx}/{dt} = A(t)x$.
\end{proof}

The above phenomenon can actually be observed in Figure \ref{fig:Floquet_exponent_pertubation_ClaHarmOs_Different_modulation_frequencies} for the harmonic oscillator setting. Actually, in the setting of the classical harmonic oscillator, $A_0$ and $A_1$ are \emph{not} real-valued. However, they originate from a differential equation with real coefficients. Since the Floquet exponents are invariant under linear transformation of the ODE, it suffices to require that the coefficient matrix can be \emph{transformed} into a real-valued coefficient matrix. The phenomenon of conjugate Floquet exponents will also be present in the application section (Section \ref{ch:metamaterials}) on metamaterials. 

\subsection{Non-degenerate case of asymptotic analysis of Floquet exponent}

In the non-degenerate setting, we are interested in how the simple eigenvalue $(F_0)_{ii}$ is perturbed. To this end, the following corollary of Lemma \ref{Eigenvalue_pertubation} will be useful.

\begin{cor}[Eigenvalue perturbation for simple constant order eigenvalue]\label{cor:eigenvalue_perturbation_for_single_eigenvalue}
    Let $M_\epsilon$ be as in Lemma \ref{Eigenvalue_pertubation} and suppose that $(M_0)_{ii}$ is a simple eigenvalue of $M_0$. Then 
        $$(M_0)_{ii} + \epsilon(M_1)_{ii} + \epsilon^2\left(\sum_{j = 1,\\ j\not=i}^{N}\frac{(M_1)_{ij}(M_1)_{ji}}{(M_0)_{ii}-(M_0)_{jj}} + (M_2)_{ii}\right) + O(\epsilon^3)$$
    is an eigenvalue of $M_\epsilon$.
\end{cor}

This leads to the following result.

\begin{lemma}[Simple Floquet exponents are perturbed quadratically]\label{lem:single_Floquet_exponents_are_perturbed_quadratically}
    In the setting of Theorem \ref{Inductive_Identity_for_Lyapunov-Floquet_decomposition}, it holds that if $(F_0)_{ii}$ is a simple constant order Floquet exponent and the first order perturbation $A_1$ has no constant component, that is, if
        $$A_1^{(0)} = 0,$$
    then $(F_0)_{ii}$ is not perturbed linearly, but quadratically. That is, the Floquet exponent $f_i(\epsilon)$ corresponding to $(F_0)_{ii}$ takes the form
        $$ f_i(\epsilon) = (F_0)_{ii} + O(\epsilon^2).$$  
\end{lemma}

\begin{proof}
    Corollary \ref{cor:eigenvalue_perturbation_for_single_eigenvalue} implies that $f_i(\epsilon)$ is given by
        $$f_i(\epsilon) = (F_0)_{ii} + \epsilon(F_1)_{ii} + O(\epsilon^2).$$
    Now, Lemma \ref{lemma:diagonal_entries_of_F1} provides a formula for the diagonal entries of $F_1$, and thus $(F_1)_{ii}$ is given by
        $$(F_1)_{ii} = \left(\frac{T}{2\pi}\int_{0}^{\frac{2\pi}{T}} A_1(t)\,dt\right)_{ii},$$
    which is precisely
        $$(F_1)_{ii} = (A^{(0)}_1)_{ii}.$$
    By assumption, $A_1$ has no constant frequency component and thus
        $$f_i(\epsilon) = (F_0)_{ii} + 0 + O(\epsilon^2)$$
    and the desired result follows.
\end{proof}

In Figure \ref{fig:ClaHarm_const_nonconst_perturbation}, this phenomenon can be observed in the case of the classical harmonic oscillator. It is clearly visible that in the non-constant setting, the Floquet exponents are perturbed quadratically, whereas in the setting where $A_1$ has also a constant component, a linear perturbation can be observed.

\begin{figure}
    \centering
    \includegraphics[width=0.5\textwidth]{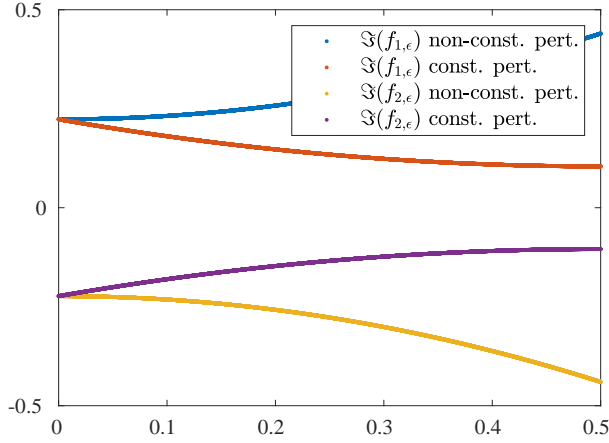}
    \caption{Imaginary part for the perturbation in the Floquet exponent when $A_1$ has a constant component and  when $A_1$ has no constant component in the classical harmonic oscillator case. The non-constant case has been modulated with frequencies $a=b=1$ and the constant case has been modulated with frequencies $a = 1$ and $b = 0$. All other parameters have been chosen as follows: $\Omega = 0.75$, $c=k=0.3$ with no phase shift, i.e. $\phi =0$.}
    \label{fig:ClaHarm_const_nonconst_perturbation}
\end{figure}

Actually, $F_\epsilon$ is given in the constant and in the non-constant setting by the following
\begin{align*}
    F_\epsilon &= \begin{pmatrix}
        -0.15 + 0.22\i	\\
        &-0.15 - 0.22\i
    \end{pmatrix} 
    + \epsilon\begin{pmatrix}
        - 0.47\i	&0.32 - 0.29\i\\
        0.32 + 0.29\i	&0.47\i
    \end{pmatrix} + O(\epsilon^2),\\
    F_\epsilon &= \begin{pmatrix}
        -0.15 + 0.22\i	\\
        &-0.15 - 0.22\i
    \end{pmatrix} 
    + \epsilon\begin{pmatrix}
        0	&- 0.81\i\\
        0.81\i	&0
    \end{pmatrix} + O(\epsilon^2),
\end{align*}
respectively. Indeed, in the setting where $A_1$ has no constant component, one observes that the diagonal entries of $F_1$ are equal to zero.

% This leads to the following result

% \begin{lemma}
%     With the assumptions of Theorem \ref{Inductive_Identity_for_Lyapunov-Floquet_decomposition} and denoting by $F_\epsilon = F_0 + \epsilon F_1 + \epsilon^2 F_2 + O(3)$ the associated Floquet exponent matrix of the ODE $\frac{d}{dt}x = A_\epsilon(t)x$, the following holds. If $(F_0)_{ii}$ is a \emph{simple} eigenvalue of $F_0$, then
%         $$ (F_0)_{ii} + \epsilon(F_1^{0})_{ii} + \epsilon^2\lambda_i^2 + O(\epsilon^3)$$
%     is an eigenvalue of $F_\epsilon$, where
%     \begin{align*}
%         \lambda_i^2 = \sum_{j = 1}^N\sum_{n \not= n_j}\frac{(A_1^{n_i -n})_{i,j}(A_1^{n-n_i})_{j,i}}{\frac{2\pi \i}{T}n - (A_0)_{jj}+ (F_0)_{ii}} + \sum_{j \not=i} \frac{(A_1^{n_i -n_j})_{i,j}(A_1^{n_j-n_i})_{j,i}}{(A_0)_{jj} - (A_0)_{ii}} + (A_2^0)_{ii}\\
%         = \sum_{j = 1}^N\sum_{m \not= n_j-n_i}\frac{(A_1^{-m})_{i,j}(A_1^{m})_{j,i}}{\frac{2\pi \i}{T}m + (A_0)_{ii} - (A_0)_{jj}} + \sum_{j \not=i} \frac{(A_1^{n_i -n_j})_{i,j}(A_1^{n_j-n_i})_{j,i}}{(A_0)_{jj} - (A_0)_{ii}} + (A_2^0)_{ii}
%     \end{align*}
% \end{lemma}

\subsection{Degenerate case of asymptotic analysis of Floquet exponent}

In this section, the case of a \emph{double} eigenvalue $(F_0)_{ii} = (F_0)_{jj}$ for some $ i \not= j$ and its perturbations will be treated. To this end, the following corollary of Lemma \ref{Eigenvalue_pertubation} will be useful.

\begin{cor}\label{cor:eigenvalue_perturbation_of_double_eigenvalue}
    Let $M_\epsilon$ be as in Lemma \ref{Eigenvalue_pertubation} and suppose that $(M_0)_{ii}$ is a double eigenvalue of $M_0$ with $(M_0)_{jj} = (M_0)_{ii}$ for some $j$. Then the eigenvalues of $M_\epsilon$ associated to $(M_0)_{ii}$ are given by the eigenvalues of the effective Hamiltonian associated to $(M_0)_{ii}$, which is given by  
        \begin{align*}
            \begin{pmatrix}
                    (M_0)_{ii} &0\\
                    0 &(M_0)_{jj}
            \end{pmatrix} + 
            \epsilon\begin{pmatrix}
                    (M_1)_{ii} &(M_1)_{ij}\\
                    (M_1)_{ji} &(M_1)_{jj}
            \end{pmatrix} +
            \epsilon^2\begin{pmatrix}
                    (M_1GM_1 + M_2)_{ii} &(M_1GM_1 + M_2)_{ij}\\
                    (M_1GM_1 + M_2)_{ji} &(M_1GM_1 + M_2)_{jj}
            \end{pmatrix}
            + O(\epsilon^3),
        \end{align*}
        with $(M_1GM_1 + M_2)_{kl}$ for $k,l \in \{i,j\}$ being equal to
        \begin{align*}
                (M_1GM_1 + M_2)_{kl} = \sum_{n = 1, n\not=i,j}^{N}\frac{(M_1)_{kn}(M_1)_{nl}}{(M_0)_{ii}-(M_0)_{nn}} + (M_2)_{kl}.
        \end{align*}
\end{cor}

Corollary \ref{cor:eigenvalue_perturbation_of_double_eigenvalue} allows to analyze the perturbation of a double constant order Floquet exponent and to also treat the case of exceptional points. In the following subsection, different possible asymptotics of a double constant order Floquet exponents are explored. Although simple constant order Floquet exponent are \emph{always} perturbed quadratically (when one assumes no constant first order perturbation of the corresponding ODE), it becomes apparent that this ``generic'' analytic behavior of Floquet exponents is no longer present in the degenerate case of a double constant order Floquet exponent. Namely, in the case of a double constant order Floquet exponent, it is possible to achieve linear perturbation of the constant order Floquet exponent by perturbing the corresponding ODE with the associated \emph{resonant frequencies}. In Subsection \ref{sec:exceptional_points} the possibility of occurrence of asymptotic exceptional points is addressed.

\subsubsection{Eigenvalue perturbation of double Floquet exponent}

Corollary \ref{cor:eigenvalue_perturbation_of_double_eigenvalue} leads to the following first order perturbation of a double constant order Floquet exponent.
    
\begin{lemma}[Linear perturbation of double Floquet exponent]\label{lemma:linear_perturbation_of_double_eigenvalue}
    With the assumptions of Theorem \ref{Inductive_Identity_for_Lyapunov-Floquet_decomposition} and denoting by $F_\epsilon = F_0 + \epsilon F_1 + \epsilon^2 F_2 + O(\epsilon^3)$ the associated Floquet exponent matrix of the equation ${dx}/{dt} = A_\epsilon(t)x$, the following holds. If $(F_0)_{ii} = (F_0)_{jj}$ with $i \not= j$ is a \emph{double} eigenvalue of $F_0$, then the first order effective Hamiltonian associated to $(F_0)_{ii}$ is given by
        \begin{align*}
            \begin{pmatrix}
                    (F_0)_{ii} \\
                    & (F_0)_{jj}
            \end{pmatrix}
            + \epsilon\begin{pmatrix}
                    (A_1^{(0)})_{ii} & (A_1^{(n_i - n_j)})_{ij} \\
                    (A_1^{(n_j - n_i)})_{ji} & (A_1^{(0)})_{jj}
            \end{pmatrix},
        \end{align*}
    where $n_i, n_j$ denote the folding numbers of $(F_0)_{ii}, (F_0)_{jj}$, respectively. Furthermore, up to first order in $\epsilon$, the Floquet exponents associated to $(F_0)_{ii}$ are given by
        \begin{align*}
            (F_0)_{ii} + \epsilon\left(\frac{(A_1^{(0)})_{ii} + (A_1^{(0)})_{jj}}{2} \pm \sqrt{\left(\frac{(A_1^{(0)})_{ii} - (A_1^{(0)})_{jj}}{2}\right)^2 + (A_1^{(n_i - n_j)})_{ij}(A_1^{(n_j - n_i)})_{ji}}\right).
        \end{align*}
\end{lemma}

\begin{proof}
The identity of the first order effective Hamiltonian is a direct consequence of Lemmas \ref{Eigenvalue_pertubation}  and  \ref{FloquetExponentMatrixForDoubleFloquetExponent}.
The equation for the first order approximation of the Floquet exponents associated to $(F_0)_{ii}$ is due to the fact that the characteristic polynomial of 
    \begin{align*}
        \begin{pmatrix}
                    (A_1^{(0)})_{ii} & (A_1^{(n_i - n_j)})_{ij} \\
                    (A_1^{(n_j - n_i)})_{ji} & (A_1^{(0)})_{jj}
            \end{pmatrix}
    \end{align*}
    is given by 
    \begin{multline*}
        \left( (A_1^{(0)})_{ii} - \lambda \right)\left((A_1^{(0)})_{jj} - \lambda\right) - (A_1^{(n_i - n_j)})_{ij}(A_1^{(n_j - n_i)})_{ji}\\
        = \lambda^2 - \left( (A_1^{(0)})_{ii} + (A_1^{(0)})_{jj}\right) \lambda + (A_1^{(0)})_{ii}(A_1^{(0)})_{jj}- (A_1^{(n_i - n_j)})_{ij}(A_1^{(n_j - n_i)})_{ji}.
    \end{multline*}
    Hence, its eigenvalues are
    \begin{align*}
        \frac{(A_1^{(0)})_{ii} + (A_1^{(0)})_{jj}}{2} \pm \sqrt{\left(\frac{(A_1^{(0)})_{ii} + (A_1^{(0)})_{jj}}{2}\right)^2 - (A_1^{(0)})_{ii}(A_1^{(0)})_{jj} + (A_1^{(n_i - n_j)})_{ij}(A_1^{(n_j - n_i)})_{ji}}\\
        = \frac{(A_1^{(0)})_{ii} + (A_1^{(0)})_{jj}}{2} \pm \sqrt{\left(\frac{(A_1^{(0)})_{ii} - (A_1^{(0)})_{jj}}{2}\right)^2 + (A_1^{(n_i - n_j)})_{ij}(A_1^{(n_j - n_i)})_{ji}},
    \end{align*}
    and the desired result follows.
\end{proof}

In the special case where the modulation has no constant component, that is, in the situation where 
\begin{align*}
    A_1^{(0)} = 0,
\end{align*}
the formulation of Lemma \ref{lemma:linear_perturbation_of_double_eigenvalue} simplifies and one recovers the following result, which has already been mentioned in Section \ref{F1_when_F0_has_double_eigenvalue}.

\begin{lemma}[First order Floquet exponent without constant modulation]
    With the assumptions of Theorem \ref{Inductive_Identity_for_Lyapunov-Floquet_decomposition}, given a double constant order Floquet exponent $(F_0)_{ii} = (F_0)_{jj}$ for some  $i \not= j$, and supposing that the constant component $A_1^{(0)}$ of the first order modulation $A_1$ is zero, it holds that the Floquet exponents corresponding to $(F_0)_{ii}$ are given by
        \begin{align*}
            (F_0)_{ii} \pm \epsilon\sqrt{A_1^{(n_i - n_j)})_{ij}(A_1^{(n_j - n_i)})_{ji}} + O(\epsilon^2),
        \end{align*}
    where $\pm(n_i - n_j)$ are the resonant frequencies of the double constant order Floquet exponent $(F_0)_{ii}$.
\end{lemma}
Provided that the first order modulation has no constant component, this lemma particularly implies the following. A multiple constant order Floquet exponent will only be perturbed linearly when the first order modulation $A_1$ has non-zero frequency components corresponding to the resonant frequencies of the double constant order Floquet exponent. Particularly, if the resonant frequency components are zero, then no linear perturbation occurs and the next lemma applies.

\begin{lemma}[Quadratic perturbation of double Floquet exponent]
    With the assumptions of Theorem \ref{Inductive_Identity_for_Lyapunov-Floquet_decomposition} and given a double constant order Floquet exponent $(F_0)_{ii} = (F_0)_{jj}$ for some  $i \not= j$ and supposing that the constant component $A_1^{(0)}$ and the resonant frequency components $A_1^{(n_i - n_j)}, \, A_1^{(n_j - n_i)}$ of the first order modulation $A_1$ are zero, it holds that the Floquet exponents corresponding to $(F_0)_{ii}$ are not perturbed linearly, but their quadratic perturbations are given by the eigenvalues of\footnote{Here, it would be possible to express $F_1$ in terms of $A_0, \, A_1, \ldots$ and then obtain explicit expansions of the perturbation of $(F_0)_{ii}$ up to quadratic order. Nevertheless, such expressions are rather cumbersome and are therefore omitted here.}
    \begin{align*}
        \begin{pmatrix}
                    (F_1G_iF_1 + F_2)_{ii} &(F_1G_iF_1 + F_2)_{ij}\\
                    (F_1G_iF_1 + F_2)_{ji} &(F_1G_iF_1 + F_2)_{jj}
            \end{pmatrix},
    \end{align*}
        with $G_i$ being a diagonal matrix given by
        \begin{align*}
            (G_i)_{ll} = \begin{cases} 0 &\text{ if } (F_0)_{ll} = (F_0)_{ii}, \\
            \frac{1}{(F_0)_{ii} - (F_0)_{ll}} &\text{ if } (F_0)_{ll} \not= (F_0)_{ii}.
            \end{cases}
        \end{align*}
\end{lemma}

\begin{proof}
    This is a direct consequence of Corollary \ref{cor:eigenvalue_perturbation_of_double_eigenvalue}.
\end{proof}

\subsubsection{Exceptional points}\label{sec:exceptional_points}
    In this subsection, conditions on the asymptotic existence of exceptional points will be elaborated. First, the definition of exceptional points will be recalled. Then, their occurrence will be analyzed in the setting of periodically modulated ODEs with constant coefficients. To illustrate the results, the existence of asymptotic exceptional points will be addressed in the setting of the modulated classical harmonic oscillator.

    \begin{definition}[Exceptional point]\label{def:EP}
        Let $A(c,t)$ be a parametrized $T$-periodic matrix-valued function $A(c,\cdot): \mathbb{R} \rightarrow \mathbb{C}^{N\times N}$. An element $c_0$ of the parameter space is called an \emph{exceptional point of the parametrized differential equation}
        $$\frac{dx}{dt} = A(c,t)x$$
        if the fundamental solution $X$ of  ${dx}/{dt} = A(c_0,t)x$ has the property that $X(T)$ is \emph{not} diagonalizable.
    \end{definition}

    There are different interpretations of exceptional points of a parametrized differential equation. One interpretation is that an exceptional point is a point $c_0$ of the parameter space where the eigenvectors of the fundamental solution $X(T)$ ``become parallel''. What is meant by this is the following. Denote by $X_c(t)$ the fundamental solution of
        $$\frac{dx}{dt} = A(c,t)x$$
    and suppose that it is possible to parametrize a basis of eigenvectors $T_c$ of $X_c(t)$ continuously. Then, when approaching an exceptional point $c_0$ one observes that $\lim_{c\rightarrow c_0} T_c$ is singular; in other words, some eigenvectors become linearly dependent in this limit.

    Another interpretation of exceptional points goes as follows. Given an exceptional point and its associated fundamental solution $X(T)$, then Theorem \ref{Floquet-Bloch_theorem} implies that
        $$X(T) = \exp(FT)$$
    is not diagonalizable. Thus $F$ is not diagonalizable, but can be written in its Jordan decomposition
        $$F = QJQ^{-1}.$$
    Taking the exponential function of a Jordan block, for example of a $2\times2$-block,
        $$\exp\begin{pmatrix}
            \sigma &1 \\
            0 & \sigma
        \end{pmatrix} = \begin{pmatrix}
            \exp(\sigma) & \sigma \\
            0 &\exp(\sigma)
        \end{pmatrix},$$
    monomial terms on the upper side of the diagonal appear. That is, exceptional points are special choices of the parameter where the solutions of the differential equation at that choice of parameter exhibit polynomial behavior. At a non-exceptional point, the solutions to the differential equation will only exhibit periodic and exponential behavior. 

    \begin{lemma}[Condition for existence of exceptional points]\label{lemma:condition_for_exceptional_point}
        Let $A(c,t)$ be a parametrized $T$-periodic matrix-valued function $A(c,\cdot): \mathbb{R} \rightarrow \mathbb{C}^{N\times N}$, then $c_0$ is an exceptional point if and only if a Floquet exponent matrix $F$ associated to the differential equation
            $$\frac{dx}{dt} = A(c_0,t)x$$
        is not diagonalizable.
    \end{lemma}

    \begin{proof}
        Theorem \ref{Floquet-Bloch_theorem} insures that the fundamental solution to the differential equation 
        $$\frac{dx}{dt} = A(c_0,t)x$$
        can be decomposed as
        $$X(t) = P(t)\exp(Ft),$$
        where $P(t)$ is $T$-periodic and $F$ is a constant matrix. By definition, $c_0$ is an exceptional point if and only if $X(T)$ is not diagonalizable, which is precisely the case when $X(T) = P(T)\exp(FT) = \exp(FT)$ is not diagonalizable. This implies however that $F$ is not diagonalizable either. For the other implication, it suffices to remark that $\exp(J)$, where $J$ is a Jordan block of multiplicity greater than $1$, is also not diagonalizable. 
    \end{proof}

    In the setting of Section \ref{Chapter:Asymptotic_Floquet_theory}, it is rather difficult to obtain an equivalence similar to the above lemma. Applying directly Lemma \ref{lemma:condition_for_exceptional_point}, it follows that the system ${dx}/{dt} = A_\epsilon(t)x$ is posed at an exceptional point if and only if $F_\epsilon$ is not diagonalizable. However, $F_\epsilon$ is only defined via the limit $\sum_{n = 0}^\infty \epsilon^nF_n$ and thus this criterion turns out to be impractical. For the sake of employability, we introduce the following definition.

    \begin{definition}[$n$th order asymptotic exceptional point] \label{def:aEP}
        In the setting of Theorem \ref{Inductive_Identity_for_Lyapunov-Floquet_decomposition}, we say that the system ${dx}/{dt} = A_\epsilon(t)x$ is posed at an \emph{$n$th order asymptotic exceptional point} if 
            $$F_0 + \epsilon F_1 + \ldots + \epsilon^n F_n$$
        is not diagonalizable for all $\epsilon >0$.
    \end{definition}

	\begin{remark}
		In contrast to \Cref{def:EP}, the coefficient matrix is not parametrized by $c$ in the setting of \Cref{def:aEP}. Phrased differently, we think of $c=c_0$ as fixed, and study asymptotic exceptional points which arise when $\epsilon$ becomes nonzero. This setting is slightly different from the exceptional points observed in \cite{ammari2020time}, where the exceptional points $c=c_0(\epsilon)$ depend on $\epsilon$.
	\end{remark}

    The above definition is different from the definition of an exceptional point; the above definition handles the case where the system ${dx}/{dt} = A_\epsilon x$ is not diagonalizable \emph{to $n$th order}. However, this does not imply that $F_\epsilon$ is not diagonalizable. In fact, if $A$ is a non-diagonalizable matrix, then for almost every matrix $B$ and for almost every $\epsilon$ the characteristic polynomial $\det(A + \epsilon B - \lambda \Id)$ has distinct roots and thus $A + \epsilon B$ is diagonalizable for almost every matrix $B$ and almost every $\epsilon$.
    An answer to the question on how $n$th order asymptotic exceptional points and ``actual'' exceptional points relate in the setting of Theorem \ref{Inductive_Identity_for_Lyapunov-Floquet_decomposition} would be very useful and should be subject of future research.

    Although an $n$th order asymptotic exceptional point does not provide an exceptional point, it is nevertheless a very useful and applicable notion. Being able to construct first order exceptional points brings us one step closer to creating $n$th order asymptotic exceptional points and further to the creation of actual exceptional points. To this end the following theorem will turn out to be useful.

    % \begin{lemma}
    %     Assume the setting of Theorem \ref{Inductive_Identity_for_Lyapunov-Floquet_decomposition}. If the system $\frac{dx}{dt} = A_\epsilon(t)x$ is posed at an $n$th order asymptotic exceptional point, then it is also posed at an exceptional point. 
    % \end{lemma}

    % \begin{proof} Let $F_\epsilon = F_0 + \epsilon F_1 + \epsilon^2F_2 + \ldots$ be an associated Floquet exponent matrix. By assumption $F_0 + \epsilon F_1 + \ldots + \epsilon^n F_n$ is not diagonalizable, let $Q$ be the basis transformation which transfomes $F_0 + \epsilon F_1 + \ldots + \epsilon^n F_n$ in its Jordan normal form $J$. Then $Q^{-1}F_\epsilon Q$ is given by
    %     $$J + O(\epsilon^n).$$  
    % \end{proof}

    \begin{theorem}[Criterion for first order exceptional point]\label{thm:criterion_for_first_order_exceptional_point}
        Assume the setting of Theorem \ref{Inductive_Identity_for_Lyapunov-Floquet_decomposition}. Then, the system ${dx}/{dt} = A_\epsilon(t)x$ with no constant first order perturbation, i.e. with $A_1^{(0)}=0$, is posed at a first order exceptional point if and only if
        \begin{enumerate}[(i)]
            \item the constant order Floquet exponent matrix $F_0$ has a double eigenvalue $(F_0)_{ii} = (F_0)_{jj}$, with $i \not= j$,
            \item and the associated entries of the first order perturbation $A_1$ satisfy that \emph{either}\footnote{Here, and throughout this work, we use the phrase ``\emph{either ... or ...}'' to refer to exclusive or.}
                \begin{align*}
                    (A_1^{(n_i - n_j)})_{ij} = 0
                \end{align*}
                \emph{or}
                \begin{align*}
                    (A_1^{(n_j - n_i)})_{ji} = 0,
                \end{align*}
                where $n_i$ and $n_j$ denote the folding numbers of $(F_0)_{ii}$ and $(F_0)_{jj}$, respectively.%\footnote{That is, $\pm(n_i - n_j)$ are the resonant frequencies of the double constant order Floquet exponent $(F_0)_{ii} = (F_0)_{jj}$.}
        \end{enumerate} 
    \end{theorem}

    \begin{proof}
        Due to the expression for $F_0 + \epsilon F_1 + O(\epsilon^2)$ in Lemma \ref{lemma:linear_perturbation_of_double_eigenvalue} and the fact that $A_1^{(0)}= 0$, it holds
        \begin{align*}F_0 + \epsilon F_1 = 
            \begin{pmatrix}
                    (F_0)_{ii} \\
                    & (F_0)_{jj}
            \end{pmatrix}
            + \epsilon\begin{pmatrix}
                    0 & (A_1^{(n_i - n_j)})_{ij} \\
                    (A_1^{(n_j - n_i)})_{ji} & 0
            \end{pmatrix},
        \end{align*}
        and thus $F_0 + \epsilon F_1$ is non-diagonalizable if and only if either $(A_1^{(n_i - n_j)})_{ij} = 0$ or $(A_1^{(n_j - n_i)})_{ji} = 0,$ and the desired result follows.
    \end{proof}

    To illustrate the above statement and to indicate on how it can be used to create first order exceptional points, the following section is devoted to the classical harmonic oscillator and the possibility for asymptotic exceptional points to occur. 

    \subsubsection{Exceptional points in the harmonic oscillator setting}
    First order exceptional points can be created via the following steps. Assume that our system is subject to a parametrization in $c$ and a perturbation in $\epsilon$, that is, we are given a system of the form 
        $$\frac{dx}{dt} = A_\epsilon(c,t)x.$$
    Suppose that the system depends continuously on $c$, satisfies the conditions of Theorem \ref{Inductive_Identity_for_Lyapunov-Floquet_decomposition} and also satisfies $A_1^{(0)} = 0$. Then
        \begin{enumerate}[(i)]
            \item First, one needs to find choices of $c_0$ where one has $\Re((A_0(c_0))_{ii}) = \Re((A_0(c_0))_{jj})$ and where one has either $(A_1(c_0,t)^{(m)})_{ij} = 0$ or $(A_1(c_0,t)^{(-m)})_{ji} = 0$ for some distinct coordinate indices $i,j$ and some frequency $m$;
            \item Next, the modulation period $T$ needs to be chosen such that $(F_0)_{ii} = (F_0)_{jj}$ \emph{and} such that $n_i - n_j = \pm m$. That is, $T > 0$ is given by
                \begin{align}
                   \left|\frac{2\pi m}{(A_0(c_0))_{ii} - (A_0(c_0))_{jj} }\right| = T.
                \end{align}
        \end{enumerate}
    Following this scheme, the following results for the classical harmonic oscillator in the setting of Section \ref{Harmonic_oscillator} are obtained.
    First, let us recall the constant and first order coefficient matrices of the associated first order ODE. With the notation $\alpha = \sqrt{c^2-4 k}$, the constant order coefficient matrix is given by
            \begin{align*}
                A_0(c,k) &= \begin{pmatrix}-\frac{c+ \alpha}{2} \\
                &-\frac{ c - \alpha}{2}\end{pmatrix}.
            \end{align*}
        For $c,k \in [0,4]$ it holds that $$ \Re(-\frac{c+ \alpha}{2}) = \Re(-\frac{ c - \alpha}{2}) \mbox{ is always true iff } \alpha = \sqrt{c^2-4 k} \mbox{ is imaginary,}$$ that is, if and only if $k \geq \frac{c^2}{4}$. Thus, one only needs to check that for some frequency $m$ the off-diagonal values of $A_1^{(\pm m)}$, where $k \geq \frac{c^2}{4}$, fulfills the above conditions. To this end, the following are the different frequency components of $A_1(c,k,\phi)$, when $a \not= b$:
            \begin{align*}
                A^{(\pm a)}_1 &= \frac{1}{4}\begin{pmatrix}-1 - \frac{c}{\alpha} & -1 - \frac{c}{\alpha}\\
                -1 + \frac{c}{\alpha} & -1 + \frac{c}{\alpha} \end{pmatrix},\\
                A^{(\pm b)}_1 &= \frac{e^{\pm \i\phi}}{2\alpha}\begin{pmatrix}1 &\frac{(c + \alpha)^2}{ 4  k}\\
                \frac{(c - \alpha)^2}{ 4  k} &-1\end{pmatrix}.
            \end{align*}
        When $a$ and $b$ are equal,  the sum of the above coefficients gives the coefficients for the corresponding frequencies. In both cases, one can respectively deduce non-existence and existence of asymptotic exceptional points. It turns out that as long as the two modulation frequencies are distinct no creation of exceptional points is possible, whereas when the modulation frequencies $a$ and $b$ are equal, it is possible to choose the damping constant $c$, the spring constant $k$ and the phase shift $\phi$ of the spring constant modulation with respect to the damping constant modulation such that first order exceptional points can be achieved. 

    \begin{lemma}[No exceptional point with distinct modulation frequencies]
        Assume the setting of the classical harmonic oscillator of Section \ref{Harmonic_oscillator}. If the modulation frequencies $a$ and $b$ are distinct, then the associated system
            \begin{align*}
                \frac{dx}{dt} = A_\epsilon(c,k,\phi)x
            \end{align*}
        with $c,k > 0$ is never posed at a first order exceptional point.
    \end{lemma}
    \begin{proof}
        In order for the above system to be posed at an exceptional point, one possibility is to have either 
        $$-1 - \frac{c}{\alpha} = 0$$
        or 
        $$ -1 + \frac{c}{\alpha} = 0.$$
        The other possibility is to have either 
        $$\frac{e^{\pm \i\phi}}{2\alpha}\frac{(c + \alpha)^2}{ 4  k} = 0$$
        or 
        $$\frac{e^{\pm \i\phi}}{2\alpha}\frac{(c - \alpha)^2}{ 4  k} = 0.$$
        Both conditions require $k=0$, which is outside of the parameter space and thus not achievable.\footnote{In the second case, the coefficients would even not be well-defined due to the parameter $k$ in the denominator.}
    \end{proof}

    When one chooses the same modulation frequencies for the modulation of the damping and the spring constant, then one needs to have either 
        \begin{align*}
            (A^{(-b = -a)}_1)_{12} &= \frac{e^{-\i\phi}}{2\alpha}\frac{(c + \alpha)^2}{ 4  k} + \frac{1}{4}(-1 - \frac{c}{\alpha})= 0
        \end{align*}
        or
        \begin{align*}
            (A^{(b = a)}_1)_{21} &= \frac{e^{\i\phi}}{2\alpha}\frac{(c - \alpha)^2}{ 4  k} + \frac{1}{4}(-1 + \frac{c}{\alpha}) = 0.
        \end{align*}
    This leads to the following equations
        \begin{align*}
            \frac{e^{-\i\phi}}{2}(c + \alpha)^2 -k\alpha - kc= 0
        \end{align*}
        and
        \begin{align*}
            \frac{e^{\i\phi}}{2}(c - \alpha)^2 -k\alpha+ kc= 0,
        \end{align*}
        respectively. Equating everything in terms of $\exp(\pm \i\phi)$, one obtains
        \begin{align*}
            e^{-\i\phi} = 2\frac{k\alpha + kc}{(c + \alpha)^2}
        \quad \mbox{ and
 } \quad     e^{\i\phi} = 2\frac{k\alpha - kc}{(c - \alpha)^2},
        \end{align*}
        respectively, which reduces to
        \begin{align*}
            e^{-\i\phi} = \frac{2k}{\alpha + c}
         \quad \mbox{ and
 } \quad  
            e^{\i\phi} = \frac{2k}{\alpha - c},
        \end{align*}
        respectively.
        Thus one needs to ask the question whether one of the right-hand sides attains values on the unit circle in $\mathbb{C}$ and if so, one needs to insure that at those points one has
        \begin{align*}
            \frac{2k}{\alpha + c} \not = \frac{\alpha - c}{2k},
        \end{align*}
        in order to have either $(A^{(-b = -a)}_1)_{12}=0$ or $(A^{(b = a)}_1)_{21}=0$. 

        \begin{lemma}[Exceptional points for the classical harmonic oscillator]
            Assume the setting of the classical harmonic oscillator of Section \ref{Harmonic_oscillator}. If the modulation frequencies $a$ and $b$ are equal, then the associated system
            \begin{align*}
                \frac{dx}{dt} = A_\epsilon(c,k,\phi)x
            \end{align*}
        with $c,k > 0$ is posed at a first order exceptional point if and only if
            \begin{align*}
                k = 1, \, c \in (0,2) \text{ and } \phi \in \left\{\i\log\left(\frac{2k}{\alpha + c}\right), -\i\log\left(\frac{2k}{\alpha - c}\right)\right\}.
            \end{align*}.
        \end{lemma}

        \begin{proof} By the reasoning above, at an exceptional point one needs to have $$\left|\frac{2k}{\alpha + c}\right|=1 \text{ or } \left|\frac{2k}{\alpha - c}\right|=1$$ and $$\frac{2k}{\alpha + c} \not= \frac{\alpha - c}{2k},$$
            where $k,c \in [0,4]$ and $k > c^2/4$, in order to have $\Re((A_0)_{11}) = \Re((A_0)_{22})$.

            Since $k,c \in [0,4]$ and $k > c^2/4$, the first identity $\abs{\frac{2k}{\alpha + c}}=1$ is equivalent to $\abs{c^2 - 4k} + c^2 = 4k^2$. Using $k > c^2/4$ again, it follows that $4k = 4k^2$ and thus $k = 1$. By the same reasoning, it follows that the second identity is also equivalent to $k = 1$.
            
            Hence, one obtains that at an exceptional point, it holds
            \begin{align*}
                k = 1, \, c \in [0,2) \text{ and } \phi \in \left\{\i\log\left(\frac{2k}{\alpha + c}\right), -\i\log\left(\frac{2k}{\alpha - c}\right)\right\}.
            \end{align*}

            On the contrary, given $k = 1, \, c \in [0,2) \text{ and } \phi \in \left\{\i\log\left(\frac{2k}{\alpha + c}\right), -\i\log\left(\frac{2k}{\alpha - c}\right)\right\}$, it holds that $\abs{\frac{2k}{\alpha + c}}=1 \text{ and } \abs{\frac{2k}{\alpha - c}}=1$. Furthermore, $\frac{2k}{\alpha + c} \not= \frac{\alpha - c}{2k},$ since this inequality is equivalent to 
                $$ 4k^2 \not= \abs{c^2 - 4k} - c^2,$$
                which reduces to 
                $$ k^2 -k + \frac{c^2}{2} \not= 0,$$
            and is thus satisfied whenever $k \not= \frac{1}{2} \pm \sqrt{\frac{1}{4} - \frac{c^2}{2}}$. Hence $c \not= 0$ and it follows that the system is posed at an exceptional point whenever $k=1$ and $c \in (0,2)$ and the phase shift $\phi$ is chosen accordingly. 
        \end{proof}

%% file: chapters/Application.tex
\section{Application: Floquet metamaterials}\label{ch:metamaterials}

Until now, the asymptotic Floquet exponent theory was only applied to the very simple example of the classical harmonic oscillator. In this section, it will be used to investigate Floquet metamaterials. In fact, this was our initial motivation for studying the asymptotics of Floquet exponents.

\subsection{Setting}
A Floquet metamaterial is a prototype material of the following form. It is composed of two materials: the \emph{background} and the \emph{resonator material}. While the background material fills almost the whole space $\mathbb{R}^d$, the resonator material only occupies disconnected domains $D_1, \ldots, D_N  \subset \mathbb{R}^d$. Those can be either repeated periodically in which case one speaks of an \emph{infinite periodic} structure, or they are not repeated periodically, but the resonator material is precisely given by ${D} = D_1 \cup \ldots \cup D_N$, in which case one speaks of a \emph{finite} structure. For simplicity, we consider only the finite case. The analysis can be easily extended to infinite periodic structures. Background and resonator material are characterized by their corresponding material parameters $\rho$ and $\kappa$, which correspond to the density and the bulk modulus in the setting of acoustic waves. To be more precise, the density $\rho$ and the bulk modulus $\kappa$ are defined as 
\begin{equation} \label{eq:resonatormod}
	\kappa(x,t) = \begin{cases}
		\kappa_0, & x \in \R^d \setminus \overline{D}, \\  \kappa_r \kappa_i(t), & x\in D_i,  i = 1,\ldots,N,
	\end{cases} \qquad \rho(x,t) = \begin{cases}
		\rho_0, & x \in \R^d \setminus \overline{D}, \\  \rho_r \rho_i(t), & x\in D_i,  i = 1,\ldots,N, \end{cases}
\end{equation}
where $\rho_0, \kappa_0, \rho_r, \kappa_r$ are positive constants and $\rho_i(t)$ and $\kappa_i(t)$ are $T$-periodic functions in $t$.  
That is, the density $\rho$ and the bulk modulus $\kappa$ are piecewise constant in space and also in time inside the background material. However, the material parameters are time-modulated inside the resonators.
The goal is to study  the \emph{subwavelength quasifrequencies} associated with the  wave equation 
\begin{equation}\label{eq:wave}
	\left(\frac{\p }{\p t } \frac{1}{\kappa(x,t)} \frac{\p}{\p t} - \nabla \cdot \frac{1}{\rho(x,t)} \nabla\right) u(x,t) = 0, \quad x\in \R^d, t\in \R.
\end{equation}

When the parameters of the wave equation \eqref{eq:wave} are periodic in time, one can apply the Floquet transform. This leads to a parametrized set of problems with restricted solution spaces. Indeed, one obtains
\begin{equation} \label{eq:wave_transf}
	\begin{cases}\ \ds \left(\frac{\p }{\p t } \frac{1}{\kappa(x,t)} \frac{\p}{\p t} - \nabla \cdot \frac{1}{\rho(x,t)} \nabla\right) u(x,t) = 0,\\[0.3em]
		\	u(x,t)e^{-\i \omega t} \text{ is $T$-periodic in $t$}, 
	\end{cases}
\end{equation}
where $\omega$ ranges over the elements of the Brillouin zone defined as before, $Y_t^* := \mathbb{C} / (\Omega \Z)$, with $\Omega$ being the frequency of the material parameters; $\Omega = {2\pi}/{T}$. If a (nontrivial) solution $u(x,t)$ to \eqref{eq:wave_transf} exists for an $\omega \in Y_t^*$, then $u(x,t)$ is called a \emph{Bloch solution} and $\omega$ its associated \emph{quasifrequency}.

In order to achieve \emph{subwavelength} quasifrequencies, one needs to assume that the \emph{contrast parameter} 
$$\delta := \frac{\rho_r}{\rho_0}$$
is small and consider the solutions to \eqref{eq:wave} as $\delta \rightarrow 0$.
One can regard \eqref{eq:wave} as parametrized with the contrast parameter $\delta$ and one can consider its solutions as $\delta \rightarrow 0$. This is called the \emph{high contrast regime}. In the setting where the frequency $\Omega$ is of order 
$O(\delta^{1/2})$, subwavelength frequencies are introduced as in \cite{ammari2020time} and are given by the following definition.
\begin{definition}[Subwavelength quasifrequency] \label{def:sub}
	A quasifrequency $\omega = \omega(\delta) \in Y^*_t$ of \eqref{eq:wave_transf} is said to be a \emph{subwavelength quasifrequency} if there is a corresponding Bloch solution $u(x,t)$, depending continuously on $\delta$, which can be written as
	$$u(x,t)= e^{\i \omega t}\sum_{n = -\infty}^\infty v_n(x)e^{\i n\Omega t},$$
	where 
	$$\omega(\delta) \rightarrow 0 \in Y_t^* \text{ and }  M(\delta)\Omega(\delta) \rightarrow 0  \in \mathbb{R} \text{ as }  \delta \to 0,$$
	for some integer-valued function $M=M(\delta)$ such that, as $\delta \to 0$, we have
	$$\sum_{n = -\infty}^\infty \|v_n\|_{L^2(D)} = \sum_{n = -M}^M \|v_n\|_{L^2(D)} + o(1).$$
\end{definition}

In \cite{ammari2020time}, a capacitance matrix approximation to the subwavelength quasifrequencies as $\delta \rightarrow 0$ was proven. In order to state the main result of \cite{ammari2020time}, we need the following definitions of the \emph{time-dependent contrast parameters}, \emph{wave speed} and \emph{time-dependent wave speeds}
$$\delta_i(t) = \frac{\rho_r \rho_i(t)}{\rho_0}, \quad v_0 = \sqrt{\frac{\kappa_0}{\rho_0}}, \quad v_i(t) = \sqrt{\frac{\kappa_r \kappa_i(t)}{\rho_r \rho_i(t)}},$$
respectively.

The \emph{capacitance matrix} is a way to encode the geometry of the metamaterial with a finite structure and also with an  infinite periodic structure into a square matrix. This matrix has the same dimension as the total number of resonators (or, in the case of a periodic structure, the total number of resonators inside a fundamental domain). The capacitance matrix theory in the high contrast regime of metamaterials is reviewed in \cite{ammari2021functional}. It is derived using \emph{Gohberg-Sigal theory} and \emph{layer potential techniques} for the finite structure case. Then, \emph{Floquet-Bloch theory} allows to extend the results to metamaterials with an infinite structure. 

For the purpose of this paper, we omit the precise definition of the capacitance matrix but observe that this is a Hermitian matrix of dimension $N$.

\begin{theorem} \label{thm:EriksMasterEquation}
    In dimension $d=3$ and being in the high contrast regime of \eqref{eq:wave}, assume that the material parameters $\kappa$ and $\rho$ are given by \eqref{eq:resonatormod} and that they satisfy 
    $$\frac{1}{\rho_i(t)} = \sum_{n = -M}^M r_{i,n} e^{\i n \Omega t}, \qquad \frac{1}{\kappa_i(t)} = \sum_{n = -M}^M k_{i,n} e^{\i n \Omega t},$$
    for some $M(\delta)\in \N$ satisfying $M(\delta) = O\left(\delta^{-\gamma/2}\right)$ as $\delta \rightarrow 0$ for some fixed $0<\gamma<1$. Furthermore, suppose that the associated time-dependent contrast parameters, wave speed and time-dependent wave speeds satisfy for all $t\in \R$ and $i=1,\ldots,N$,
    $$\delta_i(t) = O(\delta), \quad v = O(1), \quad v_i(t) = O(1) \quad \text{as } \hspace{0.2cm} \delta \rightarrow 0,$$
    respectively.
    Then, as $\delta \to 0$, the subwavelength quasifrequencies $\omega(\delta) \in Y^*_t$ to the wave equation \eqref{eq:wave} in the high contrast regime are, to leading order, given by the quasifrequencies of the system of ordinary differential equations in $y(t) = (y_i(t))_i$,
	\begin{equation}\label{eq:C_ODE}
		\sum_{j=1}^N C_{ij} y_j(t) = -\frac{|D_i|}{\delta_i(t)}\frac{d}{d t}\left(\frac{1}{\delta_iv_i^2}\frac{d (\delta_iy_i)}{d t}\right),
	\end{equation}
	for $i=1,\ldots,N$, where $C$ denotes the capacitance matrix associated to the finite structure (or the infinite periodic structure) of the considered metamaterial.
\end{theorem}

One can rewrite \eqref{eq:C_ODE} into the following system of Hill equations
\begin{equation}\label{eq:hill}
	\Psi''(t)+ M(t)\Psi(t)=0,
\end{equation}
where the vector $\Psi$ and the matrix $M$ are defined as
$$\Psi(t) = \left(\frac{\sqrt{\delta_i(t)}}{v_i(t)}y_i(t)\right)_{i=1}^N, \quad M(t) = W_1(t)C W_2(t) + W_3(t),$$
with $W_1, W_2$ and $W_3$ being diagonal matrices with corresponding diagonal entries
$$\left(W_1\right)_{ii} = \frac{v_i\delta_i^{3/2}}{|D_i|}, \qquad \left(W_2\right)_{ii} =\frac{v_i}{\sqrt{\delta_i}}, \qquad \left(W_3\right)_{ii} = \frac{\sqrt{\delta_i}v_i}{2}\frac{d }{d t}\frac{1}{(\delta_iv_i^2)^{3/2}}\frac{d (\delta_iv_i^2)}{d t},$$
for $i=1,\ldots,N$.

The asymptotic Floquet theory developed in previous sections will be applied to \eqref{eq:hill} in this section. In order to provide closed form solutions, we will treat the case of two equally sized resonators (either in the whole structure or in the fundamental domain).

\subsection{Dimer of time-modulated subwavelength resonators}\label{sec:2ResCase}

This section is devoted to applying the theory of asymptotic Floquet exponents to investigate subwavelength resonances in the setting of a dimer (i.e., a system of two identical resonators) of subwavelength resonators. In this setting, either precisely two resonators $D_1$ and $D_2$ are present in a background material or they are embedded into a fundamental domain which is repeated periodically. As pointed out before,  in both settings a capacitance matrix formulation exists. In the following, we will thus not differ whether we are in the finite or the infinite periodic setting and call the corresponding capacitance matrix simply $C$.

The goal of this section is thus to apply the theory on asymptotic Floquet exponents to the setting of the system of Hill equations given by \eqref{eq:hill}.

In order to apply the theory of asymptotic Floquet exponents, we will assume that the number of resonators is equal to two, that they are equal in size, i.e. that $\abs{D_1} = \abs{D_2}$, where $\abs{\cdot}$ denotes the volume,  and that the material parameters are given by
    \begin{align*}
        \frac{1}{\kappa_i} = 1 + \epsilon \gamma_i,\quad 
        \frac{1}{\rho_i} = 1 + \epsilon \eta_i,
    \end{align*}
where $\gamma_i$ and $\eta_i$ have finite Fourier series for $i=1,2$. Moreover, for simplicity we assume that $v_r := \sqrt{\kappa_r/\rho_0} = 1$.

Transforming \eqref{eq:hill} into a first order ODE and diagonalizing the leading order coefficient in $\epsilon$, one obtains the coefficient matrices
    \begin{align*}
        \frac{d\Phi}{dt} = (A_0 + \epsilon A_1 + \epsilon^2A_2 + O(\epsilon^3))\Phi.
    \end{align*}
The leading order diagonalized coefficient matrix is given by
    \begin{align}\label{eq:A0_2Res}
        A_0 = \frac{\sqrt{\delta}\i}{\sqrt{2|D_1|}}\begin{pmatrix}
            -\sqrt{C_{11} + C_{22} + \alpha}\\
            &  \sqrt{C_{11} + C_{22} + \alpha}\\
            && -\sqrt{C_{11} + C_{22} - \alpha}\\
            &&&\sqrt{C_{11} + C_{22} - \alpha}
            \end{pmatrix},
    \end{align}
    where $$\alpha = \sqrt{(C_{11} - C_{22})^2 + 4\abs{C_{12}}^2} \quad \mbox{ and
}    \quad
        C = \begin{pmatrix}
            C_{11} &C_{12}\\
            \overline{C_{12}} &C_{22}
        \end{pmatrix}
    $$
    denotes the associated capacitance matrix; see, for instance, \cite{ammari2021functional}.

\subsubsection{Coefficient matrices for the case of modulated densities}

When only the densities are modulated and the bulk moduli are constant,  the first and second order coefficient matrices turn out to be the following. The first order coefficient matrix $A_1$ is given by
\begin{align}\label{eq:A1_rho_modulation}
    A_1 = \frac{\i\sqrt{\delta} (\eta_1-\eta_2)}{2 \sqrt{2 |D_1|}}\begin{pmatrix}
    0 & 0 & \ds -\frac{C_{11}-C_{22}- \alpha}{ \sqrt{C_{11}+C_{22}-\alpha}} & \ds \frac{C_{11}-C_{22}- \alpha}{ \sqrt{C_{11}+C_{22}-\alpha}} \\
    \nm
    0 & 0 & \ds -\frac{C_{11}-C_{22}- \alpha}{ \sqrt{C_{11}+C_{22}-\alpha}} & \ds \frac{C_{11}-C_{22}- \alpha}{ \sqrt{C_{11}+C_{22}-\alpha}} \\
    \nm \ds
    -\frac{C_{11}-C_{22}+\alpha}{ \sqrt{ C_{11}+C_{22}+\alpha}} & \ds \frac{C_{11}-C_{22}+\alpha}{ \sqrt{ C_{11}+C_{22}+\alpha}} & 0 & 0 \\
    \nm \ds
    -\frac{C_{11}-C_{22}+\alpha}{ \sqrt{ C_{11}+C_{22}+\alpha}} & \ds \frac{C_{11}-C_{22}+\alpha}{ \sqrt{ C_{11}+C_{22}+\alpha}} & 0 & 0 
    \end{pmatrix}
\end{align}
and the second order coefficient matrix $A_2$ is given by
\begin{align*}
    A_2 = \begin{pmatrix}
        A_2^{11} & A_2^{12} \\ A_2^{21} & A_2^{22}
    \end{pmatrix},
\end{align*}
where $A_2^{11}, A_2^{12}, A_2^{21}, A_2^{22}$ are $2$-by-$2$ block matrices defined by
\begin{align*}
A_2^{11} &= \frac{\i\sqrt{\delta}}{\sqrt{2 |D_1|}}\frac{\abs{C_{12}}^2}{\alpha\sqrt{C_{11}+C_{22}+\alpha}}(\eta_1-\eta_2)^2\begin{pmatrix}
    -1 &
             1 \\ %
             -1 &
             1 \\ %
\end{pmatrix},\\
A_2^{22} &= \frac{\i\sqrt{\delta}}{\sqrt{2 |D_1|}}\frac{\abs{C_{12}}^2}{\alpha\sqrt{C_{11}+C_{22}-\alpha}}(\eta_1-\eta_2)^2\begin{pmatrix}
    1 &
             -1 \\
             1 &
             -1 \\
\end{pmatrix},\\
A_2^{12} &= \frac{\i\sqrt{\delta}}{\sqrt{2 |D_1|} }\left(\frac{\abs{C_{12}}^2}{\alpha\sqrt{ C_{11}+C_{22}-\alpha}}(\eta_1^2-\eta_2^2) + \frac{(C_{11}-C_{22})(C_{11}-C_{22}-\alpha)}{ 2\alpha\sqrt{ C_{11}+C_{22}-\alpha}}(\eta_1-\eta_2)\eta_2\right)\begin{pmatrix}
    -1 &
             1 \\
             -1 &
             1 \\
\end{pmatrix},\\
A_2^{21} &= \frac{\i\sqrt{\delta}}{\sqrt{2 } |D_1|}\left(\frac{\abs{C_{12}}^2}{\alpha\sqrt{ C_{11}+C_{22}+\alpha}}(\eta_1^2-\eta_2^2) + \frac{(C_{11}-C_{22})(C_{11}-C_{22}+\alpha)}{ 2\alpha\sqrt{ C_{11}+C_{22}+\alpha}}(\eta_1-\eta_2)\eta_2\right)\begin{pmatrix}
    1 &
             -1 \\ %
             1 &
             -1 \\ %
\end{pmatrix}.
\end{align*}

\subsubsection{Coefficient matrices for the case of modulated bulk moduli}
The situation is more intricate when instead of modulating the densities, the bulk moduli are modulated. Here, only the first order coefficient matrix $A_1$ is presented and the second order coefficient matrix $A_2$ is omitted. The first order coefficient matrix $A_1$ is given by

\begin{align}\label{eq:A1+B1_kappa_modulation}
    A_1 = \begin{pmatrix}
        A_{11} & A_{12}\\
        A_{21} & A_{22}
    \end{pmatrix} + 
    \begin{pmatrix}
        B_{11} & B_{12}\\
        B_{21} & B_{22}
    \end{pmatrix},
\end{align}
where $A_{11}, A_{12}, A_{21}, A_{22}$ are given by
\begin{align}\label{eq:A1_kappa_modulation}
    A_{11} &=  &\frac{\i\sqrt{\delta}}{4 \sqrt{2 |D_1| } }\frac{\sqrt{ C_{11}+C_{22}+\alpha}}{\alpha} &\left(\gamma_1(\alpha+C_{11}-C_{22}) + \gamma_2(\alpha-C_{11}+C_{22})\right)\begin{pmatrix} 1 &
    -1 \\
   1 &  
    -1 \\  
\end{pmatrix},\\
    A_{22} &= &\frac{\i\sqrt{\delta}}{4 \sqrt{2 |D_1| } }\frac{\sqrt{C_{11}+C_{22}-\alpha}}{\alpha} &\left(\gamma_2(\alpha+C_{11}-C_{22}) + \gamma_1(\alpha-C_{11}+C_{22})\right)\begin{pmatrix}
        1 &
        -1 \\      
        1 &      
        -1 \\      
    \end{pmatrix},\\
    A_{12} &=   &\frac{\i\sqrt{\delta}}{4\sqrt{2|D_1|}} \frac{\alpha-(C_{11}-C_{22})}{\alpha \sqrt{C_{11}+C_{22}-\alpha}}&
    (C_{11}+C_{22})(\gamma_1-\gamma_2)\begin{pmatrix}
        -1  &
        1 \\  
        -1  &  
        1 \\  
    \end{pmatrix},\\
    A_{21} &= &\frac{\i\sqrt{\delta}}{ 4 \sqrt{2 |D_1| } }\frac{\alpha+C_{11}-C_{22}}{ \alpha \sqrt{\left(\alpha+C_{11}+C_{22}\right)}}&
    (C_{11}+C_{22})(\gamma_1-\gamma_2)\begin{pmatrix}
        -1 &
         1 \\%&
        -1 &
         1 \\%&
    \end{pmatrix} 
\end{align}
and $B_{11}, B_{12}, B_{21}, B_{22}$ correspond to the summand $W_3$ and are given by

\begin{equation}\label{eq:B1_kappa_modulation}
\begin{aligned}
    B_{11} &= \frac{\i}{8}\sqrt{\frac{|D_1|}{2\delta}} \frac{\sqrt{C_{11}+C_{22}+\alpha}}{\alpha(C_{11} C_{22}-\abs{C_{12}}^2)}\Big[(C_{22} (\alpha + C_{11}-C_{22})-2 \abs{C_{12}}^2) \frac{d^2\gamma_1}{dt^2}(t)\begin{pmatrix}
        1 & -1 \\
        1 & -1 
    \end{pmatrix} \\
    &+(C_{11} (\alpha-(C_{11}-C_{22}))-2 \abs{C_{12}}^2) \frac{d^2\gamma_2}{dt^2}(t)    
    \begin{pmatrix}
        1 & -1 \\%&       
        1 & -1        
    \end{pmatrix}\Big],\\
    B_{22} &= \frac{\i}{8}\sqrt{\frac{|D_1|}{2\delta}} \frac{\sqrt{C_{11}+C_{22}-\alpha}}{(C_{11}C_{22} - \abs{C_{12}}^2)\alpha} \Big[(C_{22}(\alpha-(C_{11}-C_{22}))+2 \abs{C_{12}}^2)\frac{d^2\gamma_1}{dt^2}(t)    \begin{pmatrix}
        1 & -1 \\
        1 & -1 
    \end{pmatrix}\\
    &+(C_{11} (\alpha+C_{11}-C_{22})+2 \abs{C_{12}}^2) \frac{d^2\gamma_2}{dt^2}(t)
    \begin{pmatrix}
        1 & -1 \\
        1 & -1 
    \end{pmatrix}\Big],\\
    B_{12} &=\frac{\i}{8}\sqrt{\frac{|D_1|}{2\delta}} \frac{C_{11}-C_{22}-\alpha}{\alpha\sqrt{C_{11}+C_{22}-\alpha}} \left(\frac{d^2\gamma_1}{dt^2}(t)-\frac{d^2\gamma_2}{dt^2}(t)\right)\begin{pmatrix}
        1 & -1 \\
        1 & -1 
    \end{pmatrix},\\
    B_{21} &= -\frac{\i}{8}\sqrt{\frac{|D_1|}{2\delta}} \frac{C_{11}-C_{22}+\alpha}{\alpha\sqrt{C_{11}+C_{22}+\alpha}}\left(\frac{d^2\gamma_1}{dt^2}(t)-\frac{d^2\gamma_2}{dt^2}(t)\right)\begin{pmatrix}
        1 & -1 \\%&
        1 & -1 
    \end{pmatrix}. 
\end{aligned}
\end{equation}

\subsection{Floquet exponents and their perturbations}
    Having computed the coefficient matrices for the case of two equally sized resonators, it is possible to apply the theory of asymptotic Floquet exponents of Section \ref{asymptotic_analysis_of_floquet_exponents}. This leads to the following results. 

        \begin{lemma}[When the perturbation is non-constant then $F_1$ is off-diagonal]
            In the setting of Theorem \ref{thm:EriksMasterEquation} and supposing that the number of resonators is equal to two and that they are equally sized, it holds that if the perturbations of $1/\kappa_i$ for $i=1,2$ have mean zero over one modulation period, then the first order Floquet exponent matrix $F_1$ is off-diagonal.            
        \end{lemma}
        Here, it is important to note that the first order Floquet exponent matrix will be off-diagonal even if $1/\rho_i$ for $i=1,2$ is perturbed by a constant.
        \begin{proof}
            Theorem \ref{First_order_Floquet_matrix_and_Lyapunov_transformation} states that the diagonal entries of $F_1$ are given by the diagonal entries of $A_1^{(0)}$. In the setting of Theorem \ref{thm:EriksMasterEquation} and assuming that the number of resonators is equal to two, it follows that diagonal entries of the first order coefficient matrix $A_1$ corresponding to the perturbation of $1/\rho_i$ for $i = 1,2$ are zero, as may be deduced from  \eqref{eq:A1_rho_modulation}. The first order coefficient matrix $A_1$ corresponding to the modulation of $1/\kappa_i$ is not always off-diagonal, but, using the notation of \eqref{eq:A1+B1_kappa_modulation}, its diagonal entries are given by
            \begin{align*}
                (A_{ii})_{jj} + (B_{ii})_{jj}.
            \end{align*}
            Considering \eqref{eq:A1_kappa_modulation}, it follows that the diagonal entries are linear in 
            \begin{align*}
                \gamma_1(\alpha+C_{11}-C_{22}) &+ \gamma_2(\alpha-C_{11}+C_{22})\\
                &\text{or }\\
                \gamma_2(\alpha+C_{11}-C_{22}) &+ \gamma_1(\alpha-C_{11}+C_{22}),
            \end{align*}
            and thus it follows that $A_{ii}$ has no constant component whenever $\gamma_i$ for $i=1,2$ has no constant component. Considering $B_{11}$ and $B_{22}$, it follows from  \eqref{eq:B1_kappa_modulation}, that the diagonal entries of $B_{11}$ and $B_{22}$ are linear in 
            \begin{align*}
                \left(C_{22} (\alpha + C_{11}-C_{22})-2 \abs{C_{12}}^2\right)\frac{d^2\gamma_1}{dt^2}(t)&+\left(C_{11} (\alpha-(C_{11}-C_{22}))-2 \abs{C_{12}}^2\right)\frac{d^2\gamma_2}{dt^2}(t) \\
                &\text{and }\\
                \left(C_{22}(\alpha-(C_{11}-C_{22}))+2 \abs{C_{12}}^2\right)\frac{d^2\gamma_1}{dt^2}(t)&+\left(C_{11} (\alpha+C_{11}-C_{22})+2 \abs{C_{12}}^2\right)\frac{d^2\gamma_2}{dt^2}(t),
            \end{align*}
            respectively.
            Since only the second order derivatives of $\gamma_i$ for $i=1,2$ appear, it follows that the constant component of the diagonal entries of $B_{11}$ and $B_{22}$ are always zero. Hence, the first order Floquet exponent matrix $F_1$ is off-diagonal.
        \end{proof}
        This result, Lemma \ref{lem:single_Floquet_exponents_are_perturbed_quadratically} and its proof immediately imply the following corollary. It states that generically, if the number of resonators is equal to two, if they are equally sized and if the assumptions of Theorem \ref{thm:EriksMasterEquation} are fulfilled, then the constant order Floquet exponents are perturbed quadratically in $\epsilon$.

        \begin{cor}[Simple Floquet exponents are perturbed quadratically]
            In the setting of Theorem \ref{thm:EriksMasterEquation} and supposing that the number of resonators is equal to two, it holds that if the perturbation of $1/\kappa_i$ has mean zero over one modulation period and if the constant order Floquet exponents are distinct, then the constant order Floquet exponents are perturbed quadratically in $\epsilon$.            
        \end{cor}

        In order to achieve linear perturbation of the constant order Floquet exponents, it is necessary to produce \emph{double} constant order Floquet exponents. This can be achieved, as already explained in Subsection \ref{F1_when_F0_has_double_eigenvalue}, via folding of the constant order Floquet exponents with particular choices of the modulation frequency $\Omega$. The following lemma classifies different possible double constant order Floquet exponents in the dimer setting of Theorem \ref{thm:EriksMasterEquation}.

        \begin{lemma}[Double constant order Floquet exponents]\label{lem:double_constant_order_Floquet_exponent}
            In the setting of Theorem \ref{thm:EriksMasterEquation}, supposing that the number of resonators is equal to two and that they are equally sized, and with the notation of  \eqref{eq:A0_2Res}, if
            \begin{equation} \label{fact5.5} \frac{\sqrt{C_{11} + C_{22} + \alpha}-\sqrt{C_{11} + C_{22} - \alpha}}{\sqrt{C_{11} + C_{22} + \alpha}+\sqrt{C_{11} + C_{22} - \alpha}} \not\in \mathbb{Q},\end{equation}
            then the following are \emph{exclusive} and are the only possible double constant order Floquet exponent constellations:
            \begin{enumerate}[(i)]
                \item The constant order Floquet exponents $(F_0)_{11}$ and $(F_0)_{22}$ coincide;
                \item The constant order Floquet exponents $(F_0)_{33}$ and $(F_0)_{44}$ coincide;
                \item  The constant order Floquet exponents $(F_0)_{11}$ and $(F_0)_{33}$ coincide and (automatically) the constant order Floquet exponents $(F_0)_{22}$ and $(F_0)_{44}$ also coincide. In this case, the resonant frequencies of both double constant order Floquet exponents are equal;
                \item  The constant order Floquet exponents $(F_0)_{11}$ and $(F_0)_{44}$ coincide and (automatically) the constant order Floquet exponents $(F_0)_{22}$ and $(F_0)_{33}$ also coincide. In this case, the resonant frequencies of both double constant order Floquet exponents are equal.
            \end{enumerate} 
        \end{lemma}

        \begin{proof}
            This can be read of the formula \eqref{eq:A0_2Res} for the constant order coefficient matrix $A_0$. Exclusiveness is due to the fact that \eqref{fact5.5} holds, in which case all diagonal entries of $A_0$ cannot be folded onto each other.
        \end{proof}

        The phenomenon of simultaneous band crossings can be observed in Figure \ref{fig:f0_with_f1_correction}. There the crossings of the first and second constant order Floquet exponents are isolated, as are the crossings of the third and fourth constant order Floquet exponents. However, the crossings of the first and third constant order Floquet exponents are simultaneous with the crossing of the second and fourth constant order Floquet exponents, as stated in Lemma \ref{lem:double_constant_order_Floquet_exponent}.

        % \begin{figure}
        %     \centering
        %     \includegraphics[width=0.5\textwidth]{images/f0_with_f1_correction.eps}
        %     \caption{One can see that bands 1 and 3, 2 and 4 cross simultaneaously, whereas bands 1 and 2 and bands 2 and 3 cross isolatedly}
        %     \label{fig:f0_with_f1_correction}
        % \end{figure}

        The following subsection elucidates some further behaviors of the constant order Floquet exponent perturbation in the setting where only $1/\rho_i$ for $i=1,2$ is modulated and $1/\kappa_i$ for $i=1,2$ is not modulated. There it will become apparent that certain double Floquet exponents cannot be perturbed linearly when $1/\kappa_i$ is not modulated. This is due to the zero blocks on the diagonal of $A_1$ associated to the modulation of $1/\rho_i$ for $i=1,2$; see  \eqref{eq:A1_rho_modulation}.
        This is different in the setting where $1/\kappa_i$ for $i=1,2$ is modulated. The associated first order coefficient matrix $A_1$ has no zero blocks and thus linear perturbation of all double Floquet exponents can be achieved through a modulation of $1/\kappa_i$ for $i =1,2$. This will be the subject of Subsection \ref{subsec:Floquet_exp_pert_kappa}.

    \subsubsection{Floquet exponent perturbation for modulated densities}
     In this subsection, only the densities are modulated (the bulk moduli are not). In the dimer setting, this leads to the following phenomenon.
            
            \begin{lemma}\label{lem:no_quad_pert_for_non-paired_FloExp_rho-setting}
                In the setting of Theorem \ref{thm:EriksMasterEquation}, supposing that the number of resonators is equal to two, that they are equally sized, and with the notation of  \eqref{eq:A0_2Res}, it holds that, when only $1/\rho_i$ with $i=1,2$ is modulated and $1/\kappa_i$ for $i=1,2$ is not modulated, then a double Floquet exponent of the form $(F_0)_{11} = (F_0)_{22}$ or of the form $(F_0)_{33} = (F_0)_{44}$ is never perturbed linearly, but at least quadratically.
            \end{lemma}
            \begin{proof}
                Recalling the structure of the first order coefficient matrix $A_1$ in the setting of modulated $1/\rho_i$ and non-modulated $1/\kappa_i$ from \eqref{eq:A1_rho_modulation}: the diagonal blocks of $A_1$ are zero, and thus $(A_1)_{12}=(A_1)_{21}=(A_1)_{34}=(A_1)_{43}=0$. By Lemma \ref{lemma:linear_perturbation_of_double_eigenvalue}, it holds that precisely those entries of $A_1$ determine the linear perturbation of a double constant order Floquet exponent of the form $(F_0)_{11} = (F_0)_{22}$ or of the form $(F_0)_{33} = (F_0)_{44}$. It thus follows that double constant order Floquet exponents of that particular form are never perturbed linearly, when $1/\rho_i$ for $i=1,2$ is modulated only.
            \end{proof}

            For the double constant order Floquet exponents of the form $(F_0)_{11}=(F_0)_{44}$  and $(F_0)_{22}=(F_0)_{33}$ or of the form $(F_0)_{11}=(F_0)_{33}$  and $(F_0)_{22}=(F_0)_{44}$, linear perturbation is achievable, since the corresponding entries of the first order coefficient matrix $A_1$ are non-zero, when $1/\rho_i$ for $i=1,2$ is modulated with the resonant frequencies of the corresponding double constant order Floquet exponent. In this setting, the same procedure for producing linear perturbation as proposed in Subsection \ref{F1_when_F0_has_double_eigenvalue} applies. Since in both settings the resonant frequencies of the two different double constant order Floquet exponents are equal (see Lemma \ref{lem:double_constant_order_Floquet_exponent}), it follows that both double constant order Floquet exponents are always perturbed linearly and simultaneously and linear perturbation of only one of the double constant order Floquet exponents cannot be achieved. For example, in the setting where $(F_0)_{11}=(F_0)_{44}$  and $(F_0)_{22}=(F_0)_{33}$, and if $(F_0)_{11}=(F_0)_{44}$ is perturbed linearly, then so is $(F_0)_{22}=(F_0)_{33}$, since $(F_0)_{11}=(F_0)_{44}$  and $(F_0)_{22}=(F_0)_{33}$ share the same resonant frequencies.
            This phenomenon can be observed in Figure \ref{fig:rho_modulation}, where $1/\rho_i$ is modulated and $1/\kappa_i$ is not modulated for $i=1,2$. There, one can see in Subfigure \ref{fig:f0_plus_eps2_f2_with_f1_correction} that the double Floquet exponent $(F_0)_{11}=(F_0)_{33}$  and $(F_0)_{22}=(F_0)_{44}$ are perturbed symmetrically and linearly. In contrast, due to the zero-blocks of $A_1$, all other double constant order Floquet exponents are perturbed quadratically.

                \begin{figure}
                    \centering
                        \includegraphics[width=\textwidth]{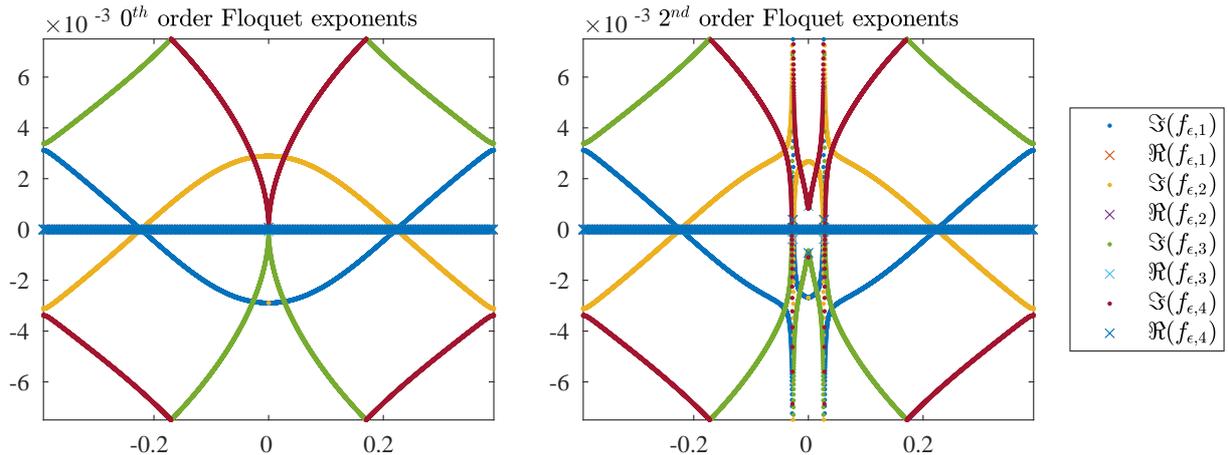}
                        \caption{Displayed are the constant order and second order Floquet exponents in the case of a dimer of subwavelength resonators. On the left figure, one can see that bands 1 and 3, 2 and 4 cross simultaneously, whereas bands 1 and 2 and bands 2 and 3 cross isolatedly. On the right figure, the Floquet exponents up to second order with $\epsilon=0.1$ are displayed. One can see the pole at the band crossing (1,2 and 3,4). This corresponds to a linear perturbation of the Floquet exponents. The pole is due to the non-uniformity of the asymptotic analysis.}
                        \label{fig:f0_with_f1_correction}
                        \label{fig:f0_plus_eps2_f2_with_f1_correction}
                       \label{fig:rho_modulation}
                \end{figure}

                % \begin{figure}
                %     \centering
                %     \includegraphics[width=0.5\textwidth]{images/Geometry_of_2Res_system.eps}
                %     \caption{Caption}
                %     \label{fig:Geometry_of_2Res_system}
                % \end{figure}
        
                % \begin{figure}
                %     \centering
                %     \includegraphics[width=\textwidth]{images/f0_f1_f2_with_f1_correction.eps}
                %     \caption{Caption}
                %     \label{fig:f0_f1_f2_with_f1_correction}
                % \end{figure}
        
                % \begin{figure}
                %     \centering
                %     \includegraphics[width=0.5\textwidth]{images/f0_plus_eps2_f2_with_f1_correction.eps}
                %     \caption{Caption}
                %     \label{fig:f0_plus_eps2_f2_with_f1_correction}
                % \end{figure}

    \subsubsection{Floquet exponent perturbation for modulated bulk moduli}\label{subsec:Floquet_exp_pert_kappa}
        While in the setting where only the densities were modulated, the behavior of the linear perturbation of double constant order Floquet exponents was rather restricted, the behavior when the bulk moduli are modulated is very wide due to the absence of zero entries in the corresponding first order coefficient matrix $A_1$. That is, when modulating $1/\kappa_i$ for $i = 1,2$ at the resonance frequencies of a double constant order Floquet exponent, linear perturbation can generically be achieved. When the double constant order Floquet exponents occur in pairs of the type $(F_0)_{11}=(F_0)_{33}$  and $(F_0)_{22}=(F_0)_{44}$ or of the type $(F_0)_{11}=(F_0)_{44}$  and $(F_0)_{22}=(F_0)_{33}$, the same behavior is present as in the case where only $1/\rho_i$ was modulated. Namely, both double constant order Floquet exponents cannot be perturbed linearly and separately, but their linear perturbation will always occur simultaneously. This also follows from Lemma \ref{conjugate}, which states that Floquet exponents of an ODE with real-valued coefficients always occur in conjugate pairs.

\subsection{Classification of exceptional points for a dimer of subwavelength resonators}
    This section is devoted to the classification of first order asymptotic exceptional points in the setting of Theorem \ref{thm:EriksMasterEquation} for a dimer of subwavelength resonators. In order to classify those exceptional points, the criterion stated in Theorem \ref{thm:criterion_for_first_order_exceptional_point} will be repeatedly applied, particularly its criterion on the first order coefficient matrix $A_1$. Furthermore, the same approach as in the subsection on exceptional points in the classical harmonic oscillator setting of Section \ref{sec:exceptional_points} will be applied.

    It will become apparent that in the setting of Lemma \ref{lem:double_constant_order_Floquet_exponent}, exceptional points through a modulation of either $\rho$ or $\kappa$ are only possible when the associated capacitance matrix $C$ is diagonal, that is, when $C_{12}=0$. The behavior is fundamentally different when both $\rho$ and $\kappa$ are modulated simultaneously. Indeed, the condition on the capacitance matrix is no longer necessary and an uncountable number of exceptional points are possible.
    
\subsubsection{Exceptional points when only the densities are modulated}
   
   When either the densities or the bulk moduli are modulated, then exceptional points are only possible, when the capacitance matrix $C$  is diagonal and has distinct diagonal entries. However, when the densities and the bulk moduli  are modulated simultaneously, (asymptotic) exceptional points without any conditions on $C$ can be created. That is, given any dimer of subwavelength resonators, it is possible to create asymptotic exceptional points.

    \begin{lemma}\label{lem:rho-modulation_exceptional_pt_C12}
        In the setting of Theorem \ref{thm:EriksMasterEquation} and Lemma \ref{lem:double_constant_order_Floquet_exponent}, suppose that only $1/\rho_i$ with $i=1,2$ is modulated and $1/\kappa_i$ for $i=1,2$ is not modulated. Then,  it holds that if the system ${d\Phi}/{dt} = A_\epsilon(t) \Phi$ is posed at a first order asymptotic exceptional point, then $C_{12}=0$.
    \end{lemma}

    \begin{proof} 
        Suppose that the system ${d\Phi}/{dt} = A_\epsilon(t) \Phi$ is posed at a first order asymptotic exceptional point.
        Due to the zero blocks on the diagonal of the first order coefficient matrix $A_1$ (equation \eqref{eq:A1_rho_modulation}), it follows that the exceptional point is due to a double constant order Floquet exponent of the form $(F_0)_{11}=(F_0)_{33}$  and $(F_0)_{22}=(F_0)_{44}$ or of the form $(F_0)_{11}=(F_0)_{44}$  and $(F_0)_{22}=(F_0)_{33}$. Thus, by Theorem \ref{thm:criterion_for_first_order_exceptional_point} at least one of the following conditions must hold:
        \begin{enumerate}[(i)]
            \item Either $\ds -\frac{C_{11}-C_{22}- \alpha}{ \sqrt{C_{11}+C_{22}-\alpha}} = 0$ or $\ds -\frac{C_{11}-C_{22}+\alpha}{ \sqrt{ C_{11}+C_{22}+\alpha}} = 0$;\label{list:cond1}
            \item Either $\,\,\,\,\,\ds \frac{C_{11}-C_{22}- \alpha}{ \sqrt{C_{11}+C_{22}-\alpha}} = 0$ or $\ds -\frac{C_{11}-C_{22}+\alpha}{ \sqrt{ C_{11}+C_{22}+\alpha}} = 0$;
            \label{list:cond2}
            \item Either $\ds -\frac{C_{11}-C_{22}- \alpha}{ \sqrt{C_{11}+C_{22}-\alpha}} = 0$ or $\,\,\,\,\,\ds \frac{C_{11}-C_{22}+\alpha}{ \sqrt{ C_{11}+C_{22}+\alpha}} = 0$;\label{list:cond3}
            \item Either $\,\,\,\,\,\ds \frac{C_{11}-C_{22}- \alpha}{ \sqrt{C_{11}+C_{22}-\alpha}} = 0$ or $\,\,\,\,\,\ds \frac{C_{11}-C_{22}+\alpha}{ \sqrt{ C_{11}+C_{22}+\alpha}} = 0$.\label{list:cond4}
        \end{enumerate}
        Conditions \eqref{list:cond1} to \eqref{list:cond4} are equivalent. They hold if and only if either $C_{11}-C_{22}- \alpha = 0 $ or $C_{11}-C_{22}+\alpha =0$. Since $C_{11}$ and $C_{22}$ are always real and recalling that $\alpha = \sqrt{(C_{11} - C_{22})^2 + 4\abs{C_{12}}^2}$, it follows that at least one of the conditions holds if and only if
            $$(C_{11}-C_{22})^2 =  (C_{11} - C_{22})^2 + 4\abs{C_{12}}^2,$$
        and thus one of the conditions holds if and only if $C_{12}=0$.
    \end{proof}

    \begin{lemma}\label{lem:rho_exceptional_pt_properties}
        In the setting of Theorem \ref{thm:EriksMasterEquation} and Lemma \ref{lem:double_constant_order_Floquet_exponent}, suppose that only $1/\rho_i$ with $i=1,2$ is modulated and that $1/\kappa_i$ for $i=1,2$ is not modulated. If the system ${d \Phi}/{dt} = A_\epsilon(t) \Phi$ is posed at a first order asymptotic exceptional point, then the following results hold:\begin{enumerate}[(i)]
            \item The capacitance matrix is given by $$C = \begin{pmatrix}
                C_{11} &0\\
                0 &C_{22}
            \end{pmatrix}.$$\label{cor:statement1}
            \item The constant order coefficient matrix $A_0$ takes the form: $$A_0 = \begin{cases}
                \i\sqrt{\frac{\delta}{|D_1|}}\left(\begin{smallmatrix}
                -\sqrt{C_{11}}\\
                &  \sqrt{C_{11}}\\
                && -\sqrt{C_{22}}\\
                &&&\sqrt{C_{22}}
                \end{smallmatrix}\right) &\text{if } C_{11} \geq C_{22},\\ 
                \nm
                \i\sqrt{\frac{\delta}{|D_1|}}\left(\begin{smallmatrix}
                    -\sqrt{C_{22}}\\
                    &  \sqrt{C_{22}}\\
                    && -\sqrt{C_{11}}\\
                    &&&\sqrt{C_{11}}
                    \end{smallmatrix}\right) &\text{otherwise}.\end{cases}$$ \label{cor:statement2}
            \end{enumerate}
    \end{lemma}

    \begin{proof} The first statement is a direct consequence of Lemma \ref{lem:rho-modulation_exceptional_pt_C12}. Using the formula for the constant order coefficient matrix $A_0$ given by  \eqref{eq:A0_2Res}, the second statement is then a consequence of the first one.
    \end{proof}

    \begin{lemma}
        In the setting of Theorem \ref{thm:EriksMasterEquation} and Lemma \ref{lem:double_constant_order_Floquet_exponent}, suppose that only $1/\rho_i$ with $i=1,2$ is modulated and that $1/\kappa_i$ for $i=1,2$ is not modulated. Then, if the system ${d\Phi}/{dt} = A_\epsilon(t) \Phi$ is posed at a first order asymptotic exceptional point, we have $C_{11} \neq C_{22}$.
    \end{lemma}

%Then, the following statements are equivalent:
%\begin{enumerate}[(i)]
%\item The system ${d\Phi}/{dt} = A_\epsilon(t) \Phi$ is posed at a first order asymptotic exceptional point.
%\item The system ${d\Phi}/{dt} = A_\epsilon(t) \Phi$ is posed at an \emph{artifical} first order asymptotic exceptional point.
%\end{enumerate}

    \begin{proof} 
    	Suppose on the contrary that $C_{11} = C_{22}$, in other words that the unmodulated system has degenerate Floquet exponents. However, in that case the first order coefficient matrix $A_1$ is identically zero by \eqref{eq:A1_rho_modulation} and thus the system cannot be posed at a first order exceptional point.
    \end{proof}

    \begin{remark} Strictly speaking the formula for $A_2$ does not apply in the case where $C_{11} = C_{22}$ and $C_{12}=0$, since then $\alpha = 0$ and it appears in the denominators of certain coefficients. Nevertheless, analogous formulas are straight-forward to compute in this case.
    \end{remark}

    \begin{lemma}\label{lem:except_pt_rho-setting_F0_cases}
        In the setting of Theorem \ref{thm:EriksMasterEquation} and Lemma \ref{lem:double_constant_order_Floquet_exponent}, suppose that only $1/\rho_i$ with $i=1,2$ is modulated and that $1/\kappa_i$ for $i=1,2$ is not modulated. If the system ${dx}/{dt} = A_\epsilon(t) x$ is posed at an exceptional point, then the constant order Floquet exponent matrix $F_0$ is either of the form $(F_0)_{11}=(F_0)_{44}$  and $(F_0)_{22}=(F_0)_{33}$ or of the form $(F_0)_{11}=(F_0)_{33}$  and $(F_0)_{22}=(F_0)_{44}$.
    \end{lemma}

    \begin{proof} At a first order asymptotic exceptional point the constant order Floquet exponent matrix needs to have a double eigenvalue.
        By Lemma \ref{lem:double_constant_order_Floquet_exponent}, the only cases are either $(F_0)_{11}=(F_0)_{22}$ or $(F_0)_{33}=(F_0)_{44}$ or the cases stated above. However, by Lemma \ref{lem:no_quad_pert_for_non-paired_FloExp_rho-setting} double constant order Floquet exponents of the form $(F_0)_{11}=(F_0)_{22}$ or $(F_0)_{33}=(F_0)_{44}$ cannot be perturbed linearly in the setting where only $\rho$ is modulated. Thus, they cannot lead to an exceptional point and at an exceptional point the Floquet exponent matrix always needs to be the form $(F_0)_{11}=(F_0)_{44}$  and $(F_0)_{22}=(F_0)_{33}$ or of the form $(F_0)_{11}=(F_0)_{33}$  and $(F_0)_{22}=(F_0)_{44}$.
    \end{proof}

    \begin{theorem}[Classification of first order asymptotic exceptional points when $\rho$ is modulated only]\label{thm:classification_1st_ord_exc_pt_rho-setting}
        In the setting of Theorem \ref{thm:EriksMasterEquation} and Lemma \ref{lem:double_constant_order_Floquet_exponent}, suppose that only $1/\rho_i$ with $i=1,2$ is modulated and that $1/\kappa_i$ for $i=1,2$ is not modulated. Then the following equivalent statements hold:
        \begin{enumerate}[(i)]
            \item The system ${dx}/{dt} = A_\epsilon(t) x$ is posed at a first order asymptotic exceptional point.\label{thm:statement_art_1}
            \item The capacitance matrix $C$ takes the form
            \begin{align*}
                C = \begin{pmatrix}
                    C_{11} &0\\
                    0 &C_{22}
                \end{pmatrix}
            \end{align*} with $C_{11} \not= C_{22}$ and one of the following cases applies:
            \begin{enumerate}[(a)]
                \item The constant order Floquet exponents take the form $(F_0)_{11} = (F_0)_{33}$ and $(F_0)_{22} = (F_0)_{44}$  
                    and the modulation frequency $\Omega$ is given by \begin{align}
                        \Omega = \frac{1}{n}\sqrt{\frac{\delta}{|D_1|}}\left|\sqrt{C_{11}}-\sqrt{C_{22}}\right| &\text{ for some } n \in \mathbb{N}-\{0\}.
                        \end{align} Furthermore, the modulation $\eta_i$ with $i = 1,2$ has to satisfy $\eta^{(n)}_1 \not= \eta^{(n)}_2$ or $\eta^{(-n)}_1 \not= \eta^{(-n)}_2$.
                \item The constant order Floquet exponents take the form $(F_0)_{11} = (F_0)_{44}$ and $(F_0)_{22} = (F_0)_{33}$ 
                    and the modulation frequency $\Omega$ is given by \begin{align}
                        \Omega = \frac{1}{n}\sqrt{\frac{\delta}{|D_1|}}\left(\sqrt{C_{11}}+\sqrt{C_{22}}\right) &\text{ for some } n \in \mathbb{N}-\{0\}.
                        \end{align} Furthermore, the modulation $\eta_i$ with $i = 1,2$ has to satisfy $\eta^{(n)}_1 \not= \eta^{(n)}_2$ or $\eta^{(-n)}_1 \not= \eta^{(-n)}_2$.
            \end{enumerate}
        \end{enumerate}
    \end{theorem}

    \begin{proof} Assume first that the system is posed at a first order exceptional point. Then, using Lemma \ref{lem:rho_exceptional_pt_properties}, it follows that the capacitance matrix is diagonal and that $C_{11}\not=C_{22}$. By Lemma \ref{lem:except_pt_rho-setting_F0_cases}, it follows that the constant order Floquet exponent matrix $F_0$ is either of the form $(F_0)_{11}=(F_0)_{33}$  and $(F_0)_{22}=(F_0)_{44}$ or of the form $(F_0)_{11}=(F_0)_{44}$  and $(F_0)_{22}=(F_0)_{33}$. In the first case, that is, in the case where $(F_0)_{11}=(F_0)_{33}$  and $(F_0)_{22}=(F_0)_{44}$, it holds that
        \begin{align}
            \abs{(A_0)_{11}-(A_0)_{33}} = \abs{(A_0)_{22}-(A_0)_{44}} = n \Omega \text{ for some } n \in \mathbb{N}
        \end{align}
        and thus $\Omega$ is given by
            $$\Omega = \frac{1}{n}\sqrt{\frac{\delta}{|D_1|}}\left|\sqrt{C_{11}}-\sqrt{C_{22}}\right|.$$
        Depending on whether $C_{11}> C_{22}$ or $C_{11} < C_{22}$, the first order coefficient matrix $A_1$ is given by 
        \begin{align*}
            A_1 = \begin{cases}
                \ds \frac{\i\sqrt{\delta} (\eta_1-\eta_2)}{2 \sqrt{|D_1|} }\begin{pmatrix}
    0 & 0 & 0 & 0 \\
    0 & 0 & 0 & 0 \\
    -\frac{C_{11}-C_{22}}{ \sqrt{ C_{11}}} & \frac{C_{11}-C_{22}}{ \sqrt{ C_{11}}} & 0 & 0 \\
    -\frac{C_{11}-C_{22}}{ \sqrt{ C_{11}}} & \frac{C_{11}-C_{22}}{ \sqrt{ C_{11}}} & 0 & 0 
    \end{pmatrix} &\text{ if } C_{11}> C_{22},\\
    \\
   \ds  \frac{\i\sqrt{\delta} (\eta_1-\eta_2)}{2 \sqrt{|D_1|}}\begin{pmatrix}
        0 & 0 & -\frac{C_{22}-C_{11}}{ \sqrt{C_{11}}} & \frac{C_{22}-C_{11}}{ \sqrt{C_{11}}} \\
        \nm 
        0 & 0 &  -\frac{C_{22}-C_{11}}{ \sqrt{C_{11}}} & \frac{C_{22}-C_{11}}{ \sqrt{C_{11}}} \\
        \nm
        0 & 0 & 0 & 0 \\
        0 & 0 & 0 & 0 
        \end{pmatrix} &\text{ if } C_{11}< C_{22}.
            \end{cases}
        \end{align*}
        By Theorem \ref{thm:criterion_for_first_order_exceptional_point}, one needs to compute $(A_1^{(n_i -n_j)})_{ij}$, where $i$ and $j$ are the coordinates of the corresponding double Floquet exponent and $\pm(n_i- n_j)$ the corresponding resonant frequencies. When $C_{11}>C_{22}$, then $n_1 - n_3 = -n$ and $n_2 - n_4 = n$. Thus, the corresponding entries of the first order Floquet exponent matrix are given by
        \begin{align*}
            &\begin{pmatrix}
                (F_0)_{11}&0\\
                0 &(F_0)_{33}
            \end{pmatrix}
                \,-\,\epsilon\frac{\i\sqrt{\delta} (\eta^{(n)}_1-\eta^{(n)}_2)}{2 \sqrt{|D_1|} }\frac{C_{11}-C_{22}}{ \sqrt{ C_{11}}}\begin{pmatrix}
                0 &0\\
                1 &0\\
            \end{pmatrix},
            \end{align*}
            and 
            \begin{align*}
            &\begin{pmatrix}
                (F_0)_{22}&0\\
                0 &(F_0)_{44}
            \end{pmatrix}
                \,+\,\epsilon\frac{\i\sqrt{\delta} (\eta^{(-n)}_1-\eta^{(-n)}_2)}{2 \sqrt{|D_1|} }\frac{C_{11}-C_{22}}{ \sqrt{ C_{11}}}\begin{pmatrix}
                0 &0\\
                1 &0\\
            \end{pmatrix}.
        \end{align*}
        Hence, in order to have a first order asymptotic exceptional point, one needs either $\eta^{(n)}_1 \not = \eta^{(n)}_2$ or $\eta^{(-n)}_1 \not = \eta^{(-n)}_2$. The other cases are handled by the same procedure.
        On the contrary, suppose that $\Omega,\,\eta_1$ and $\eta_2$ satisfy one of the above cases, then Theorem \ref{thm:criterion_for_first_order_exceptional_point} states that the system is posed at a first order asymptotic exceptional point.
    \end{proof}

    Analyzing further the two cases in Theorem \ref{thm:classification_1st_ord_exc_pt_rho-setting} for first order asymptotic exceptional points in the case where only $\rho$ is modulated and taking into account the formula for the corresponding first order coefficient matrix $A_1$ given by  \eqref{eq:A1_rho_modulation}, one obtains the following asymptotics. Assuming first that $C_{11} > C_{22}$, then $\alpha = C_{11}-C_{22}$. In the case where $F_0$ takes the form $(F_0)_{11} = (F_0)_{33}$ and $(F_0)_{22} = (F_0)_{44}$ and in the case where $F_0$ takes the form $(F_0)_{11} = (F_0)_{44}$ and $(F_0)_{22} = (F_0)_{33}$, one has
    \begin{align*}
        &\begin{pmatrix}
            (F_0)_{11} &0\\
            0 &(F_0)_{11}
        \end{pmatrix} -
        \epsilon\frac{\i\sqrt{\delta} (\eta_1^{(n)}-\eta_2^{(n)})}{2 \sqrt{ |D_1|} }\frac{C_{11}-C_{22}}{ \sqrt{ C_{11}}}\begin{pmatrix}
            0 & 0\\
            1 &0
        \end{pmatrix},\\
        &\begin{pmatrix}
            (F_0)_{22} &0\\
            0 &(F_0)_{22}
        \end{pmatrix} +
        \epsilon\frac{\i\sqrt{\delta} (\eta_1^{(-n)}-\eta_2^{(-n)})}{2 \sqrt{ |D_1|} }\frac{C_{11}-C_{22}}{ \sqrt{ C_{11}}}\begin{pmatrix}
            0 & 0\\
            1 &0
        \end{pmatrix}.
    \end{align*}
    When $C_{22} > C_{11}$, then $\alpha = C_{22}-C_{11}$ and the first order asymptotic exceptional points take the form
    \begin{align*}
        &\begin{pmatrix}
            (F_0)_{11} &0\\
            0 &(F_0)_{33}
        \end{pmatrix} -
        \epsilon\frac{\i\sqrt{\delta} (\eta_1^{(-n)}-\eta_2^{(-n)})}{2 \sqrt{ |D_1|}}\frac{C_{11}-C_{22}}{ \sqrt{ C_{22}}}\begin{pmatrix}
            0 &1 \\
            0 &0
        \end{pmatrix},\\
        &\begin{pmatrix}
            (F_0)_{22} &0\\
            0 &(F_0)_{44}
        \end{pmatrix} +
        \epsilon\frac{\i\sqrt{\delta} (\eta_1^{(n)}-\eta_2^{(n)})}{2 \sqrt{ |D_1|} }\frac{C_{11}-C_{22}}{ \sqrt{ C_{22}}}\begin{pmatrix}
            0 &1 \\
            0 &0
        \end{pmatrix}
    \end{align*}
    and
    \begin{align*}
        &\begin{pmatrix}
            (F_0)_{11} &0\\
            0 &(F_0)_{44}
        \end{pmatrix} +
        \epsilon\frac{\i\sqrt{\delta} (\eta_1^{(-n)}-\eta_2^{(-n)})}{2 \sqrt{ |D_1|} }\frac{C_{11}-C_{22}}{ \sqrt{ C_{22}}}\begin{pmatrix}
            0 &1 \\
            0 &0
        \end{pmatrix},\\
        &\begin{pmatrix}
            (F_0)_{22} &0\\
            0 &(F_0)_{33}
        \end{pmatrix} -
        \epsilon\frac{\i\sqrt{\delta} (\eta_1^{(n)}-\eta_2^{(n)})}{2 \sqrt{ |D_1|} }\frac{C_{11}-C_{22}}{ \sqrt{ C_{22}}}\begin{pmatrix}
            0 &1 \\
            0 &0
        \end{pmatrix},
    \end{align*}
    when $F_0$ takes the form $(F_0)_{11} = (F_0)_{33}$ and $(F_0)_{22} = (F_0)_{44}$ or when $F_0$ takes the form $(F_0)_{11} = (F_0)_{44}$ and $(F_0)_{22} = (F_0)_{33}$, respectively.

\subsubsection{Exceptional points when only the bulk moduli are modulated}
   
The case where only the bulk moduli  are modulated is similar to the case where only the densities are modulated. The following lemma holds.  

    \begin{lemma}\label{lem:except_pt_kappa-setting_F0_cases}
        In the setting of Theorem \ref{thm:EriksMasterEquation} and Lemma \ref{lem:double_constant_order_Floquet_exponent}, suppose that only $1/\kappa_i$ with $i=1,2$ is modulated and that $1/\rho_i$ for $i=1,2$ is not modulated. If ${dx}/{dt} = A_\epsilon(t) x$ is posed at an exceptional point, then the constant order Floquet exponent matrix $F_0$ is either of the form $(F_0)_{11}=(F_0)_{44}$  and $(F_0)_{22}=(F_0)_{33}$ or of the form $(F_0)_{11}=(F_0)_{33}$  and $(F_0)_{22}=(F_0)_{44}$.
    \end{lemma}
    
    \begin{proof}
        As in the proof of Lemma \ref{lem:except_pt_rho-setting_F0_cases}, the only possible cases for the eigenvalues of $F_0$ are either $(F_0)_{11}=(F_0)_{22}$ or $(F_0)_{33}=(F_0)_{44}$ or the cases stated above. However, the structure of the first order coefficient matrix $A_1 + B_1$ in the case where only $1/\kappa_i$ for $i =1,2$ is modulated yields that its off-diagonal entries can never be equal to zero independently. Thus, in the case of an exceptional point, $F_0$ has to be of the form  $(F_0)_{11}=(F_0)_{44}$  and $(F_0)_{22}=(F_0)_{33}$ or of the form $(F_0)_{11}=(F_0)_{33}$  and $(F_0)_{22}=(F_0)_{44}$.
    \end{proof}

    \begin{lemma}\label{lem:kappa-modulation_exceptional_pt_C12}
        In the setting of Theorem \ref{thm:EriksMasterEquation} and Lemma \ref{lem:double_constant_order_Floquet_exponent}, supposing that only $1/\kappa_i$ with $i=1,2$ is modulated and that $1/\rho_i$ for $i=1,2$ is not modulated, then it holds that if ${d\Phi}/{dt} = A_\epsilon(t) \Phi$ is posed at a first order asymptotic exceptional point, then $C_{12}=0$.        
    \end{lemma}
    Thus the only possibility to create first order asymptotic exceptional points in the setting where $C_{12}\not=0$ and under the assumptions of Lemma \ref{lem:double_constant_order_Floquet_exponent}, that is, under the assumption
    $$\frac{\sqrt{C_{11} + C_{22} + \alpha}-\sqrt{C_{11} + C_{22} - \alpha}}{\sqrt{C_{11} + C_{22} + \alpha}+\sqrt{C_{11} + C_{22} - \alpha}} \not\in \mathbb{Q},$$
     is to modulate both the bulk modulus $\kappa$ \emph{and} the density $\rho$ simultaneously.

    \begin{proof}
        Suppose by contradiction that $C_{12}\not=0$. Then, in the case of a first order asymptotic exceptional point with resonant frequency $n = n_1 -n_2$, $n = n_1 - n_3$, $n = n_2 -n_3$ or $n = n_2 -n_4$, one needs to have either $(A_{12}+B_{12})^{(n)}= 0$, that is,
        \begin{align*}
         \ds   \frac{\i\sqrt{\delta}}{4\sqrt{2|D_1|}} \frac{\alpha-(C_{11}-C_{22})}{\alpha \sqrt{C_{11}+C_{22}-\alpha}}
        (C_{11}+C_{22})\left(\gamma^{(n)}_1-\gamma^{(n)}_2\right) \\ \nm \qquad \ds = \i \sqrt{\frac{|D_1|}{2}} \frac{1}{4\sqrt{\delta}} \frac{C_{11}-C_{22}-\alpha}{\alpha\sqrt{C_{11}+C_{22}-\alpha}}\left(\Bigl(\frac{d^2\gamma_1}{dt^2}\Bigr)^{(n)}-\Bigl(\frac{d^2\gamma_2}{dt^2}\Bigr)^{(n)}\right),
        \end{align*}
        or $(A_{21}+B_{21})^{(-n)}= 0$, that is,
        \begin{align*}
           \ds \frac{\i\sqrt{\delta}}{ 4 \sqrt{2 |D_1| } }\frac{\alpha+C_{11}-C_{22}}{ \alpha \sqrt{\alpha+C_{11}+C_{22}}}
        (C_{11}+C_{22})\left(\gamma^{(-n)}_1-\gamma^{(-n)}_2\right) \\ 
        \nm \qquad \ds = -\i\sqrt{\frac{|D_1|}{2}} \frac{1}{4\sqrt{\delta}} \frac{C_{11}-C_{22}+\alpha}{\alpha\sqrt{C_{11}+C_{22}+\alpha}}\left(\Bigl(\frac{d^2\gamma_1}{dt^2}\Bigr)^{(-n)}-\Bigl(\frac{d^2\gamma_2}{dt^2}\Bigr)^{(-n)}\right).
        \end{align*}
        In the case where $C_{12}\not=0$,  the two equations reduce to either
        \begin{align*}
            -\frac{\delta}{|D_1|}
        (C_{11}+C_{22})\left(\gamma^{(n)}_1-\gamma^{(n)}_2\right) = \Bigl(\frac{d^2\gamma_1}{dt^2}\Bigr)^{(n)}-\Bigl(\frac{d^2\gamma_2}{dt^2}\Bigr)^{(n)}
        \end{align*}
        or
        \begin{align*}
            -\frac{\delta}{|D_1|}
        (C_{11}+C_{22})\left(\gamma^{(-n)}_1-\gamma^{(-n)}_2\right) = \Bigl(\frac{d^2\gamma_1}{dt^2}\Bigr)^{(-n)}-\Bigl(\frac{d^2\gamma_2}{dt^2}\Bigr)^{(-n)},
        \end{align*}
        respectively. However, in the case of a real-valued perturbation $\gamma_i$ with $i = 1,2$, the Fourier series coefficients $\gamma_i^{(n)}$ and $\gamma_i^{(-n)}$ need to be conjugate to each other and thus the above equations cannot be fulfilled independently. This is a contradiction and hence we conclude that $C_{12}=0$.
    \end{proof}

    \begin{lemma}
        In the setting of Theorem \ref{thm:EriksMasterEquation}, supposing that the number of resonators is equal to two and that only $1/\kappa_i$ with $i=1,2$ is modulated and $1/\rho_i$ for $i=1,2$ is not modulated, then it holds that if $C_{12}=0$ and $C_{11} \not= C_{22}$, the off-diagonal blocks $A_{12} + B_{12}$ and $A_{21} + B_{21}$ of the first order coefficient matrix of the system ${d\Phi}/{dt} = A_\epsilon(t) \Phi$ are given by 
        \begin{align*}
            A_{12} + B_{12} &= \begin{cases}
                \begin{pmatrix}
                    0 &0\\
                    0 &0
                \end{pmatrix} \text{ if } C_{11} > C_{22},\\
                \\ \ds
                \frac{\i}{4}\left[\frac{\sqrt{\delta}}{\sqrt{|D_1|}} \frac{C_{11}+C_{22}}{\sqrt{C_{11}}}
                (\gamma_1-\gamma_2) + \frac{\sqrt{|D_1|}}{\sqrt{\delta}} \frac{1}{\sqrt{C_{11}}}\left(\frac{d^2\gamma_1}{dt^2}-\frac{d^2\gamma_2}{dt^2}\right)\right]\begin{pmatrix} 
                    -1 &1\\
                    -1 &1
                    \end{pmatrix} \\ 
                    \qquad \text{ if } C_{22} > C_{11},
                \end{cases}\\
                \\
            A_{21} + B_{21} &= \begin{cases} \ds
                \frac{\i}{4}\left[\frac{\sqrt{\delta}}{\sqrt{|D_1|}}\frac{C_{11}+C_{22}}{\sqrt{C_{11}}}(\gamma_1-\gamma_2) + \frac{\sqrt{|D_1|}}{\sqrt{\delta}}\frac{1}{\sqrt{C_{11}}}\left(\frac{d^2\gamma_1}{dt^2}-\frac{d^2\gamma_2}{dt^2}\right)\right]\begin{pmatrix} 
                    -1 & 1\\
                    -1 &1
                    \end{pmatrix} \\
                    \qquad \text{ if } C_{11} > C_{22},\\
                    \\
                    \begin{pmatrix}
                        0 &0\\
                        0 &0
                    \end{pmatrix} \text{ if } C_{22} > C_{11}.
                \end{cases}
        \end{align*}
    \end{lemma}

    \begin{proof}
        This is obtained from equations \eqref{eq:A1_kappa_modulation} and \eqref{eq:B1_kappa_modulation}.
    \end{proof}

    \begin{theorem}[Classification of first order asymptotic exceptional points when only $\kappa$ is modulated]\label{thm:classification_1st_ord_exc_pt_kappa-setting}
        In the setting of Theorem \ref{thm:EriksMasterEquation} and Lemma \ref{lem:double_constant_order_Floquet_exponent}, suppose that only $1/\kappa_i$ with $i=1,2$ is modulated with no constant component and that $1/\rho_i$ for $i=1,2$ is not modulated. Then, the following equivalent statements hold:
        \begin{enumerate}[(i)]
            \item The system ${dx}/{dt} = A_\epsilon(t) x$ is posed at a first order asymptotic exceptional point.
            \item The capacitance matrix $C$ takes the form
            \begin{align*}
                C = \begin{pmatrix}
                    C_{11} &0\\
                    0 &C_{22}
                \end{pmatrix}
            \end{align*} with $C_{11} \not= C_{22}$ and one of the following cases applies:
            \begin{enumerate}[(a)]
                \item The constant order Floquet exponents take the form $(F_0)_{11} = (F_0)_{33}$ and $(F_0)_{22} = (F_0)_{44}$  
                    and the modulation frequency $\Omega$ is given by \begin{align}
                        \Omega = \frac{1}{n}\sqrt{\frac{\delta}{|D_1|}}\left|\sqrt{C_{11}}-\sqrt{C_{22}}\right| &\text{ for some } n \in \mathbb{N}-\{0\}.
                        \end{align} Furthermore, the modulation $\eta_i$ with $i = 1,2$ has to satisfy either $\eta^{(n)}_1 \not= \eta^{(n)}_2$ or $\eta^{(-n)}_1 \not= \eta^{(-n)}_2$ and the resonant frequency $n$ is such that $n \not= \sqrt{\frac{\delta}{|D_1|}
                        (C_{11}+C_{22})}$.
                \item The constant order Floquet exponents take the form $(F_0)_{11} = (F_0)_{44}$ and $(F_0)_{22} = (F_0)_{33}$ 
                    and the modulation frequency $\Omega$ is given by \begin{align}
                        \Omega = \frac{1}{n}\sqrt{\frac{\delta}{|D_1|}}(\sqrt{C_{11}}+\sqrt{C_{22}}) &\text{ for some } n \in \mathbb{N}-\{0\}.
                        \end{align} Furthermore, the modulation $\eta_i$ with $i = 1,2$ has to satisfy either $\eta^{(n)}_1 \not= \eta^{(n)}_2$ or $\eta^{(-n)}_1 \not= \eta^{(-n)}_2$ and the resonant frequency $n$ is such that $n \not= \sqrt{\frac{\delta}{|D_1|}
                        (C_{11}+C_{22})}$.
            \end{enumerate}
        \end{enumerate}
    \end{theorem}

    \begin{proof}
        The proof is analogous to the proof of Theorem \ref{thm:classification_1st_ord_exc_pt_rho-setting}.
    \end{proof}

    \subsubsection{Exceptional points when both the densities and the bulk moduli are modulated}
    In this section, we seek to classify first order asymptotic exceptional points in the case where both the densities and the bulk moduli are modulated. This case is fundamentally different to the case where either the density or the bulk modulus is modulated. In fact, the capacitance matrix needs no longer to be off-diagonal, that is, the condition $C_{12}\not = 0$ is no longer necessary. Furthermore, the possibility of a first order asymptotic exceptional point depends mainly on the proportion of the modulation $\gamma$ of the bulk moduli to the modulation $\eta$ of the densities. Indeed, in the case of a first order asymptotic exceptional point, their ratio is determined by the capacitance matrix entries, that is, by the geometry of the corresponding dimer of subwavelength resonators. This will become clearer in Theorem \ref{thm:classificiation_EP_rho_kappa}. The following lemma will be of use in the proof of Theorem \ref{thm:classificiation_EP_rho_kappa}. 

    \begin{lemma}\label{lem:except_pt_kappa-rho-setting_F0_cases}
        In the setting of Theorem \ref{thm:EriksMasterEquation} and Lemma \ref{lem:double_constant_order_Floquet_exponent}, if the linear system of ODEs ${dx}/{dt} = A_\epsilon(t) x$ is posed at a first order asymptotic exceptional point, then the constant order Floquet exponent matrix $F_0$ is either of the form $(F_0)_{11}=(F_0)_{44}$  and $(F_0)_{22}=(F_0)_{33}$ or of the form $(F_0)_{11}=(F_0)_{33}$  and $(F_0)_{22}=(F_0)_{44}$.
    \end{lemma}

    \begin{proof}
        The proof immediately follows from the proof of Lemma \ref{lem:except_pt_kappa-setting_F0_cases}, since the diagonal blocks of the first order coefficient matrix $A_1$ coincide with the diagonal blocks of the first order coefficient matrix $A_1$ of the setting of Lemma \ref{lem:except_pt_kappa-setting_F0_cases}. Thus, the proof of Lemma \ref{lem:except_pt_kappa-setting_F0_cases} also applies in the setting where $\rho$ and $\kappa$ are modulated and hence, Lemma \ref{lem:except_pt_kappa-rho-setting_F0_cases} follows.
    \end{proof}

    \begin{theorem}[Classification of exceptional points when $\rho$ and $\kappa$ 
    are modulated]\label{thm:classificiation_EP_rho_kappa}
        In the setting of Theorem \ref{thm:EriksMasterEquation} and Lemma \ref{lem:double_constant_order_Floquet_exponent}, using the notation of Section \ref{sec:2ResCase} and supposing that $C_{12} \not= 0$, then the following statements are equivalent:
        \begin{enumerate}[(i)]
            \item The system ${dx}/{dt} = A_\epsilon(t) x$ is posed at a first order asymptotic exceptional point.
            \item One of the following cases is satisfied:
                \begin{enumerate}[(a)]
                    \item The constant order Floquet exponent matrix $F_0$ is of the form $(F_0)_{11}=(F_0)_{33}$  and $(F_0)_{22}=(F_0)_{44}$, that is, $\Omega$ is given by 
                        \begin{align*}
                            \Omega = \frac{1}{n}\frac{\sqrt{\delta}}{\sqrt{2|D_1|}}\left(\sqrt{C_{11} + C_{22} + \alpha} - \sqrt{C_{11} + C_{22} - \alpha}\right) \text{ with } n \in \mathbb{N}
                        \end{align*}
                        and the modulation of $\kappa$ and $\rho$ satisfy either
                        \begin{align*}
                            -\frac{C_{11}+C_{22}-4\sqrt{C_{11}C_{22}-\abs{C_{12}}^2}}{2\sqrt{(C_{11}-C_{22})^2+4\abs{C_{12}}}}&(\gamma_1-\gamma_2)^{(n)} = (\eta_1-\eta_2)^{(n)}\\
                            \text{ or }\\
                            \frac{C_{11}+C_{22}-4\sqrt{C_{11}C_{22}-\abs{C_{12}}^2}}{2\sqrt{(C_{11}-C_{22})^2+4\abs{C_{12}}}}&(\gamma_1-\gamma_2)^{(n)} =(\eta_1-\eta_2)^{(n)}.
                        \end{align*}
                    \item The constant order Floquet exponent matrix $F_0$ is of the form $(F_0)_{11}=(F_0)_{44}$  and $(F_0)_{22}=(F_0)_{33}$, that is, $\Omega$ is given by 
                    \begin{align*}
                        \Omega = \frac{1}{n}\frac{\sqrt{\delta}}{\sqrt{2|D_1|}}\left(\sqrt{C_{11} + C_{22} + \alpha} + \sqrt{C_{11} + C_{22} - \alpha}\right) \text{ with } n \in \mathbb{N}
                    \end{align*}
                        and the modulation of $\kappa$ and $\rho$ satisfy either
                        \begin{align*}
                            -\frac{C_{11}+C_{22}+4\sqrt{C_{11}C_{22}-\abs{C_{12}}^2}}{2\sqrt{(C_{11}-C_{22})^2+4\abs{C_{12}}}}&(\gamma_1-\gamma_2)^{(n)} = (\eta_1-\eta_2)^{(n)}\\
                            \text{ or }\\
                            \frac{C_{11}+C_{22}+4\sqrt{C_{11}C_{22}-\abs{C_{12}}^2}}{2\sqrt{(C_{11}-C_{22})^2+4\abs{C_{12}}}}&(\gamma_1-\gamma_2)^{(n)} =(\eta_1-\eta_2)^{(n)}.
                        \end{align*}
                \end{enumerate}
        \end{enumerate}
    \end{theorem}

    \begin{proof}
        When the system is posed at a first order asymptotic exceptional point and since the modulation of $\kappa$ has no constant component, one of the off-diagonal blocks of the first order coefficient matrix $A_1$ needs to be equal to zero. This precisely corresponds to the following equations:
        \begin{align*}
            \frac{1}{2\alpha}\bigg(\frac{\sqrt{\delta}}{\sqrt{|D_1|}}
    (C_{11}+C_{22})(\gamma_1-\gamma_2)^{(n)}&+\frac{\sqrt{|D_1|}}{\sqrt{\delta}}\Bigl(\frac{d^2\gamma_1}{dt^2}-\frac{d^2\gamma_2}{dt^2}\Bigr)^{(n)} \bigg) = \frac{\sqrt{\delta}}{\sqrt{|D_1|}}(\eta_1-\eta_2)^{(n)}\\
    &\text{or }\\
    \frac{1}{2\alpha}\bigg(\frac{\sqrt{\delta}}{\sqrt{|D_1|}}
    (C_{11}+C_{22})(\gamma_1-\gamma_2)^{(-n)}&+\frac{\sqrt{|D_1|}}{\sqrt{\delta}}\Bigl(\frac{d^2\gamma_1}{dt^2}-\frac{d^2\gamma_2}{dt^2}\Bigr)^{(-n)}\bigg) = -\frac{\sqrt{\delta}}{\sqrt{|D_1|}}(\eta_1-\eta_2)^{(-n)},
        \end{align*}
        where the first equation corresponds to the upper diagonal block and the lower equation corresponds to the lower diagonal block. Denoting $(\gamma_1-\gamma_2)^{(n)}$ by $g$ and $(\eta_1-\eta_2)^{(n)}$ by $h$, one obtains the following equations:
        \begin{align*}
            \frac{1}{2\alpha}\bigg(\frac{\sqrt{\delta}}{\sqrt{|D_1|}}
    (C_{11}+C_{22})&-\frac{\sqrt{|D_1|}}{\sqrt{\delta}}n^2\Omega^2\bigg)g = \frac{\sqrt{\delta}}{\sqrt{|D_1|}}h\\
    &\text{or }\\
    \frac{1}{2\alpha}\bigg(\frac{\sqrt{\delta}}{\sqrt{|D_1|}}
    (C_{11}+C_{22})&-\frac{\sqrt{|D_1|}}{\sqrt{\delta}}n^2\Omega^2\bigg)\bar{g} = -\frac{\sqrt{\delta}}{\sqrt{|D_1|}}\bar{h},
        \end{align*}
        where $\bar{g}$ and $\bar{h}$ denote the complex conjugates of $g$ and $h$, respectively.
        Lemma \ref{lem:except_pt_kappa-rho-setting_F0_cases} implies that $\Omega$ needs to be of the form
        \begin{align*}
            \Omega = \frac{1}{n}\frac{\sqrt{\delta}}{\sqrt{2|D_1|}}\begin{cases}
                \sqrt{C_{11} + C_{22} + \alpha} &-\hspace{0.2cm} \sqrt{C_{11} + C_{22} - \alpha} \\ &\text{or}\\
                \sqrt{C_{11} + C_{22} + \alpha} &+\hspace{0.2cm} \sqrt{C_{11} + C_{22} - \alpha}.
            \end{cases}
        \end{align*}
        The first case corresponds to the setting where $(F_0)_{11}=(F_0)_{33}$  and $(F_0)_{22}=(F_0)_{44}$, while the second case corresponds to the setting where $(F_0)_{11}=(F_0)_{44}$  and $(F_0)_{22}=(F_0)_{33}$. Thus, one obtains for the equations corresponding to the upper block
        \begin{align*}
            \frac{1}{2\alpha}\Big(
    (C_{11}+C_{22})&-(\sqrt{C_{11} + C_{22} + \alpha} - \sqrt{C_{11} + C_{22} - \alpha})^2\Big)g = h \\
            \text{ and }\\
            \frac{1}{2\alpha}\Big(
            (C_{11}+C_{22})&-(\sqrt{C_{11} + C_{22} + \alpha} + \sqrt{C_{11} + C_{22} - \alpha})^2\Big)g = h
        \end{align*}
        and for the equations corresponding to the lower block, one correspondingly obtains
        \begin{align*}
            \frac{1}{2\alpha}\Big(
            (C_{11}+C_{22})&-(\sqrt{C_{11} + C_{22} + \alpha} - \sqrt{C_{11} + C_{22} - \alpha})^2\Big)\bar{g} = -\bar{h} \\
            \text{ and }\\
            \frac{1}{2\alpha}\Big(
            (C_{11}+C_{22})&-(\sqrt{C_{11} + C_{22} + \alpha} + \sqrt{C_{11} + C_{22} - \alpha})^2\Big)\bar{g} = -\bar{h}.
        \end{align*}
        Reducing the equations further, one obtains that in order to have an asymptotic exceptional point, the following two cases are possible.
        In the first case, when $F_0$ has the form $(F_0)_{11}=(F_0)_{33}$  and $(F_0)_{22}=(F_0)_{44}$, one needs either
        \begin{align*}
            -\frac{C_{11}+C_{22}-4\sqrt{C_{11}C_{22}-\abs{C_{12}}^2}}{2\sqrt{(C_{11}-C_{22})^2+4\abs{C_{12}}}}g &= h &\text{ for a zero upper block }\\
            \text{ or }\\
            \frac{C_{11}+C_{22}-4\sqrt{C_{11}C_{22}-\abs{C_{12}}^2}}{2\sqrt{(C_{11}-C_{22})^2+4\abs{C_{12}}}}g &=h &\text{ for a zero lower block.}
        \end{align*}
        In the second case, when $F_0$ has the form $(F_0)_{11}=(F_0)_{44}$  and $(F_0)_{22}=(F_0)_{33}$, one needs either
        \begin{align*}
            -\frac{C_{11}+C_{22}+4\sqrt{C_{11}C_{22}-\abs{C_{12}}^2}}{2\sqrt{(C_{11}-C_{22})^2+4\abs{C_{12}}}}g &= h &\text{ for a zero upper block }\\
            \text{ or }\\
            \frac{C_{11}+C_{22}+4\sqrt{C_{11}C_{22}-\abs{C_{12}}^2}}{2\sqrt{(C_{11}-C_{22})^2+4\abs{C_{12}}}}g &=h &\text{ for a zero lower block.}
        \end{align*}
        This proves one direction of the theorem. The other direction is an application of Theorem \ref{thm:criterion_for_first_order_exceptional_point}.
    \end{proof}

%% file: chapters/Concluding_Remarks.tex
\section{Concluding remarks} \label{concluding_remarks}

This paper presented a new analysis to obtain an asymptotic approximation of  Floquet's exponent matrix and Floquet exponents associated to an ODE of the form $$\frac{dx}{dt} = A_\epsilon(t)x,$$ where 
$A_\epsilon(t) =  A_0 + \epsilon A_1(t) + \epsilon^2A_2(t) + \ldots$
is an analytic function of $\epsilon$ in a neighborhood of $0$ with $A_0$ being diagonal and constant in time, and $A_n(t)$ having a finite Fourier series for all $n \geq 1$. This analysis for an asymptotic approximation of  Floquet's exponent matrix was presented in Theorem \ref{Inductive_Identity_for_Lyapunov-Floquet_decomposition} and provides a method to asymptotically approximate the fundamental solution of an ODE of the above form (Section \ref{sec:asymp_exp_of_fund_sol}). Furthermore, it laid the groundwork for an asymptotic analysis of Floquet exponents and exceptional points. It was proven that, generically, Floquet exponents are perturbed quadratically when the first order coefficient matrix $A_1$ has no constant component in its Fourier series (Lemma \ref{lem:single_Floquet_exponents_are_perturbed_quadratically}). Linear perturbation is only possible when $A_1$ has a constant component or at a double constant order Floquet exponent. However, a double constant order Floquet exponent is not sufficient for a linear perturbation, but the system needs to be perturbed with the so-called resonant frequency of the double constant order Floquet exponent (Lemma \ref{lemma:linear_perturbation_of_double_eigenvalue}). Thus, in order to obtain linear perturbation of constant order Floquet exponents, a double constant order Floquet exponent needs to be created via folding of the system through certain choices of the modulation period $T$ and then via the modulation in the resonant frequency. The precise procedure was presented in Section \ref{F1_when_F0_has_double_eigenvalue}. There, the procedure was also illustrated with the example of the classical harmonic oscillator, which was introduced in Section \ref{Harmonic_oscillator}.

The study of exceptional points in the case of Floquet metamaterials was the main motivation for the study of asymptotic Floquet theory. 
Indeed, the derived asymptotic expansions allowed for a characterization of first order asymptotic exceptional points.
 This characterization lead to a deep analysis of the occurrence of asymptotic exceptional points in the setting of the classical harmonic oscillator and in the setting of  a dimer of time-modulated subwavelength resonators. 
 In the case of time-modulated subwavelength resonators, two fundamentally different behaviors were observed.
 When either the density $\rho$ of the resonators or the bulk modulus $\kappa$ of the resonators are perturbed,  first order exceptional points only occur if the associated capacitance matrix of the dimer is diagonal, see respectively Theorem \ref{thm:classification_1st_ord_exc_pt_rho-setting} and Theorem \ref{thm:classification_1st_ord_exc_pt_kappa-setting}.
 However, in the case where the density $\rho$ of the resonators \emph{and} the bulk modulus $\kappa$ are perturbed simultaneously, this restriction is no longer present.
 Indeed, Theorem \ref{thm:classificiation_EP_rho_kappa} states that exceptional points are obtained if and only if certain frequency components of the perturbations of $\rho$ and of $\kappa$ fulfill a certain ratio, which is determined by the capacitance matrix, that is, by the geometry of the system.  The asymptotic exceptional points observed here have a rather different origin compared to the ones observed in \cite{ammari2020time}, where the exceptional points emerge in a neighborhood of a degenerate point rather than the degenerate point itself.
 
From these results mentioned above, it follows that time-modulation of the resonators provides an uncountable number of first order asymptotic exceptional points for any dimer of subwavelength resonators.
This result can be considered as a first step in the classification of first order asymptotic exceptional points for Floquet metamaterials. The formulas for the Floquet's exponent matrix provide a tool to also analyze the occurrence of higher order asymptotic exceptional points. Furthermore, the procedures used in this paper lay a groundwork for the classification of first order exceptional points in the setting of  an \emph{arbitrary} finite number of time-modulated subwavelength resonators and are key for studying topological properties of Floquet metamaterials.